\documentclass[a4paper,11pt,oneside,reqno]{amsart}

\usepackage{hyperref}
\usepackage[headinclude,DIV13]{typearea}
\areaset{15.1cm}{25.0cm}
\usepackage[utf8]{inputenc}
\usepackage[T1]{fontenc}
\usepackage{amsthm,amsmath,txfonts,amssymb,bbm,hyperref,mathtools,csquotes,thmtools,verbatim,xfrac,bigints}
\usepackage{chngcntr}
\usepackage{textcomp}
\usepackage{tikz}
\usepackage{pgfplots}
\usepackage{multicol}

\usepackage[UKenglish]{babel}

\theoremstyle{plain}
\newtheorem{thm}{Theorem}[section]
\newtheorem{cor}[thm]{Corollary}
\newtheorem{prop}[thm]{Proposition}
\newtheorem{lemma}[thm]{Lemma}

\newtheorem{assumption}{Assumption}
\theoremstyle{definition}
\newtheorem{definition}[thm]{Definition}
\theoremstyle{remark}
\newtheorem{rem}[thm]{Remark}

\def\paragraph#1{\noindent \textbf{#1}}

\numberwithin{equation}{section}

\usepackage{graphicx}
\usepackage{xcolor}
\usepackage{color}
\usepackage{tikz}
\usepackage{pgfplots}
\usetikzlibrary{shapes.misc}
\tikzset{cross/.style={cross out, draw=black, minimum size=2*(#1-\pgflinewidth), inner sep=0pt, outer sep=0pt},
	cross/.default={2pt}}
\def\addlegendimage{\csname pgfplots@addlegendimage\endcsname}

\pgfkeys{/pgfplots/number in legend/.style={%
		/pgfplots/legend image code/.code={%
			\node at (0.295,-0.0225){#1};
		},%
	},
}

\definecolor{darkgreen}{rgb}{0,.6,0}
\definecolor{darkagenta}{rgb}{.5,0,.5}
\definecolor{darkred}{rgb}{1,0,0}
\definecolor{darkblue}{rgb}{0,0,.4}
\definecolor{black}{rgb}{0,0,0}
\definecolor{white}{rgb}{1,1,1}

\begin{document}
	\title[Convergence of the centred maximum of the 2d scale-inhomogeneous DGFF]{Extremes of the 2d scale-inhomogeneous discrete {G}aussian free field: {C}onvergence of the maximum in the regime of weak correlations}
	\author{Maximilian Fels, Lisa Hartung}
	\address{M. Fels\\Institut f\"ur Angewandte Mathematik\\
		Rheinische Friedrich-Wilhelms-Universität\\ Endenicher Allee 60\\ 53115 Bonn, Germany }
	\email{fels@iam.uni-bonn.de, wt.iam.uni-bonn.de/maximilian-fels}
	\address{L. Hartung\\Institut f\"ur Mathematik\\ Johannes Gutenberg-Universität Mainz\\ Staudingerweg 9\\ 55099 Mainz, Germany}
	\email{lhartung@uni-mainz.de, https://sites.google.com/view/lisahartung/home}
	\thanks{M.F. is funded by the Deutsche Forschungsgemeinschaft (DFG, German Research Foundation) - project-id 211504053 - SFB 1060 and Germany’s Excellence Strategy – GZ 2047/1, project-id 390685813 – “Hausdorff Center for Mathematics” at Bonn University.}
	\maketitle
	\begin{abstract}
		We continue the study of the maximum of the scale-inhomogeneous discrete Gaussian free
		field in dimension two  that was  initiated in \cite{paper1} and continued in \cite{paper3}. In this paper, we consider the regime of 
		weak correlations and prove the  convergence in	law of the centred maximum to a randomly shifted Gumbel distribution. In particular, we obtain limiting expressions for the random shift. As in the case of variable speed branching Brownian motion, the shift is of the form $CY$, where $C$ is a constant that depends only on the
		variance at the shortest scales, and $Y$ is a random variable that depends only on the 
		variance at the largest scales.	Moreover, we investigate the geometry of highest local maxima. We show that they 
		occur in clusters of finite size that are separated  by macroscopic distances. The poofs are based on Gaussian comparison with branching random walks and second moment estimates.
	\end{abstract}
	\section{Introduction}
	In recent years, log-correlated (Gaussian) processes have received considerable attention, see e.g. \cite{MR3594368,MR3911893,2016arXiv160600510B,MR3164771,2015arXiv150304588D,fyodorov1,MR3361256}. Some prominent examples that fall into this class are branching Brownian motion (BBM), the branching random walk (BRW), the $2$d discrete Gaussian free field (DGFF), local maxima of the randomised Riemann zeta function on the critical line and cover times of Brownian motion on the torus. One of the reasons why these processes are interesting is that their correlation structure is such that it becomes relevant for the properties of the extremes of the processes.
The $2$d scale-inhomogeneous discrete Gaussian free field first appeared in \cite{MR3354619}, where it served as a tool in order to prove Poisson-Dirichlet statistics of the extreme values of the $2$d DGFF. Moreover, it is the natural analogue model of variable-speed BBM or the time-inhomogeneous BRW in the context of the two-dimensional DGFF. To be more precise, we start with a formal definition of the model studied in this paper and then, present our new results on the maximum value.
	\subsection{The discrete Gaussian free field}
	Let $V_N \coloneqq ([0,N) \cap \mathbbm{Z} )^2$. The interior of $V_N$ is defined as $V_N^o \coloneqq ([1,N-1] \cap \mathbbm{Z})^2$ and the boundary of $V_N$ is denoted by $ \partial{V}_N \coloneqq V_N \setminus V_N^o$. Moreover, for points $u,v \in V_N$ we write $u \sim v$, if and only if $\|u-v \|_2 =1$, where $\| . \|_2$ is the Euclidean norm.
Let $\mathbb{P}_u$ be the law of a SRW $\{W_k\}_{k \in \mathbbm{N}}$ starting at $u \in {\mathbbm{Z}}^2$. The normalised Green kernel is given by
\begin{equation}
G_{V_N}(u,v) \coloneqq \frac{\pi}{2} \mathbbm{E}_u \left[\sum_{i=0}^{\tau_{\partial{V}_N-1}}\mathbbm{1}_{\{W_i =v\}}\right], \text{ for } u,v \in V_N.
\end{equation}
Here, $\tau_{\partial{V}_N}$ is the first hitting time of the boundary $\partial{V}_N$ by $\{W_k\}_{k \in \mathbbm{N}}$.
For $\delta >0$, we set $V^{\delta}_N\coloneqq (\delta N, (1-\delta)N)^2\cap \mathbb{Z}^2$. By \cite[Lemma 2.1]{MR2243875}, we have, for $\delta\in (0,1)$ and $u,v\in V^\delta_N$,
\begin{align}\label{equation:correlation_dgff}
G_{V_N}(u,v)= \log N - \log_{+} \|u-v\|_2 + O(1),
\end{align}
where $\log_{+}(x)=\max\left\{0,\log (x) \right\}$.
\begin{definition}\label{definition:DGFF}
	The $2$d discrete Gaussian free field (DGFF) on $V_N$, $\phi^N \coloneqq \{\phi_v^N \}_{v \in V_N}$, is a centred Gaussian field with covariance matrix $G_{V_N}$ and entries  $G_{V_N}(x,y)=\mathbbm{E}[\phi^N_x \phi^N_y],$ for $x,y \in V_N.$
\end{definition}
From \autoref{definition:DGFF} it follows that $\phi_v^N = 0$ for $v \in \partial{V}_N,$ i.e. we have Dirichlet boundary conditions.
	\subsection{The two-dimensional scale-inhomogeneous discrete Gaussian free field}
	\begin{definition}{(The $2$d scale-inhomogeneous discrete Gaussian free field).}\label{definition:scale_inh_DGFF} \\
	Let $\phi^N= \{\phi^N_v \}_{v\in V_N}$ be a $2$d DGFF on $V_N$.
	For $v=(v_1,v_2)\in V_N$ and $\lambda \in (0,1)$, let
	\begin{align}
	[v]_{\lambda}\equiv[v]^N_{\lambda} \coloneqq& \left( \left[v_1-\frac{1}{2}N^{1-\lambda},v_1+\frac{1}{2}N^{1-\lambda}\right.\right] \times \left[ \left.v_2-\frac{1}{2}N^{1-\lambda},v_2+\frac{1}{2}N^{1-\lambda}\right]\right) \cap V_N,
	\end{align}
	and set $[v]^N_{0} \coloneqq V_N$ and $[v]^N_{1} \coloneqq \{v\}$. We denote by $[v]^o_{\lambda}$ the interior of $[v]_\lambda$. Let $\mathcal{F}_{\partial{[v]_{\lambda}} \cup [v]_{\lambda}^c } \coloneqq \sigma\left(\{ \phi^N_v, v \notin [v]_{\lambda}^o \}\right)$ be the $\sigma-$algebra generated by the random variables outside $[v]_{\lambda}^o$.
	For $v\in V_N$, let
	\begin{align}\label{equation:condition_dgff__lambda_env}
	\phi^N_v(\lambda)=\mathbb{E}\left[\phi^N_v | \mathcal{F}_{\partial{[v]_{\lambda}} \cup [v]_{\lambda}^c } \right],\quad  \lambda \in [0,1].
	\end{align}
	We denote by $\nabla \phi^N_v(\lambda)$ the derivative $\partial_\lambda \phi^N_v(\lambda)$ of the DGFF at vertex $v$ and scale $\lambda$.
	Moreover, let $s\mapsto \sigma(s)$ be a non-negative function such that $\mathcal{I}_{\sigma^2}(\lambda)\coloneqq\int_{0}^{\lambda}\sigma^2(x)\mathrm{d}x$ is a function on $[0,1]$ with $\mathcal{I}_{\sigma^2}(0)=1$ and $\mathcal{I}_{\sigma^2}(1)=1$.
	Then the $2$d scale-inhomogeneous DGFF on $V_N$ is a centred Gaussian field $ \psi^N \coloneqq \{\psi^N_v \}_{v \in V_N}$ defined as
	\begin{align}\label{equation:def_scale_dgff}
	\psi^N_v\coloneqq \int_{0}^{1} \sigma(s) \nabla \phi^N_v(s) \mathrm{d}s.
	\end{align}
	In the case when $\sigma$ is a right-continuous step function taking finitely many values, \cite[(1.11)]{paper1} shows that it is a centred Gaussian field with covariance given by
	\begin{align}
	\mathbb{E}\left[\psi^{N}_v\psi^N_w\right]= \log N \mathcal{I}_{\sigma^2}\left(\frac{\log N - \log_{+} \|v-w\|_2}{\log N}\right) +O(\sqrt{\log(N)}),\quad \text{for } v,w\in V_N^\delta.
	\end{align}
\end{definition}
	\section{Main results}
	In the case of finitely many scales, Arguin and Ouimet \cite{MR3541850} showed the first order of the maximum and the size of the level sets.
\begin{assumption}\label{assumption:1}
	In the rest of the paper, $\{\psi^N_v\}_{v\in V_N}$ is always a 2d scale-inhomogeneous DGFF on $V_N$. Moreover, we assume that $\mathcal{I}_{\sigma^2}(x)<x$, for $x\in (0,1),$ and that $\mathcal{I}_{\sigma^2}(1)=1$, with $s\mapsto \sigma(s)$ being differentiable at $0$ and $1$, such that $\sigma(0)<1$ and $\sigma(1)>1$.
\end{assumption}
In \cite{paper1}, we proved, in the case when $s\mapsto \mathcal{I}_{\sigma^2}$ is piecewise linear, that the maximum centred by $m_N$ has exponential tails. In particular, in the case of the right-tail, our results are precise up to a multiplicative constant. As a simple consequence we obtained the sub-leading logarithmic correction to the maximum value . Provided \autoref{assumption:1}, there are right-continuous, non-negative step functions, $s\mapsto \sigma_1(s),\, s\mapsto  \sigma_2(s)$, taking finitely many values, such that, for $x\in (0,1)$,
\begin{align}\label{equation:1}
	\mathcal{I}_{\sigma_1^2}(x)\leq\mathcal{I}_{\sigma^2}(x)\leq\mathcal{I}_{\sigma_2^2}(x)<x,
\end{align}
and such that $\mathcal{I}_{\sigma_1^2}(1)=\mathcal{I}_{\sigma_2^2}(1)=1$. \cite{paper1} shows that for scale-inhomogeneous DGFFs with parameters $\sigma_1$ or $\sigma_2$, the maximum value is given by $2\log N-\frac{1}{4}\log \log N +O_P(1)$, where $O_P(1)$ means that remainder is stochastically bounded and that the centred maxima are tight. \eqref{equation:1}, Sudakov-Fernique and \cite{paper1} imply that the maximum value under \autoref{assumption:1} is given by
\begin{align}\label{equation:expected_max}
\psi^*_N\coloneqq \max_{v \in V_N} \psi^N_v=2\log N-\frac{1}{4}\log \log N +O_P(1).
\end{align}
In particular, the maximum, $\psi^*_N$, centred by $m_{N} \coloneqq 2\log N-\frac{\log \log N}{4}$ is tight. Our main result in this paper is convergence in distribution of the centred maximum.
\begin{thm}\label{thm:convergence_in_law}
	Let $\{\psi^N_v\}_{v\in V_N}$ satisfy \autoref{assumption:1}. Then, the sequence $\left\{\psi^{*}_N- m_{ N} \right\}_{N\geq 0}$ converges in distribution. In particular, there is a constant $\beta(\sigma(1))>0$ depending only on the final variance, and a random variable $Y(\sigma(0))$ which is almost surely non-negative, finite and depends only on the initial variance, such that, for any $z\in \mathbb{R}$,
	\begin{align}
		\mathbb{P}\left(\psi^{*}_N - m_{ N} \leq +z\right)\overset{N\rightarrow \infty}{\rightarrow} \mathbb{E}\left[\exp\left[-\beta(\sigma(1)) Y(\sigma(0))e^{-2z}\right]\right],\quad \text{as }N\rightarrow \infty.
	\end{align}
\end{thm}
Note that the limiting law is universal in the sense that only $\sigma(0)$ and $\sigma(1)$ affect the limiting law. In particular, the choice of $\sigma(s)$, for $s\in (0,1)$, does not affect the law, as long as $\mathcal{I}_{\sigma^2}(x)<x$, for $x\in (0,1)$.
In the proof of \autoref{thm:convergence_in_law} one needs to understand the genealogy of particles close to the maximum. Since this is of independent interest, we state it as a separate theorem.
\begin{thm}\label{thm:prob_dist_maximalparticles}
		Let $\{\psi^N_v\}_{v\in V_N}$ satisfy \autoref{assumption:1}. Then, there exists a constant $c>0$, such that
	\begin{align}\label{equation:thm2}
	\lim\limits_{r\rightarrow \infty}\lim\limits_{N\rightarrow \infty}\mathbbm{P}\left(\exists u,v\in V_N \text{ with } r\leq \|u-v\|_2\leq \frac{N}{r} \text{ and } \psi^N_u,\psi^N_v\geq m_{ N}-c\log\log r\right)=0.
	\end{align}
\end{thm}
As the field is strongly correlated, \autoref{thm:prob_dist_maximalparticles} implies that local maxima of the scale-inhomogeneous DGFF are surrounded by very heigh points in $O(1)$ neighbourhoods. Moreover, the local maxima are at distance $O(N)$ to each other and therefore, almost independent.
	\subsection{Related work}
	The special case $\sigma(x) \equiv 1$, for $x\in [0,1]$, is the usual 2d DGFF. In this case, building upon work by Bolthausen, Bramson, Deuschel, Ding, Giacomin and Zeitouni \cite{MR2772390,MR2846636,MR3101848,MR3262484}, Bramson, Ding and Zeitouni \cite{MR3433630} proved convergence in law of the centred maximum. Generalizing this approach, Ding, Roy and Zeitouni \cite{2015arXiv150304588D} proved convergence of the centred maximum for more general log-correlated Gaussian fields. In the 2d DGFF, Biskup and Louidor \cite{MR3509015,2016arXiv160600510B} proved convergence of the full extremal process to a cluster Cox process.
Moreover, they derived several properties of the random intensity measure appearing in the Cox process, which they identified as the so-called critical Liouville quantum gravity measure. \\
\indent Another closely related model is (variable-speed) branching Brownian motion (BBM). Variable-speed BBM, introduced by Derrida and Spohn \cite{MR971033}, is the natural analogue model of the 2d scale-inhomogeneous DGFF in the context of BBM. In order to define the model, fix a time horizon $t>0$, a super-critical (continuous time) Galton-Watson tree and a strictly increasing function $A:[0,1]\rightarrow [0,1]$, with $A(0)=0,\, A(1)=1$. For two leaves $v$ and $w$, we denote by $d(v,w)$ their overlap, which is the time of their most recent common ancestor. Variable-speed BBM in time $t$, is a centred Gaussian process, indexed by the leaves of a super-critical (continuous time) Galton-Watson tree, and covariance $t A(d(v,w)/t)$. BBM is the special case when $A(x)=x$, for $x\in [0,1]$. It coincides with the continuous random energy model (CREM) on the Galton-Watson tree \cite{Derrida1,Derrida2,MR2070335}. The extremal process of BBM was investigated in \cite{MR0494541,MR893913,MR3101852,MR2838339,MR2985174,MR3335016,MR3129797,MR3678484}, and those of variable-speed BBM in \cite{MR3164771,MR3351476,MR2981635,MR3531703}. In the weakly correlated regime, i.e. when $A(x)<x$, for $x\in (0,1)$, $A^\prime(0)<1$ and $A^\prime(1)>1$, Bovier and Hartung \cite{MR3164771,MR3351476} proved convergence of the extremal process to a cluster Cox process. They identified the random intensity measure as the so-called ``McKean-martingale'' which differs from the random intensity measure, the ``derivate-martingale'', which appears in BBM.
Works by Bovier and Kurkova \cite{MR2070335} for general variance profiles show that in the context of GREM the first order of the maximum is determined by the concave hull of $A$. Building upon results obtained by Fang and Zeitouni \cite{MR2981635}, Maillard and Zeitouni \cite{MR3531703} proved in the case variable-speed BBM with strictly decreasing speed, that the 2nd order correction is proportional to $t^{1/3}$. As also in the case of the 2d scale-inhomogeneous DGFF all variances profiles can be achieved, studying its extremes in the analogue setting of strictly decreasing speed would be of great interest.
	\subsection{Outline of proof}\label{section:sketch_pf}
	We start to explain the proof of \autoref{thm:prob_dist_maximalparticles} as these ideas are also used in the proof of \autoref{thm:convergence_in_law}. In order to prove \autoref{thm:prob_dist_maximalparticles}, we have to show with high probability, that there cannot be two vertices in $V_N$ at ``intermediate distance'' to each other, i.e. in between $O(1)$ and $O(N)$, and both reaching an extremal height. We therefore study the sum of two vertices, under the additional restriction that their distance is ``intermediate'', i.e. such that $r\leq \|u-v\|\leq N/r$ with $r\ll N$. The idea here is, if both vertices reach extreme heights, their sum must exceed twice an extremal threshold. This reasoning works, since tightness of the centred maximum implies that there cannot be a vertex being considerably larger than the expected maximum. To analyse the maximum of the sum of particles of the scale-inhomogeneous DGFF, we prove a variant of Slepian's lemma which allows to compare this quantity with the maximum of the sum of particles of corresponding inhomogeneous branching random walks. We show that using a truncated second moment method.
\begin{center}
	\begin{tikzpicture}[scale=0.4]\label{pic:3field_decomp}
	\draw (0,0) grid [step=3cm] (12,12) node at (6,0) (center_bottom) {};
	\draw node [align=flush center] at (6,-0.8) {$V_N$};
	\filldraw (10.5,1.5) circle (1pt);
	\filldraw (10.5,4.5) circle (1pt);
	\draw node at (10.5,1.2) {$B_{N/K,i}$};
	\draw node at (10.5,4.2) {$B_{N/K,j}$};
	\draw[red] (9,9) grid [step=0.5cm] (12,12);
	\filldraw (11.75,10.25) circle (1pt)   node[anchor=west,xshift=2ex]  {$B_{K^\prime,i}$};
	\draw[thick] (0,0) rectangle (12,12);
	\draw (10.5,1.5) -- (12.5,1.5);
	\draw (10.5,4.5) -- (12.5,4.5);
	\draw (12.5,1.5) -- (12.5,4.5);
	\draw node at (16,3) {``coarse field''};
	\draw (11.75,10.25) -- (12.5,9.1) node[anchor=west] {``local field''};
	\draw (10.5,7.5) circle (1pt);
	\draw (10.5,7.5) -- (12.9,7.5) node[anchor=west] {``intermediate field''};
	
	\node(V_N_1) at(-0.125,12.2) {};
	\node(V_N_2) at (12.125,12.2) {};
	\draw (V_N_1) edge (V_N_2);
	\draw (0,12.1) -- (0,12.5);
	\draw (12,12.1) --(12,12.5);
	\node at (6,12.65) {$N$};
	
	\node(b1) at (8.8,9.5) {};
	\node(b2) at (8.8,10.0) {};
	\draw (8.8,9.5) -- (8.8,10);
	\draw (8.6,9.5) -- (8.9,9.5);
	\draw (8.6,10) -- (8.9,10);
	\node at (7.7,9.74) {$K^\prime$};
	
	\draw (0,-0.2) -- (3,-0.2);
	\draw (0,-0.1) -- (-0,-0.4);
	\draw (3,-0.1) -- (3,-0.4);
	\node at (1.5,-0.8) {$N/(K)$};
	\node [below=1cm, align=flush center,text width=8cm] at (6,0)
	{
		Figure $1$: $3$-field decomposition
	};
	\end{tikzpicture}
\end{center}
\par \autoref{thm:prob_dist_maximalparticles} suggests that to understand the law of the centred maximum, it suffices to consider local maxima in ``small'' $O(1)$ neighbourhoods, while the ``small'' neighbourhoods are far, i.e. $O(N)$, apart. The fact that these neighbourhoods are very far apart, makes them correlated only on the level of the first increments, $\phi^N_v(\lambda_{1})-\phi^N_v(0)$, for some $\lambda_{1}>1$, as boxes of side length $N^{1-\lambda_{1}}$ and centred at local maxima do not overlap. In particular, the remaining increments, $\phi^N_v(\lambda)-\phi^N_v(\mu)$, for $\lambda>\mu\geq \lambda_1$, for distinct such neighbourhoods are independent. We split these two different contributions by studying the sum of two independent Gaussian fields. To do so, decompose the box $V_N$ into $K^2$ boxes $B_{N/K,i}$ and $(N/K^\prime)^2$ boxes $B_{K^\prime,j}$ with side lengths $N/K$ and $K^\prime$, where $K,K^\prime \ll N$. One of the Gaussian fields is the ``coarse field'', which is defined such that it is constant in each box $B_{N/K,i}$. It encodes initial increments and correlations of the field between different boxes $B_{N/K,i}$. The other Gaussian field is the ``fine field''. It is independent between different boxes $B_{N/K,i}$, and encodes the remaining increments, including the local neighbourhoods. The ``fine field'' is then decomposed further into a field capturing the ``intermediate'' increments and an independent ``local field'', which captures the increments in the small neighbourhoods, $B_{K^\prime,j}$, that carry the local maxima. Instead of working directly with the scale-inhomogeneous 2d DGFF, we define a Gaussian field, $\{S^N_v\}_{v\in V_N}$, as a sum of four independent Gaussian fields, with covariance structure of the ``coarse field'', ``local field'', ``intermediate field'' and an additional independent Gaussian field. The additional field is defined such that variances of the scale-inhomogeneous DGFF and the approximating field match asymptotically, which is crucial in order to use Gaussian comparison to reduce the proof of \autoref{thm:convergence_in_law} to show convergence of the centred maximum of the approximating process, $\{S^N_v\}_{v\in V_N}$. The ``coarse and local field'' are instances of appropriately scaled 2d DGFFs, the ``intermediate field'' is a collection of modified branching random walks (MIBRW). The advantage of working with the approximating process is that the ``coarse field'' is constant in large boxes, which substantially simplifies the analysis. To justify this approximation, it is essential to control its covariance structure, and how it differs from that of the scale-inhomogeneous DGFF. In particular, one needs to understand the influence of this difference on the law of the centred maximum. This is done similarly as in \cite{2015arXiv150304588D}, adapting an idea from \cite{MR3509015}, to show a certain invariance principle: Partition $V_N$ into sub-boxes $V_L$, where $L$ can be either of order $K$ or $N/K$, with $K\ll N$. If one adds i.i.d. Gaussians of bounded variance to each sub-box $V_L$, i.e. the same random variable to each vertex in a sub-box, then the law of the centred maximum is given by a deterministic shift of the original law. Moreover, the shift can be stated explicitly. This is the contents of \autoref{lemma:1} and its proof uses \autoref{thm:prob_dist_maximalparticles} and Gaussian comparison.\\ Another key step in the proof of convergence in law of the centred maximum of the approximating process, $\{S^N_v\}_{v\in V_N}$, is to understand the correct right-tail asymptotics of the (auxiliary) process. This is provided in \autoref{proposition:sharp_right_tail_estimate}, which is proved using a truncated second moment method. The truncation uses a localizing property of vertices reaching extreme heights, similar to the one observed in variable speed BBM. The idea is that intermediate increments of extremal vertices have to stay far below the maximum possible increment. For vertices to become very heigh at the end, this is then compensated by extraordinarily huge final increments. Based on a localization of increments of the auxiliary process for vertices that are local extremes (cp. \autoref{proposition:position_at_variance_change}), one is able to define random variables with the correct tails and distributions, whose parameters are determined through those of the ``coarse and local field'', and therefore independent of $N$. This is done in \eqref{equation:4.35}, \eqref{equation:Bernoulli_rv} and \eqref{equation:4.37}. These are then coupled to the auxiliary process and allow to obtain convergence in law of the centred maximum, and further, for an explicit description of the limit distribution.
\\ \\
\textit{Outline of the paper:}
In \autoref{section:aux_processes} we recall the definition of the corresponding inhomogeneous branching random walk (IBRW) and the modified inhomogeneous branching random walk (MIBRW), introduced in \cite{paper1}, and state covariance estimates. The proof of \autoref{thm:prob_dist_maximalparticles}  is provided in \autoref{section:pf2} and the proof of \autoref{thm:convergence_in_law} in \autoref{section:pf1}. In \autoref{section:appendix} we state Gaussian comparison tools such as Slepian's lemma, the inequality of Sudakov-Fernique and provide proofs of the additional covariance estimates. \autoref{lemma:1} and \autoref{lemma:2} are proved in \autoref{section:justification_approximation}, and the proof of the right-tail asymptotics, i.e. \autoref{proposition:sharp_right_tail_estimate}, is provided in \autoref{subsection:proof_right_tail}.
	\section*{Acknowledgements}
	The authors want to thank Anton Bovier for his careful reading and for his valuable comments that led to improvements in the presentation of this paper.
	We would also like to thank both authors' home institutions for their hospitality.
	\section{Frequently occurring auxiliary processes}\label{section:aux_processes}
	\subsection{Inhomogeneous branching random walk}
	Let $n\in \mathbb{N}$ and set $N=2^n$. For $k = 0,1,\dotsc,n$, let $\mathcal{B}_k$ denote the collection of subsets of $\mathbbm{Z}^2$ consisting of squares of side length $2^k $ with corners in $\mathbbm{Z}^2$, and let $\mathcal{BD}_k$ denote the subset of $\mathcal{B}_k$ consisting of squares of the form $([0,2^k - 1]\cap \mathbbm{Z})^2+ (i2^k,j2^k).$ Note that the collection $\mathcal{BD}_k$ partitions $\mathbbm{Z}^2$ into disjoint squares. For $v \in V_N,$ let $\mathcal{B}_k (v)$ denote the set of  elements $B \in \mathcal{B}_k$ with $v \in B$. Let $B_{k}(v)$ be the unique box $B_{k}(v)\in \mathcal{BD}_k$ that contains $v$.
\begin{definition}[Inhomogeneous branching random walk (IBRW)]\label{definition:ibrw}
	Let $\{a_{k,B}\}_{k\geq 0,B \in \mathcal{BD}_k}$ be an i.i.d. family of standard Gaussian random variables. Define the inhomogeneous branching random walk $ \{ R_v^N \}_{v\in V_N}$, by
	\begin{equation}\label{equation:3.1}
		R_v^N (t)\coloneqq \sum_{k=n-t}^{n} \sqrt{\log(2)}a_{k,B_k(v)}\int_{n-k-1}^{n-k} \sigma\left(\frac{s}{n}\right)\mathrm{d}s ,
	\end{equation}
	where $0 \leq t \leq n,$ $t \in \mathbbm{N}$.
\end{definition}
	\subsection{Modified inhomogeneous branching random walk}
	For $N=2^n$, $v \in V_N$, let $\mathcal{B}_k^N (v)$ be the collection of subsets  of $\mathbbm{Z}^2$ consisting of squares of size $2^k$ with lower left corner in $V_N$ and containing $v$. Note that the cardinality of $\mathcal{B}_k^N (v)$ is $2^k$. For two sets $B,B^\prime$, write $B\sim_N B^\prime$ if there are integers, $i,j$, such  that $B^\prime= B+(iN,jN)$. Let $\{b_{k,B} \}_{k \geq 0 , B \in \mathcal{B}_k^N}$ denote an i.i.d. family of centred  Gaussian random variables with unit variance, and define
\begin{equation}
	b_{k,B}^N \coloneqq \begin{cases}
							b_{k,B}, \, B \in \mathcal{B}_k^N, \\
							b_{k,B^{'}}, \, B \sim_N B^{'} \in \mathcal{B}_k^N.
						\end{cases}
\end{equation}
\begin{definition}[Modified inhomogeneous branching random walk (MIBRW)]
	The modified inhomogeneous branching random walk (MIBRW) $\{\tilde{S}_v^N\}_{v \in V_N}$ is defined by
	\begin{equation}\label{equation:def_mibrw}
		\tilde{S}_v^N (t) \coloneqq \sum_{k=n-t}^{n} \sum_{B \in \mathcal{B}_k^N (v)} 2^{-k} \sqrt{\log(2)}  b_{k,B}^N \int_{n-k-1}^{n-k} \sigma\left(\frac{s}{n}\right)\mathrm{d}s,
	\end{equation}
	where $0 \leq t \leq n,$ $t \in \mathbbm{N}$.
\end{definition}
	\subsection{Covariance estimates}
	In order to compare the auxiliary processes with the scale-inhomogeneous DGFF, one needs estimates on their covariances, which are provided in this section. Let $\|\cdot \|_2$ be the usual Euclidean distance and $\|\cdot \|_\infty$ the maximum distance. In addition, introduce the following two distances on the torus induced by $V_N$, i.e. for $v,w\in V_N$,
\begin{align}
	&d^N(v,w) \coloneqq \min_{z:\, z-w \in (N\mathbb{Z})^2} \|v-z\|_2, &d_\infty^N(v,w)\coloneqq \min_{z: \, z-w \in (N\mathbb{Z})^2}\|v-z\|_\infty.
\end{align}
Note that the Euclidean distance on the torus is smaller than the standard Euclidean distance, i.e. for all $v,w\in V_N$, it holds $d^N(v,w)\leq \|v-w\|_2$. Equality holds if $v,w \in (\sfrac{N}{4},\sfrac{N}{4}) + V_{\sfrac{N}{2}} \subset V_N.$
\begin{lemma}\label{lemma:cov_comp}\cite[Lemma~3.3]{paper1}
	For any $\delta>0$, there exists a constant $\alpha>0$ independent of $N=2^n$, such that the following estimates hold: For any $v,w \in V_N$,
	\begin{enumerate}
		\item[i.] $\left|\mathbb{E}\left[\tilde{S}^N_v\tilde{S}^N_w\right]- \log N \mathcal{I}_{\sigma^2}\left(1-\frac{\log_{+}d^N(v,w)}{\log N}\right)\right| \leq \alpha.$
	\end{enumerate}
Further, for any $u,v\in V^\delta_N$, and any $x,y \in V_N +(2N,2N)\subset V_{4N}:$ 
\begin{enumerate}
		\item[ii.] $\left|\mathbb{E}\left[\psi^{N}_u \psi^{N}_v\right]-\log N \mathcal{I}_{\sigma^2}\left(1-\frac{\log_{+} \|u-v\|_2}{\log N}\right)\right| \leq \alpha $
		\item[iii.] $\left|\mathbb{E}\left[\psi^{4N}_x\psi^{4N}_y\right]- \mathbb{E}\left[\tilde{S}^{4N}_x\tilde{S}^{4N}_y\right]\right| \leq \alpha.$ \label{cov_est_6}
	\end{enumerate}
\begin{proof}
	The proof is given in \autoref{subsection:covariance_estimates}.
\end{proof}
\end{lemma}
In the following lemma, we identify the asymptotic behaviour of covariances of the scale-inhomogeneous 2d DGFF close to the diagonal and for two vertices at macroscopic distance, i.e. at distance of order of the side length of the underlying box.
\begin{lemma}\label{lemma:convergence_cov_neardiag_offdiag}
	There are continuous functions, $f:(0,1)^2\mapsto \mathbb{R}$ and $h:[0,1]^2\backslash\{(x,x):x\in [0,1]\}\mapsto \mathbbm{R}$, and a function, $g:\mathbbm{Z}^2\times \mathbbm{Z}^2\mapsto \mathbbm{R}$, such that the following two statements hold:
	\begin{enumerate}
		\item[i.]	For all $L,\epsilon,\delta>0$, there exists an integer $N_0=N_0(\epsilon,\delta,L)>0$ such that, for all $x\in [0,1]^2$ with $xN\in V_N^{\delta},$ $u,v\in [0,L]^2$ and $N\geq N_0$, we have
		\begin{align}
			\left|\mathbbm{E}\left[\psi^N_{xN+u}\psi^N_{xN+v}\right] -\log(N)-\sigma^2(0) f(x)-\sigma^2(1) g(u,v) \right|<\epsilon.
		\end{align}
		\item[ii.]	For all $L,\epsilon,\delta>0$, there exists an integer $N_1=N_1(\epsilon,\delta,L)>0$ such that, for all $x,y\in [0,1]^2$ with $xN, yN\in V_N^{\delta}$ as well as $|x-y|\geq 1/L$ and $N\geq N_1$, we have
		\begin{align}
			\left|\mathbbm{E}\left[\psi^N_{xN}\psi^N_{yN}\right]-\sigma^2(0) h(x,y)\right|<\epsilon.
		\end{align}
	\end{enumerate}
\begin{proof}
	The proof is given in \autoref{subsection:covariance_estimates}.
\end{proof}
\end{lemma}
	\section{Proof of \autoref{thm:prob_dist_maximalparticles}}\label{section:pf2}
	In order to prove \autoref{thm:prob_dist_maximalparticles}, we have to show with high probability that there cannot be two vertices at ``intermediate distance'' to each other and both reaching an extremal height. We therefore study the sum of two vertices, under the additional restriction that their distance is ``intermediate''. For such sums, we first prove a version of Slepian's lemma, which relates these functionals of the scale-inhomogeneous DGFF to the same functionals of a suitable IBRW.
\begin{lemma}\label{lemma:slepian_sums}
	Let $\{\chi^N_v\}_{v\in V_N}$ and $\{\eta^N_v\}_{v\in V_N}$ be two centred Gaussian processes, such that
	\begin{align}\label{equation:4..1}
		\mathbb{E}\left[\eta^N_u \eta^N_v\right]&\leq \mathbb{E}\left[\chi^N_u \chi^N_v\right]\quad \forall u,v\in V_N,\\
		\mathrm{Var}\left[\eta^N_u\right]&=\mathrm{Var}\left[\chi^N_u\right]\qquad \,\, \forall u\in V_N.
	\end{align}
	Let $\Omega_{m,r}\coloneqq \{A\subset V_N: |A|=m, u,v\in A \Rightarrow r\leq \|u-v\|_2 \leq N/r \}$.
	For any $r\geq 0$, $N>r$ and any $\lambda \in \mathbb{R}$, it holds that
	\begin{align}\label{equation:slepian_sum}
	\mathbb{P}\left(\max_{A\in \Omega_{m,r}}\sum_{v\in A}\eta^N_v\leq \lambda \right)\leq	\mathbb{P}\left(\max_{A\in \Omega_{m,r}}\sum_{v\in A}\chi^N_v\leq \lambda \right).
	\end{align}
	\begin{proof}
		The idea is to use Gaussian interpolation. We first introduce the necessary set-up. For $h\in [0,1]$ and $u\in V_N$, let
		\begin{align}\label{equation:def_X}
		X^h_u= \sqrt{h}\eta^N_u+\sqrt{1-h} \chi^N_u
		\end{align}
		be a Gaussian random variable, interpolating between the scale-inhomogeneous DGFF and the time-inhomogeneous BRW.
		Moreover, let $s>0$, set $\Phi_s(x)=\frac{1}{\sqrt{2 \pi s^2}}\int_{-\infty}^{x} \exp\left[-\frac{z^2}{2s^2}\right]\mathrm{d}z$ and write $x_A =\sum_{v\in A} x_v$, for $A\subset V_N$.
		We define
		\begin{align}\label{equation:def_F__s}
		F_{s}(x_1,\dotsc, x_{4^n})= \prod_{A\in \Omega_{m,r}}\Phi_s(\lambda-x_A).
		\end{align}
		Clearly, $F_s$ is bounded uniformly in $s$, smooth for all $s>0$, and converges pointwise to $F(x_1,\dotsc,x_{4^n})=\mathbbm{1}_{x_A\leq u, \in A\in \Omega_{m,r}}$ at all continuity points of $F$. We have that, for $i\neq j$,
		\begin{align}\label{equation:2nd_derivatives}
		\frac{\partial^2 F_s}{\partial x_i \partial x_j}(x_1,\dotsc,x_{4^n})&=\sum_{\substack{A\in \Omega_{m,r}\\x_i,x_j \in A }}\Phi_s^{\prime \prime}(\lambda-x_A)\prod_{B\in \Omega_{m,r}, B\neq A}\Phi_s(\lambda-x_B)\nonumber\\
		&\quad+\sum_{\substack{A\in \Omega_{m,r}\\x_i \in A }}\sum_{\substack{B\in \Omega_{m,r}\\x_j \in B, B\neq A }}\Phi_s^\prime(\lambda-x_A)\Phi_s^\prime(\lambda-x_B)\prod_{C\in \Omega_{m,r}, C\neq A,B}\Phi_s(\lambda-x_C).
		\end{align}
		Observe that, by dominated convergence, for $A\in \Omega_{m,r}$,
		\begin{align}\label{equation:1st}
		\mathbb{E}\left[\Phi_s^{\prime \prime}(\lambda-X^h_A)\right]=\int \phi_{h,A}(x) \frac{\lambda-x}{\sqrt{2 \pi s^2}s^2}\exp\left[-\frac{(\lambda-x)^2}{2s^2}\right]\mathrm{d}x\rightarrow 0,
		\end{align}
		as $s \rightarrow 0$, and where $\phi_{h,A}$ is the density of the centred Gaussian $\sum_{v\in A}X^h_v$. By \eqref{equation:1st} and as $\prod_{B\in \Omega_{m,r}, B\neq A}\Phi_s(\lambda-x_B)\leq 1$,
		\begin{align}\label{equation:1st_sum}
		\sum_{\substack{A\in \Omega_{m,r}\\x_i,x_j \in A }} \mathbb{E}\left[\Phi_s^{\prime \prime}(\lambda-X^h_A)\prod_{B\in \Omega_{m,r}, B\neq A}\Phi_s(\lambda-X^h_B)\right]\rightarrow 0,
		\end{align}
		as $s\rightarrow 0$.
		Next, we turn to the second sum in \eqref{equation:2nd_derivatives}. For $A,B \in \Omega_{m,r}$, we have
		\begin{align}\label{equation:a.18}
		\mathbb{E}\left[\Phi_s^\prime(\lambda-X^h_A)\Phi_s^\prime(\lambda-X^h_B)\prod_{C\in \Omega_{m,r}, C\neq A,B}\Phi_s(\lambda-X^h_C)\right]
		\leq \mathbb{E}\left[\Phi_s^\prime(\lambda-X^h_A)\Phi_s^\prime(\lambda-X^h_B)\right]\nonumber\\
		=\int  \phi_{h,A,B}(x,y)\frac{1}{2\pi s^2}\exp\left[-\frac{(\lambda-x)^2+(\lambda-y)^2}{2s^2}\right]\mathrm{d}x \mathrm{d}y\rightarrow \phi_{h,A,B}(\lambda,\lambda),
		\end{align}
		as $s\rightarrow 0$ and where
		\begin{align}
		\phi_{h,A,B}(x,y)=\frac{1}{2\pi \sigma^h_A \sigma^h_B\sqrt{1-\varrho_{h,A,B}^2}}\exp\left[-\frac{1}{2(1-\varrho_{h,A,B}^2)}\left(\frac{x^2}{(\sigma^h_A)^2}+\frac{y^2}{(\sigma^h_B)^2}-2\varrho_{h,A,B}\frac{xy}{\sigma^h_A\sigma^h_B}\right)\right]
		\end{align}
		with $(\sigma^h_A)^2=\mathrm{Var}\left[\sum_{v\in A}X^h_v\right],\,(\sigma^h_B)^2=\mathrm{Var}\left[\sum_{v\in B}X^h_v\right]$ and $\varrho_{h,A,B}=\frac{\mathbb{E}\left[\left(\sum_{v\in A} X^h_A\right) \left(\sum_{v\in B}X^h_B\right)\right]}{\sqrt{\mathrm{Var}\left[\sum_{v\in A}X^h_A\right]\mathrm{Var}\left[\sum_{v\in B}X^h_v\right]}}$. $\phi_{h,A,B}(x,y)$ is the density of the bivariate distributed random vector $\left(\sum_{v\in A} X^h_v,\sum_{v\in B}X^h_v\right)$.
		Observe that,
		\begin{align}\label{equation:upbound_bivariate}
		\phi_{h,A,B}(x,y)\leq \frac{1}{2\pi \sqrt{1-\varrho_{A,B}^2}}\exp\left[-\frac{x^2+y^2}{2(1+\varrho_{A,B})}\right],
		\end{align}
		where $\varrho_{A,B}=\max\left(\mathbb{E}\left[\left(\sum_{v\in A}\eta_v\right)\left(\sum_{v\in B}\eta_v\right)\right],\mathbb{E}\left[\left(\sum_{v\in A}\chi^N_v\right)\left(\sum_{v\in B}\chi^N_v\right)\right]\right)$. Inserting \eqref{equation:upbound_bivariate} into \eqref{equation:a.18} and using this with \eqref{equation:1st_sum} in \eqref{equation:2nd_derivatives} and letting $s\rightarrow 0$, allows to use Kahane's theorem \cite{MR858463}, to obtain
		\begin{align}\label{equation:comparison1}
		\mathbb{P}\left(\forall A\in \Omega_{m,r}:\, \sum_{v\in A} \eta_v\leq \lambda \right)-\mathbb{P}\left(\forall A\in \Omega_{m,r}:\, \sum_{v\in A}\chi^N_v\leq \lambda \right) \qquad\qquad\qquad\qquad\qquad\nonumber\\
		\leq\frac{1}{2\pi} \sum_{1\leq i< j\leq 4^n} \sum_{\substack{A\in \Omega_{m,r}\\x_i \in A }} \sum_{\substack{B\in \Omega_{m,r}\\x_j \in B, B\neq A }} \frac{(\Lambda^1_{A,B}-\Lambda^0_{A,B})_{+}}{\sqrt{1-\varrho_{A,B}^2}} \exp\left[-\frac{2 \lambda^2}{2(1+\varrho_{A,B})}\right],
		\end{align}
		with $\Lambda^0_{A,B}=\mathbb{E}\left[\left(\sum_{v\in A}\chi^N_v\right)\left(\sum_{v\in B}\chi^N_v\right)\right]$ and $\Lambda^1_{A,B}= \mathbb{E}\left[\left(\sum_{v\in A}\eta_v\right)\left(\sum_{v\in B}\eta_v\right)\right]$.
		By \eqref{equation:4..1}, $(\Lambda^1_{A,B}-\Lambda^0_{A,B})_+ =0$, and thus, \eqref{equation:comparison1} implies \eqref{equation:slepian_sum}.
	\end{proof}
\end{lemma}

In the following proposition, we determine the position of extremal particles of an inhomogeneous BRW at the times when its variance changes. This is a direct consequence of \cite[Proposition 2.1]{MR3164771} in the weakly-correlated regime of variable speed BBM. Set $i(t,n)\coloneqq t\wedge (n-t)$.
\begin{prop}\label{proposition:position_at_variance_change}
	Let $\{R^N_v\}_{v\in V_N}$ be a inhomogeneous BRW with $\mathcal{I}_{\sigma^2}(x)<x$, for $x\in (0,1)$ and $\sigma(0)<1<\sigma(1)$. Let $s\in \mathbbm{R}$.
	Then, there is a constant $r_0>0$ such that for any $r>r_0$, $N=2^n$, $N$ sufficiently large, and any $\gamma \in (1/2,1)$,
	\begin{align}\label{equation:prob_pos_at_variance_change_2}
	\mathbbm{P}\left(\exists v\in V_N,\,t\in[\log r,n-\log r]: R^N_v\geq m_N-s,R^N_v(t )-2\log 2\mathcal{I}_{\sigma^2}\left(\frac{t}{n}\right)n \notin [-i^\gamma(t,n) ,i^\gamma(t,n)] \right)\nonumber\\
	\leq C e^{2s}\sum_{k=\lfloor \log  r \rfloor}^{\infty}k^{\frac{1}{2}-\gamma}\exp\left[-k^{\frac{2\gamma-1}{2}}\right],
	\end{align}
	where $C\leq \frac{8}{\sqrt{\log 2}-\frac{\log n+4s}{4n}}.$
\end{prop}
By Gaussian comparison and since we have $\mathcal{I}_{\sigma^2}(x)<x$, for $x\in (0,1)$, it turns out that for our purposes, it suffices to consider a two-speed branching random walk, $(X^N_v(j))_{v\in V_N, 0\leq j \leq n}$. We choose the first speed to be $0$ and the second to be $\sigma_{max}^2$, where $\sigma_{max}=\text{ess} \sup\{ \sigma(s):0\leq s\leq 1\}$. Note that  $\sigma_{max}<\infty$, as $\mathcal{I}_{\sigma^2}(x)<x$, for $x \in (0,1)$. To match variances, the change of speed occurs at scale $1-1/\sigma_{max}^2$.
Write $u \underset{j}{\sim} v$, for $u,v \in V_N$, if $j$ is the largest integer such that $\mathcal{BD}_{n-j}(u)\cap \mathcal{BD}_{n-j}(v)\neq \emptyset$, i.e. in the language of BRW the ``splitting-time'' of $u$ and $v$ is $j$. The following Proposition is the analogue statement to \autoref{thm:prob_dist_maximalparticles} for the two-speed BRW and key in the proof of \autoref{thm:prob_dist_maximalparticles}.
\begin{prop}\label{corollary:upbound_prob_maxsum_BRW}
	There is a constant $C>0$, such that for any constant $c>0$ and any $z\geq0$,
	\begin{align}\label{corollary:upbound_prob_maxsum_BRW_statement1}
		&\mathbbm{P}\left(\exists j\in (\log r,n-\log r),\,\exists u \underset{j}{\sim}v:\, X^N_u,X^N_v\geq m_N -c\log\log r +z\right)\\
		&\quad\leq C\left(4^{-\log r}\exp\left[-4z+4c \log \log r\right] +  \log(r)^{-1/2}\exp\left[2\log 2 (1-\sigma_{max}^2)\log r+2c \log \log r-2z \right]\right).\nonumber
	\end{align}
	In particular, there are $c,r_0>0$, such that for all $r>r_0$ and $n$ sufficiently large,
	\begin{align}\label{corollary:upbound_prob_maxsum_BRW_statement3}
		\mathbbm{E}\left[\max_{u\underset{s}{\sim} v,s\in\{\log r,\dotsc,n-\log r\}} X^N_u+X^N_v\right]\leq 2 m_N-c \log\log  r.
	\end{align}
 \begin{proof}
  	We first consider the case when $u\underset{j}{\sim} v$ and $j<n/\sigma_{max}^2$. In this case, the particles split before the change in speed occurs. The speed change occurs at scale $1-\lambda=1-1/\sigma_{max}^2$. Note that there are $4^{2n-j}$ such pairs, and as the initial speed is zero, $X^N_u,X^N_v$ are independent. Hence,
 	\begin{align}\label{equation:5.28}
 	\mathbb{P}&\left(\exists j\in (\log r,\lfloor n(1-1/\sigma_{max}^2)\rfloor),\,\exists u \underset{j}{\sim} v:\, X^N_u,X^N_v\geq m_N -c\log\log r+z\right)\nonumber\\
 	&\quad\leq \sum_{j=\log r}^{\lfloor n(1-1/\sigma_{max}^2)\rfloor} 4^{2n-j}\mathbb{P}\left( X^N_u\geq m_N -c\log\log r+z\right)^2 \leq \tilde{C} \sum_{j=\log r}^{\lfloor n(1-1/\sigma_{max}^2)\rfloor} 4^{2n-j} \frac{\log (2) n}{(m_N-c\log r +z)^2}\nonumber\\
 	&\qquad\qquad\times \exp\left[-4 \log (2) n+\log n+4(z-c\log \log r)\right]\leq C 4^{-\log r}\exp\left[-4z+4c \log \log r\right].
 	\end{align}
 	where $\tilde{C},C>0$ are finite constants and the last inequality follows from a Gaussian tail bound.
 Next, we treat the case when particles split after the change of speed. Let $\gamma\in (1/2,1)$ and set $i(j,n)\coloneqq (n-\sigma_{max}^2(n-j))\wedge(\sigma_{max}^2(n-j))$ and $A_1(j)\coloneqq \{x\in \mathbb{R}:\, |x-\frac{n-\sigma_{max}^2(n-j)}{n}m_N|\leq i^\gamma(j,n)\}$. As the extremal particles of the BRW stay with high probability in $A_1(j)$, for $j \in \{\log r,\dotsc, n-\log r\}$ (see \autoref{proposition:position_at_variance_change} for a precise statement), we can compute as follows:
 	\begin{align}\label{equation:4..30}
 		&\mathbb{P}\left(\exists s\in (\lfloor n(1-1/\sigma_{max}^2) \rfloor+1,n- \log r),\,\exists u\underset{s}{\sim}v: X^N_u,X^N_v\geq  m_N- c \log\log r +z \right) \qquad\qquad\nonumber\\
 		&\qquad\leq C \sum_{j=\lfloor n(1-1/\sigma_{max}^2) \rfloor+1}^{n-\log r} \int_{A_1(j)} 4^{2n-j} \mathbb{P}\left(X^N_v(n)-X^N_v(j)\geq m_N-c\log\log r +z - x \right)^2\nonumber\\&\qquad\qquad\qquad\qquad
 		\times\frac{ 1}{\sqrt{2 \pi \log 2 (n-\sigma_{max}^2(n-j))}}\exp\left[-\frac{x^2}{2\log 2(n-\sigma_{max}^2(n-j))} \right] \mathrm{d}x+\epsilon.
 	\end{align}
 	By a Gaussian tail bound and using that by the integral restriction, $(m_N-x)^2\geq (\frac{\sigma_{max}^2(n-j)}{n} m_N-i(j,n)^\gamma)^2$, the summand in \eqref{equation:4..30} is bounded from above by
 	\begin{align}\label{equation:4..31}
 		 C \frac{\sigma_{max}^2(n-j)\exp\left[-\frac{(m_N-c\log\log r+z)^2}{\log 2(2n-\sigma_{max}^2(n-j))} \right]}{\sqrt{2 \pi \log 2 (n-\sigma_{max}^2(n-j))}(\frac{\sigma_{max}^2(n-j)}{n} m_N-c\log\log r +z-i^\gamma(j,n))^2}\nonumber\\ \times  4^{2n-j} \int_{A_1(j)} \exp\left[-\frac{\left(x-(m_N-c\log\log r+ z)\frac{2(n-\sigma_{max}^2(n-j))}{2n-\sigma_{max}^2(n-j)}\right)^2}{2\log 2\frac{(n-\sigma_{max}^2(n-j))\sigma_{max}^2(n-j)}{2n-\sigma_{max}^2(n-j)}} \right] \mathrm{d}x.
 	\end{align}
 	Changing variables, i.e. $x=\sqrt{\log 2 \sigma_{max}^2(n-j)\frac{n-\sigma_{max}^2(n-j)}{2n-\sigma_{max}^2(n-j)}}y +\frac{2(m_N-c\log\log r +z)(n-\sigma_{max}^2(n-j))}{2n-\sigma_{max}^2(n-j)}$, and neglecting the upper restriction in $A_1 (j)$, \eqref{equation:4..31} is bounded from above by
 	\begin{align}\label{equation:4..32}
 		C  &\frac{(\sigma_{max}^2(n-j))^{3/2}}{\left(\frac{\sigma_{max}^2(n-j)}{n}m_N-c\log\log r+z-i^\gamma(j,n) \right)^2\sqrt{2\pi\log 2 (2n-\sigma_{max}^2(n-j))}}\nonumber\\
 		&\times \exp\left[-\frac{(m_N-c\log\log r +z)^2}{\log 2(2n-\sigma_{max}^2(n-j))}\right] 4^{2n-j} \int_{A^\prime_1(j)} \exp\left[-y^2/2\right] \mathrm{d}y,
 	\end{align}
 	with $A^\prime_1(j) =\left[-\frac{m_N}{n  }\tilde{\sigma}(n,j)-(z-c\log \log r) \sqrt{\frac{n-\sigma_{max}^2(n-j)}{\log 2 \sigma_{max}^2(n-j)(2n-\sigma_{max}^2(n-j))}})-\frac{i^\gamma(j,n)}{\tilde{\sigma}(n,j)},+\infty \right]$, and where $\tilde{\sigma}(n,j)=\sqrt{ \frac{\sigma_{max}^2(n-j)(n-\sigma_{max}^2(n-j))}{\log 2(2n-\sigma_{max}^2(n-j))}}$. By a Gaussian tail bound applied to the integral, \eqref{equation:4..32} is bounded from above by
 	\begin{align}\label{equation:4.33}
 		  O\left(\frac{1}{(n-j)\sqrt{n-\sigma_{max}^2(n-j)}}\right) 4^{2n-j} \exp\left[-\frac{(m_N-c \log\log r+z)^2}{\log 2 (2n-\sigma_{max}^2(n-j))}-\frac{i^{2\gamma}(j,n)\log 2(2n-\sigma_{max}^2(n-j))}{2 \sigma_{max}^2(n-j)(n-\sigma_{max}^2(n-j))}\right]  \nonumber\\
 		\times \exp\left[-\frac{m_N i^\gamma(j,n)}{n}- \frac{m_N^2\sigma_{max}^2(n-j)(n-\sigma_{max}^2(n-j))}{2n^2\log 2(2n-\sigma_{max}^2(n-j))}-\frac{m_N(z-c\log \log r)}{n}\frac{n-\sigma_{max}^2(n-j)}{2n-\sigma_{max}^2(n-j)}\right].
 	\end{align}
  Keeping only the dominant terms, one sees that the exponential is bounded from above by
 	\begin{align}\label{equation:4.34}
 		\exp\left[2\log 2 (n-j)(1-\sigma_{max}^2)+2c\log \log r-2z +\frac{\sigma_{max}^2\frac{n-j}{n}+1}{2}\log n-c_1 i^\gamma(j,n)-c_2 i^{2\gamma-1}(j,n)\right],
 	\end{align}
 	where $c_1,c_2>0$ are some finite constants.
 	Inserting \eqref{equation:4.34} into \eqref{equation:4.33}, allows to bound \eqref{equation:4..30} from above by
 	\begin{align}\label{equation:4..35}
 			\sum_{j=\lfloor n(1-1/\sigma_{max}^2) \rfloor+1}^{n-\log r} O\left(\frac{1}{(n-j)\sqrt{n-\sigma_{max}^2 (n-j)}}\right) \exp\left[2(n-j)(1-\sigma_{max}^2)\log 2 -2 z+ \frac{\sigma_{max}^2\frac{n-j}{n}+1}{2}\log n   \right.\nonumber\\ \left.+2c\log\log r-c_1 i^\gamma(j,n)-c_2 i^{2\gamma-1}(j,n) \right]
 			\leq O\left(\frac{1}{\sqrt{\log r}}\right) \exp\left[2\log 2 (1-\sigma_{max}^2)\log r+2c \log \log r-2z\right].
 	\end{align}
 	Since $\sigma_{max}>1$, \eqref{equation:4..35} tends to zero, as $n\rightarrow \infty$.
 	\eqref{corollary:upbound_prob_maxsum_BRW_statement3} is an immediate consequence of \eqref{equation:5.28} and \eqref{equation:4..35}. This concludes the proof of \autoref{corollary:upbound_prob_maxsum_BRW}.
 \end{proof}
\end{prop}
Similarly, as for the IBRW, we have a localization for extremal particles of the MIBRW, which is the statement of the following lemma.
\begin{lemma}\label{lemma:localization_MIBRW}
	Let $\{\tilde{S}^N_v\}_{v\in V_N}$ be the MIBRW, defined in \eqref{equation:def_mibrw}. Let $s\in \mathbbm{R}$.
	Then, for any $\epsilon>0$, there is a constant $r_0>0$ such that for any $r>r_0$, $N=2^n$, $N$ sufficiently large, and any $\gamma \in (1/2,1)$,
	\begin{align}
		\mathbb{P}\left(\exists v\in V_N,\,t\in [\log r,n-\log r]:\,\tilde{S}^N_v\geq m_N -s,\, \tilde{S}^N_v(t)\notin [-i^\gamma(t,n),i^\gamma(t,n)]\right)\nonumber\\\leq C e^{2s}\sum_{k=\lfloor \log  r \rfloor}^{\infty}k^{\frac{1}{2}-\gamma}\exp\left[-k^{\frac{2\gamma-1}{2}}\right],
	\end{align}
	where $C\leq \frac{8}{\sqrt{\log 2}-\frac{\log n+4s}{4n}}.$
\end{lemma}
We do not give a proof here, as it is basically identical to the one of \autoref{proposition:position_at_variance_change}.
\begin{proof}[Proof of \autoref{thm:prob_dist_maximalparticles}]
	Note that the tree distance of two vertices $u,v\in V_N$ on the underlying tree of the IBRW, $\{X^N_v\}_{v\in V_N}$, is up to an additional constant smaller than the Euclidean distance. Hence, by \autoref{lemma:cov_comp} $ii.$ there is a $\kappa \in \mathbb{N}$ and non-negative constants $\{a_v\}_{v\in V_N}$ such that, for all $N\in \mathbb{N}$ and all $u,v\in V_N$,
	\begin{align}
		\mathbb{E}\left[X^{2^k N}_{2^k u}X^{2^k N}_{2^k v}\right]\leq \mathbb{E}\left[\psi^N_u \psi^N_v\right]+a_u a_v,
	\end{align}
	and
	\begin{align}
		\mathrm{Var}\left[X^{2^k N}_{2^k u}\right]=\mathrm{Var}\left[\psi^N_u\right]+a_u^2.
	\end{align}
	Thus, we may apply \autoref{lemma:slepian_sums} with $m=2$ and obtain, for any $\lambda\in \mathbb{R}$,
	\begin{align}\label{equation:prev1}
		\mathbb{P}&\left((\exists u,v\in V_N,\,r\leq \|u-v\|_2\leq N/r:\, \psi^N_u+\psi^N_v\geq \lambda \right)\nonumber\\
		&\qquad\leq \mathbb{P}\left((\exists u,v\in V_N,\,r\leq \|u-v\|_2\leq N/r:\, \psi^N_u a_u G+\psi^N_v +a_v G\geq \lambda \right)\nonumber\\
		&\qquad\leq \mathbb{P}\left((\exists u,v\in V_N,\,r\leq \|u-v\|_2\leq N/r:\, X^{2^\kappa N}_{2^\kappa u}+X^{2^\kappa N}_{2^k v}\geq \lambda \right)\nonumber\\
		 &\qquad\leq \mathbb{P}\left((\exists u,v\in V_{2^\kappa N},\,r\leq \|u-v\|_2\leq 2^\kappa N/r:\, X^{2^\kappa N}_{u}+X^{2^\kappa N}_{v}\geq \lambda \right),
	\end{align}
	where $G$ is an independent standard Gaussian.
	Choosing $\lambda= m_N-c\log \log r$ and applying \autoref{corollary:upbound_prob_maxsum_BRW} to last probability in \eqref{equation:prev1} yields \eqref{equation:thm2}, which concludes the proof of \autoref{thm:prob_dist_maximalparticles}.
\end{proof}

	\section{Proof of \autoref{thm:convergence_in_law}}\label{section:pf1}
	The following proposition allows to control the right tail of the maximum over subsets. 
	\begin{prop}\label{proposition:max_over_A}
		Let $\epsilon >0$ and $\{\bar{\psi}^N_v\}_{v\in V_N}$ be a centred Gaussian field such that, for all $v,w\in V_N$, $|\mathbb{E}\left[\bar{\psi}^N_v \bar{\psi}^N_w \right]-\mathbb{E}\left[\psi^N_v \psi^N_w\right]| \leq \epsilon$.  If $N$ is sufficiently large, then, for any $A\subset V_N$ and for all $z\geq 1, y\geq0$, we have
		\begin{align}\label{equation:statement_prop_max_over_A}
		\mathbbm{P}\left(\max_{v\in A} \bar{\psi}^N_v\geq m_N + z-y  \right)\leq C \frac{|A|}{|V_N|} e^{-2(z-y)}.
		\end{align}
		\begin{proof}[Proof of \autoref{proposition:max_over_A}.]
			By the covariance assumptions and \autoref{lemma:cov_comp} $i., iii.$ one can apply Slepian's lemma, to deduce that there exists $k\in \mathbbm{N}$, such that for all sufficiently large $N\in \mathbbm{N}$ and any $\lambda\in \mathbbm{R}$,
			\begin{align}
			\mathbbm{P}\left(\max_{v\in A} \bar{\psi}^N_v\geq \lambda \right)\leq \mathbbm{P}\left(\max_{v\in 2^k A} R^{2^k}_v \geq \lambda \right).
			\end{align}
			Thus, it suffices to show \eqref{equation:statement_prop_max_over_A} with $R^N$ instead of $\bar{\psi}^N$.
			Note that for any $v\in V_N$, $R^N_v\sim \mathcal{N}\left(0,n\log 2\right)$. Thus, by a first moment bound and a standard Gaussian tail estimate,
			\begin{align}
				\mathbb{P}\left(\max_{v\in A}R^N_v \geq m_N+y-z \right)&\leq C |A| \frac{n\log 2}{(m_N+z-y)\sqrt{2\pi n \log 2}}\exp\left[-\frac{(m_N+z-y)^2}{2n \log 2} \right]\nonumber\\ &\leq C |A| \frac{n\log 2}{(m_N+z-y)\sqrt{2\pi n \log 2}} \exp\left[-2n\log 2 +1/2 \log n-2(z-y)\right]\nonumber\\ &  \leq C \frac{|A|}{|V_N|}\exp\left[-2(z-y)\right],
			\end{align}
			where the constant $C>0$ may change from line to line and where we used that $|V_N|=2^{2n}.$
		\end{proof}
	\end{prop}

	\subsection{Approximation via an auxiliary field}\label{section:approx_3field}
	Let $N=2^n$ be an integer, much larger as any other integers forthcoming. For two integers $L=2^l$ and $K=2^k$, partition $V_N$ into a disjoint union of $(KL)^2$ boxes, with each of side length $N/KL$, and denote the partition by $\mathcal{B}_{N/KL}= \{B_{N/KL,i}:\,i=1,\dotsc,(KL)^2 \}$. Let $v_{N/KL,i}\in V_N$ be the left bottom corner of box $B_{N/KL,i}$ and write $w_i=\frac{v_{N/KL,i}}{N/KL}$. This allows to consider the grid points $\{w_i\}_{i=1,\dotsc,(KL)^2}$ as elements of $V_{KL}$.
Analogously, let $K^{\prime}=2^{k^\prime}$ and $L^\prime=2^{l^\prime}$ be another two integers and let $\mathcal{B}_{K^\prime L^\prime}=\{B_{K^\prime L^\prime,i}:\, i=1,\dotsc, [N/(K^\prime L^\prime)]^2 \}$ be a partitioning of $V_N$ with boxes $B_{K^\prime L^\prime,i}$, each of side length $K^\prime L^\prime$. The left bottom corner of a box $B_{K^\prime L^\prime,i}$ is denoted by $v_{K^\prime L^\prime,i}$.
One should think of $N/KL$ being much larger than $K^\prime L^\prime$. Considering \autoref{lemma:convergence_cov_neardiag_offdiag}, it turns out that this allows to define the corresponding approximating fields in such a way that they have only a fixed variance parameter, which makes them easier to analyse.
The macroscopic or ``coarse field'', $\{S^{N,c}_v:\, v\in V_N\}$, is defined as a centred Gaussian field on $V_N$ with covariance matrix $\Sigma^c$ and entries given by
\begin{align}\label{equation:4.1}
	\Sigma^c_{u,v}\coloneqq \sigma^2(0) \mathbbm{E}\left[\phi^{KL}_{w_i}\phi^{KL}_{w_j}\right],\quad \text{for } u\in B_{N/KL,i},\,v\in B_{N/KL,j},
\end{align}
where $\{\phi^{KL}_v\}_{v\in V_{KL}}$ is a standard 2d DGFF on $V_{KL}$. This field captures the macroscopic dependence.\\
The microscopic or ``bottom field'', $\{S^{N,b}_v:\,v\in V_N\}$, is a centred Gaussian field with covariance matrix $\Sigma^b$ defined entry-wise as
\begin{align}\label{equation:S^NB}
	\Sigma^b_{u,v}\coloneqq 
	\begin{cases}
		\sigma^2(1) \mathbbm{E}\left[\phi^{K^\prime L^\prime}_{u-v_{K^\prime L\prime,i}}\phi^{K^\prime L^\prime}_{v-v_{K^\prime L^\prime,i}} \right],& \text{if }u,v\in B_{K^\prime L^\prime,i}\\
		0,&\text{else},					
	\end{cases}
\end{align}
where $\{\phi^{K^\prime L^\prime}_v\}_{v\in V_{K^\prime L^\prime}}$ is a 2d DGFF on $V_{K^\prime L^\prime}$. This field is supposed to capture the ``local'' correlations.\\
The third Gaussian field, $\{S^{N,m}_v:\,v\in V_N \}$, is a collection of MIBRWs on $B_{N/KL,i}$, $i=1,\dotsc, (KL)^2$, i.e.
\begin{align}\label{equation:S^N,m_v}
	S^{N,m}_v\coloneqq \sum_{j=l^\prime+k^\prime}^{n-l-k}\sum_{B\in \mathcal{B}_j(v_{K^\prime L^\prime,i^\prime})}2^{-j}\sqrt{\log(2)}b^N_{i,j,B} \int_{n-j-1}^{n-j}\sigma\left(\frac{s}{n}\right)\mathrm{d}s,\quad \text{for }v\in B_{N/KL,i}\cap B_{K^\prime L^\prime,i^\prime},
\end{align}
with $\{b^N_{i,j,B}:\, i=1,\dotsc,(KL)^2,\,j\geq 0,\, B\in \mathcal{B}^N_j \}$ being a family of independent standard Gaussian random variables. 
Recall that $\mathcal{B}_j(v_{K^\prime L^\prime,i^\prime})$ is the collection of boxes $B\subset V_N$, of side length $2^j$, that contain the element $v_{K^\prime L^\prime,i^\prime}$. This field is supposed to capture the ``intermediate'' correlations.
To obtain sufficiently precise covariance estimates, one needs to avoid boundary effects, which can achieved working on a suitable subset of $V_N$. Consider therefore the partitioning into $N/L$- and $L-$boxes, i.e. $\mathcal{B}_{N/L}=\{B_{N/L,i}:\,1\leq i\leq L^2 \}$ and $\mathcal{B}_L=\{B_{L,i}:\,1\leq i\leq (N/L)^2\}$. Analogously, let $v_{N/L,i}$ and $v_{L,i}$ be the left bottom corners of boxes $B_{N/L,i}$, $B_{L,i}$ containing $v$. For a box $B$, let $B^\delta\subset B$ the set $B^\delta=\{v\in B:\min_{z\in \partial B}\|v-z\|\geq \delta l_B\},$ where $l_B$ denotes the side length of the box $B$. Finally, set
\begin{align}\label{equation:actual_decomposition}
V_{N,\delta}^{*}\coloneqq\{\mathop{\cup}_{1\leq i \leq L^2}B^{\delta}_{N/L,i}\}\cap \{\mathop{\cup}_{1\leq i \leq (KL)^2} B^{\delta}_{N/KL,i} \}\cap \{\mathop{\cup}_{1\leq i \leq (N/L)^2}B^{\delta}_{L,i} \}\cap \{\mathop{\cup}_{1\leq i \leq (N/KL)^2}B^{\delta}_{KL,i}\}.
\end{align}
As $|V^{*}_{N,\delta}|\geq (1-16\delta)|V_N|$, and using \autoref{proposition:max_over_A} with $A=(V^{*}_{N,\delta})^c$, we have
\begin{align}\label{equation:5..8}
	\mathbb{P}\left(\max_{v\in \left(V_{N,\delta}^{*}\right)^c} S^N_v \geq m_N+z \right)\leq 16 \delta\, \mathbb{P}\left(\max_{v\in V_N}S^N_v\geq m_N+z\right),
\end{align}
which tends to $0$, as $\delta \rightarrow 0$. Thus, it suffices to consider the maximum of the field on the set $V^{*}_{N,\delta}$.
\begin{center}
	\begin{tikzpicture}[scale=0.4]
	\draw (0,0) grid [step=3cm] (12,12) node at (6,0) (center_bottom) {};
	\draw node [align=flush center] at (6,-0.5) {$V_N$};
	\filldraw (10.5,1.5) circle (1pt);
	\filldraw (10.5,4.5) circle (1pt);
	\draw node at (10.5,1.2) {$B_{N/KL,i}$};
	\draw node at (10.5,4.2) {$B_{N/KL,j}$};
	\draw[red] (9,9) grid [step=0.5cm] (12,12);
	\filldraw (11.75,10.25) circle (1pt)   node[anchor=west,xshift=2ex]  {$B_{K^\prime L^\prime,i}$};
	\draw[thick] (0,0) rectangle (12,12);
	\draw (10.5,1.5) -- (12.5,1.5);
	\draw (10.5,4.5) -- (12.5,4.5);
	\draw (12.5,1.5) -- (12.5,4.5);
	\draw node at (13.4,3) {$S^{N,c}_\cdot$};
	\draw (11.75,10.25) -- (12.5,9.1) node[anchor=west] {$S^{N,b}_\cdot$};
	\draw (10.5,7.5) circle (1pt);
	\draw (10.5,7.5) -- (12.9,7.5) node[anchor=west] {$S^{N,m}_\cdot$};
	
	\node(V_N_1) at(-0.125,12.2) {};
	\node(V_N_2) at (12.125,12.2) {};
	\draw (V_N_1) edge (V_N_2);
	\draw (0,12.1) -- (0,12.5);
	\draw (12,12.1) --(12,12.5);
	\node at (6,12.6) {$N$};
	
	\node(b1) at (8.8,9.5) {};
	\node(b2) at (8.8,10.0) {};
	\draw (8.8,9.5) -- (8.8,10);
	\draw (8.6,9.5) -- (8.9,9.5);
	\draw (8.6,10) -- (8.9,10);
	\node at (7.7,9.74) {$K^\prime L^\prime$};
	
	\draw (0,-0.2) -- (3,-0.2);
	\draw (0,-0.1) -- (-0,-0.4);
	\draw (3,-0.1) -- (3,-0.4);
	
	\draw (3,6) -- (-0.5,7.5);
	\draw (3,6) circle (2pt);
	\node at (-1.4,7.75) {$v_{N/KL,i}$};
	\draw (11.5,11.5) -- (12.5,12) node[anchor=west] {$v_{K^\prime L^\prime,j}$};
	\draw (11.5,11.5) circle (1pt);
	\node at (1.5,-0.8) {$N/(KL)$};
	\node [below=1cm, align=flush center,text width=8cm] at (6,0)
	{
		Figure $2$: $3$-field decomposition
	};
	\end{tikzpicture}
\end{center}
Using Gaussian comparison, we reduce the proof of \autoref{thm:convergence_in_law} to showing convergence in law of the centred maximum of an auxiliary field. Therefore, we need to have precise estimates on the variances and covariances, which is what we provide in the following. 
In order to use Slepian's lemma, we actually need, for each $v\in V_N$, equality of variances. This is usually achieved by adding suitable independent Gaussian random variables, which is done in the following lemma. In particular, the lemma states that one can choose the constants in such a way, that, asymptotically, they only depend on the ``fine scales'', i.e. they live on boxes $B_{K^\prime L^\prime,i}$,. In the rest of the paper, limits are taken in the order $N,K^\prime, L^\prime,K$ and then $L$, for which we write $(L,K,L^\prime,K^\prime,N)\Rightarrow\infty$.
\begin{lemma}\label{lemma:variance estimate}
	Let $\{\Phi_j\}_{1\leq j\leq (N/K^\prime L^\prime)^2}$ be a family of i.i.d. standard Gaussian random variables.
	For $v\in B_{K^\prime L^\prime,j}$, $j=1,\dotsc,(N/K^\prime L^\prime)^2$  and $v\equiv \bar{v} \mod K^\prime L^\prime$, i.e. $\bar{v}= v- v_{K^\prime L^\prime, j}$, there exists a collection of non-negative constants $\{a_{K^\prime L^\prime,\bar{v}}\}_{K^\prime L^\prime,\bar{v}}$, such that if we set
	\begin{align}\label{equation:def_3-field_approximand}
	S^N_v\coloneqq S^{N,c}_v+S^{N,b}_v+S^{N,m}_v+a_{K^\prime L^\prime,\bar{v}}\Phi_j,
	\end{align}
	then
	\begin{align}\label{equation:var_3-field_approximand}
	\limsup\limits_{(L,K,L^\prime,K^\prime)\Rightarrow \infty}\limsup\limits_{N\rightarrow \infty} \left|\mathrm{Var}\left(S^N_v\right)-\mathrm{Var}\left(\psi^N_v\right)-4\alpha\right|=0.
	\end{align}
	\begin{proof}
		Considering \autoref{lemma:cov_comp} ii., \eqref{equation:4.1}, \eqref{equation:S^NB} and \eqref{equation:S^N,m_v}, a simple computation shows that, for any $v\in V_{N,\delta}^{*}$,
		\begin{align}\label{equation:4.5}
		\mathrm{Var}\left(S^{N,c}_v\right)+\mathrm{Var}\left(S^{N,b}_v\right)+\mathrm{Var}\left(S^{N,m}_v\right)= \log N+O_N(1),
		\end{align}
		where the term $O_N\left(1\right)$ means that the constants are uniformly bounded in $N$. In particular, by \autoref{lemma:cov_comp} $iii.$ one has
		\begin{align}\label{equation:4.6}
		\left|\mathrm{Var}\left(S^{N,c}_v\right)+\mathrm{Var}\left(S^{N,b}_v\right)+\mathrm{Var}\left(S^{N,m}_v\right)-\mathrm{Var}\left(\psi^N_v\right)\right|\leq 4 \alpha.
		\end{align}
		By \eqref{equation:4.6}, there exist non-negative constants $\{a_{N,v}\}_{v\in B_{N/KL,i}}$, $1\leq i\leq (KL)^2$, such that
		\begin{align}\label{eq:properties_a_N,v}
		\mathrm{Var}\left(S^{N,c}_v+S^{N,b}_v+S^{N,m}_v\right)+a_{N,v}^2=\mathrm{Var}\left(\psi^N_v\right)+4\alpha.
		\end{align}
		Note that $\{a_{N,v}\}_{v\in B_{N/KL,i}}$ implicitly depend on $KL$ and by \eqref{equation:4.6}, one gets
		\begin{align}\label{equation:7.64}
		\max_{v \in V_{N,\delta}^{*}}a_{N,v}\leq \sqrt{8\alpha}.
		\end{align}
		For $v\in B^\delta_{N/KL,i}\cap V^\delta_N$, writing $v\equiv \bar{v}\mod K^\prime L^\prime$, where $\bar{v}=v- v_{N/KL,i}$, for $v\in B_{N/KL,i}$, and using \autoref{lemma:convergence_cov_neardiag_offdiag} $i.$ and \cite[(1.29)]{1712.09972},
		\begin{align}\label{equation:4.10}
		a_{N,v}^2&=4 \alpha+\mathrm{Var}\left(\psi^N_v\right) -\sigma^2(0)\mathrm{Var}\left(\phi^{KL}_{w_i}\right)-\sigma^2(1)\mathrm{Var}\left(\phi^{K^\prime L^\prime}_{\bar{v}}\right) -\mathcal{I}_{\sigma^2}\left(\frac{l+k}{n},\frac{n-l^\prime-k^\prime}{n}\right)\log(N)\nonumber\\ & =4\alpha+\sigma^2(0) f(v/N)-\sigma^2(0)f(w_i/(KL))-\sigma^2(1) f(\bar{v}/(K^\prime L^\prime))+\epsilon_{N,KL,K^\prime L^\prime}(v),
		\end{align}
		which is non-negative. By \autoref{lemma:convergence_cov_neardiag_offdiag} $i.$, $f$ is continuous and using $\|\frac{v}{N}-\frac{w_i}{KL}\|=\|\frac{v-v_{KL,i}}{N}\|\rightarrow 0$, as $(L,K,N)\Rightarrow \infty$, we have in the same limit, $|f(v/N)-f(w_i/(KL))|\rightarrow 0$. Moreover, by using \cite[(1.29)]{1712.09972}, \autoref{lemma:convergence_cov_neardiag_offdiag} $i.$ and \eqref{eq:properties_a_N,v} in the first line of \eqref{equation:4.10}, it follows that
		\begin{align}\label{equation:4.11}
			\limsup\limits_{(L,K,L^\prime,K^\prime)\Rightarrow \infty}\limsup\limits_{N\rightarrow \infty} \sup_{v\in V_{N,\delta}^{*}}\epsilon_{N,KL,K^\prime L^\prime}(v)=0.
		\end{align}
		Regarding \eqref{equation:4.10}, \eqref{equation:4.11}, and that $\mathrm{Var}\left[\phi^{K^\prime L^\prime}_v\right]\leq \log(K^\prime L^\prime)+\alpha$, for all $v\in V_N$, there exist non-negative $a_{K^\prime L^\prime,\bar{v}}$, such that
		\begin{align}\label{eq:a_N,bar(v)}
		a_{N,v}^2=a_{K^\prime L^\prime,\bar{v}}^2+\epsilon_{N,KL,K^\prime L^\prime}(v).
		\end{align}
		Using \cite[Lemma~B.3,Lemma~B.4,Lemma~B.5]{2016arXiv160600510B}, one obtains
		\begin{align}
		\limsup\limits_{(L,K,L^\prime,K^\prime)\Rightarrow \infty}\limsup\limits_{N\rightarrow \infty}\sup_{u,v\in V_{N,\delta}^{*}:\|u-v\|_\infty\leq L^\prime}\left|\mathrm{Var}\left(\phi^{K^\prime L^\prime}_u\right)-\mathrm{Var}\left(\phi^{K^\prime L^\prime}_v\right)\right|=0,
		\end{align}
		which, together with \eqref{equation:4.10} and \eqref{equation:4.11}, implies
		\begin{align}\label{equation:5.18}
		\left|a^2_{K^\prime L^\prime,\bar{u}}-a^2_{K^\prime L^\prime, \bar{v}}\right|\leq \sup_{v\in V_{N,\delta}^{*}} \epsilon_{N,KL,K^\prime L^\prime}(v),\quad \forall u,v\in V_{N,\delta}^{*}: \|u-v\|_\infty\leq L^\prime.
		\end{align}
		For $v\in B_{K^\prime L^\prime,j}$, $j=1,\dotsc,(N/K^\prime L^\prime)^2$  and $v\equiv \bar{v} \mod K^\prime L^\prime$, set
		\begin{align}
		S^N_v\coloneqq S^{N,c}_v+S^{N,b}_v+S^{N,m}_v+a_{K^\prime L^\prime,\bar{v}}\Phi_j.
		\end{align}
		By \eqref{eq:properties_a_N,v} and \eqref{eq:a_N,bar(v)}, it holds that, for $v\in V_{N,\delta}^{*}$,
		\begin{align}
		\limsup\limits_{(L,K,L^\prime,K^\prime)\Rightarrow \infty}\limsup\limits_{N\rightarrow \infty} \left|\mathrm{Var}\left(S^N_v\right)-\mathrm{Var}\left(\psi^N_v\right)-4\alpha\right|=0,
		\end{align}
		which concludes the proof of \autoref{lemma:variance estimate}.
	\end{proof}
\end{lemma}
The next goal is to show that it suffices to prove convergence of the centred maximum of the approximating process, $\{S^N_v\}_{v\in V_N}$, defined in \eqref{equation:def_3-field_approximand}. This can be done by using Gaussian comparison. The previous lemma, \autoref{lemma:variance estimate}, provides asymptotically equal variances, and the following lemma provides covariance estimates for $\{S^N_v\}_{v\in V_N}$. Crucially, for vertices close-by or at macroscopic distance, the covariances coincide asymptotically.
\begin{lemma}\label{lemma:cov_estimates_3field}
	There exists a non-negative sequence $\{\epsilon_{N,KL,K^\prime L^\prime}^{'}\}_{N,K,L,K^\prime,L^\prime \geq 0}$, and bounded constants $C_{\delta},C>0$, such that $\limsup\limits_{(L,K,L^\prime,K^\prime)\Rightarrow \infty}\limsup\limits_{N\rightarrow \infty} \epsilon_{N,KL,K^\prime L^\prime}^{'}=0$, and for all $u,v\in V_{N,\delta}^{*}:$
	\begin{enumerate}
		\item[i.] If $u,v\in B_{L^\prime,i},$ then $\left|\mathbbm{E}\left[\left(S^{N}_u-S^N_v\right)^2\right]- \mathbbm{E}\left[\left(\psi^N_u-\psi^N_v \right)^2\right]\right|\leq \epsilon_{N,KL,K^\prime L^\prime}^{'}.$
		\item[ii.] If $u\in B_{N/L,i},$ $v\in B_{N/L,j}$ and $i\neq j$, then $\left|\mathbbm{E}\left[S^N_u S^N_v \right]-\mathbbm{E}\left[\psi^N_u\psi^N_v\right]\right|\leq \epsilon_{N,KL,K^\prime L^\prime}^{'}.$
		\item[iii.] In all other cases, i.e. if $u,v\in B_{N/L,i}$ but $u\in B_{L^\prime,i^\prime}$ and $v\in B_{L^\prime,j^\prime}$, for some $i^\prime \neq j^\prime$, it holds that $\left|\mathbbm{E}\left[S^N_uS^N_v \right]-\mathbbm{E}\left[\psi^N_u\psi^N_v \right]\right|\leq C_{\delta}+ 40\alpha$.
	\end{enumerate}
\begin{proof}
	See \autoref{subsection:covariance_estimates}.
\end{proof}
\end{lemma}
 We use the L\'{e}vy-Prokhorov metric, $d(\cdot,\cdot)$, which is, for two probability measures on $\mathbb{R}$, $\mu$ and $\nu$, given by
\begin{align}
	d(\mu,\nu)\coloneqq \inf \{\delta>0:\, \mu(B)\leq \nu(B^\delta)+\delta, \text{ for all open sets B} \},
\end{align}
where $B^\delta= \{ y\in \mathbb{R}:\, |x-y| < \delta, \text{ for some } x\in B \}$. Moreover, let
\begin{align}
	\tilde{d}(\mu,\nu)=\inf\{\delta>0:\, \mu((x,\infty))\leq \nu((x-\delta,\infty))+\delta, \text{ for all }x\in \mathbb{R} \}.
\end{align}
and observe that if $\tilde{d}(\mu,\nu)=0$, then $\nu$ stochastically dominates $\mu$. For random variables $X,Y$ with laws $\mu_X, \mu_Y$, write $d(X,Y)$ instead of $d(\mu_X,\mu_Y)$, and likewise for $\tilde{d}(\cdot,\cdot)$. The following lemma reduces the proof of \autoref{thm:convergence_in_law} to show convergence in law of $S^*_N\coloneqq \max_{v \in V_N} S^N_v$.
\begin{lemma}\label{lemma:asymp_equiv_centred_laws1}
	Let $\{S^N_v\}_{v\in V_N}$ be the field defined in \eqref{equation:def_3-field_approximand}. Then,
	\begin{align}
	\limsup\limits_{(L,K,L^\prime,K^\prime)\Rightarrow \infty} \limsup\limits_{N\rightarrow \infty} d(\psi^*_N- m_N,S^*_N-m_N-4\alpha)=0.
	\end{align}
\end{lemma}
The proof of \autoref{lemma:asymp_equiv_centred_laws1} is based on the two following lemmas, whose proofs are postponed and given in \autoref{section:justification_approximation}. The overall idea is the following: Having asymptotically precise covariance estimates for vertices close-by or at macroscopic distance, and in order to use Slepian's lemma, we would like to add independent Gaussian fields living on those scales and control how the laws of the corresponding centred maxima change under such perturbations. It turns out, that this leads to a deterministic shift (see \autoref{lemma:1}). Having this control, we can then prove \autoref{lemma:asymp_equiv_centred_laws1}. First, introduce additional notation.
\par Fix a positive integer $r\in \mathbbm{N}$ and let $\mathcal{B}_r$ a partition of $V_{\lfloor N/r \rfloor r}$ into sub-boxes of side length $r$. Let $\mathcal{B}=\cup_{r\in \mathbbm{N},r\leq N} \mathcal{B}_r$ and $\{g_B\}_{B\in \mathcal{B}}$ be a collection of i.i.d. standard Gaussian random variables. For $v\in V_N$, denote by $B_r(v)\in \mathcal{B}_r$ the box containing $v$. For $s=(s_1,s_2)\in \mathbb{R}_+^2$, and two positive integers, $r_1,r_2$, define
\begin{align}\label{equation:6.17}
\tilde{\psi}^N_{v,s,r_1,r_2}= \psi^N_v+s_1 g_{B_{r_1}(v)}+s_2 g_{B_{N/r_2}(v)}.
\end{align}
Set $\tilde{\psi}^{*}_{N,s,r_1,r_2}=\max\limits_{v\in V_N} \tilde{\psi}^N_{v,s,r_1,r_2}$ and similarly, $\tilde{S}^N_{v,s,r_1,r_2}= S^N_v+s_1 g_{B_{r_1}(v)}+s_2 g_{B_{N/r_2}(v)},$ and $\tilde{S}^{*}_{N,s,r_1,r_2}=\max\limits_{v\in V_N} \tilde{S}^N_{v,s,r_1,r_2}.$
\begin{lemma}\label{lemma:1}
Let $\{S^N_v\}_{v\in V_N}$ be the field defined in \eqref{equation:def_3-field_approximand}. Then,
	\begin{align}
	\limsup\limits_{r_1,r_2 \rightarrow \infty}\limsup\limits_{N\rightarrow\infty} d\left(\psi^{*}_N-m_{N}, \tilde{\psi}^{*}_{N,s,r_1,r_2} - m_{N}- \|s\|_2^2 \right)=0,
	\end{align}
	and
	\begin{align}
	\limsup\limits_{r_1,r_2 \rightarrow \infty}\limsup\limits_{N\rightarrow\infty} d\left(S^{*}_N-m_{N}, \tilde{S}^{*}_{N,s,r_1,r_2} - m_{N}- \|s\|_2^2 \right)=0.
	\end{align}
\end{lemma}
\begin{lemma}\label{lemma:2}
Let $\{\bar{\psi}^N_v \}_{v \in V_N}$ be a centred Gaussian field such that, for all $u,v\in V_N,$ $N\in \mathbbm{N}$ and some arbitrary $\epsilon>0$, $| \mathrm{Var}\left(\psi^N_v \right)- \mathrm{Var}\left(\bar{\psi}^N_v \right)|\leq \epsilon$. Set $\bar{\psi}^*_N\coloneqq \max_{v \in V_N}\bar{\psi}^N_v.$
	Then there is a function, $l=l(\epsilon)$, with $l(\epsilon)\rightarrow 0$, as $\epsilon \rightarrow 0$, such that, if $\mathbbm{E}\left[\bar{\psi}^N_u\bar{\psi}^N_v \right]\leq \mathbbm{E}\left[\psi^N_u\psi^N_v \right]+\epsilon$,
	\begin{align}\label{equation:4.23}
	\limsup\limits_{N\rightarrow\infty} \tilde{d}\left(\psi^{*}_N-m_{N}, \bar{\psi}^*_N-m_{N} \right)\leq l(\epsilon).
	\end{align}
	Else if $\mathbbm{E}\left[\bar{\psi}^N_u\bar{\psi}^N_v \right] +\epsilon \geq \mathbbm{E}\left[\psi^N_u\psi^N_v \right]$, then
	\begin{align}\label{equation:4.25}
	\limsup\limits_{N\rightarrow\infty} \tilde{d}\left(\bar{\psi}^*_N-m_{N},\psi^{*}_N-m_{N} \right)\leq l(\epsilon).
	\end{align}
\end{lemma}
\autoref{lemma:1} and \autoref{lemma:2} allow to prove \autoref{lemma:asymp_equiv_centred_laws1}.
\begin{proof}[Proof of \autoref{lemma:asymp_equiv_centred_laws1}:]
	As in \eqref{equation:6.17}, we write
	\begin{align}
	\tilde{\psi}^N_{v,s,r_1,r_2}= \psi^N_v+s_1 g_{B_{r_1}(v)}+s_2 g_{B_{N/r_2}(v)},
	\end{align}
	and analogously,
	\begin{align}
	\tilde{S}^N_{v,s,r_1,r_2}=S^N_v+ s_1 g_{B_{r_1}(v)}+s_2 g_{B_{N/r_2}(v)},
	\end{align}
	where $s=(s_1,s_2)\in (0,\infty)^2$, $r_1,r_2\in \mathbbm{N}_{+}$ and $\{g_B\}_B$ being a collection of i.i.d. Gaussian random variables. Recall that
	$\mathcal{B}_r$ is a collection of sub-boxes of side length $r$ and that this forms a partition of $V_{\lfloor N/r \rfloor r}$.
	By \eqref{equation:5..8}, we only need to show that, for any $\delta>0$,
	\begin{align}\label{equation:7.20}
	\limsup\limits_{(L,K,L^\prime,K^\prime)\Rightarrow \infty} \limsup\limits_{N\rightarrow \infty} d\left(\max_{v \in V_{N,\delta}^{*}} \psi^N_v- m_N,\max_{v \in V_{N,\delta}^{*}} S^N_v-m_N-4\alpha\right)=0.
	\end{align}
	Thus, fix $\delta>0$ and let $\sigma_{*}^2=C_{\delta}+40 \alpha$ with the constant $C_{\delta}$ as in \autoref{lemma:cov_estimates_3field}, $\sigma_{lw}=(0,\sqrt{\sigma_{*}^2+4\alpha})$ and $\sigma_{up}=(\sigma_{*},0).$
	We consider the two Gaussian fields $\left\{\tilde{\psi}^N_{v,L^\prime,0,L,\sqrt{\sigma_{*}^2+4\alpha}} \right\}_{v\in V_{N,\delta}^{*}}$ and $\left\{\tilde{S}^N_{v,L^\prime,\sigma_{*},L,0} \right\}_{v\in V_{N,\delta}^{*}}$.
	By \autoref{lemma:cov_estimates_3field} $i.,ii.,iii.$ and \eqref{equation:var_3-field_approximand}, one gets for $u,v\in V_{N,\delta}^{*}$,
	\begin{align}
	\left|\mathrm{Var}\left(\tilde{\psi}^N_{v,L^\prime,0,L,\sqrt{\sigma_{*}^2+4\alpha}}\right)-\mathrm{Var}\left(\tilde{S}^N_{v,L^\prime,\sigma_{*},L,0} \right)\right| \leq \bar{\epsilon}_{N,KL,K^\prime L^\prime},\label{equation:6.75}
	\end{align}
	and
	\begin{align}
	\mathbbm{E}\left[\tilde{S}^N_{u,L^\prime,\sigma_{*}^2,L,0}\tilde{S}^N_{v,L^\prime,\sigma_{*},L,0} \right]\leq \mathbbm{E}\left[\tilde{\psi}^N_{u,L^\prime,0,L,\sqrt{\sigma_{*}^2+4\alpha}}\tilde{\psi}^N_{v,L^\prime,0,L,\sqrt{\sigma_{*}^2+4\alpha}} \right]+\bar{\epsilon}_{N,KL,K^\prime L^\prime},\label{equation:6.76}
	\end{align}
	where $\limsup\limits_{(L,K,L^\prime,K^\prime,N)\Rightarrow \infty} \bar{\epsilon}_{N,KL,K^\prime L^\prime}=0.$
	\autoref{lemma:1} implies both
	\begin{align}\label{equation:4.30}
	\limsup\limits_{(L,K,L^\prime,K^\prime)\Rightarrow \infty}\limsup\limits_{N\rightarrow \infty} d\left(\max_{v\in V^{*}_{N,\delta}}\tilde{\psi}^N_{v,L^\prime,0,L,\sqrt{\sigma_{*}^2+4\alpha}}-m_N-(\sigma_{*}^2+4\alpha),\max_{v \in V^{*}_{N,\delta}}\psi^N_v -m_N \right)=0,
	\end{align}
	and
	\begin{align}\label{equation:4.31}
	\limsup\limits_{(L,K,L^\prime,K^\prime)\Rightarrow \infty} \limsup\limits_{N\rightarrow \infty} d\left(\max_{v\in V_{N,\delta}^{*}}\tilde{S}^N_{v,L^\prime,\sigma_{*},L,0}-m_N-\sigma_{*}^2,\max_{v\in V^{*}_{N,\delta}}S^N_v-m_N\right)=0.
	\end{align}
	Having \eqref{equation:6.75} and \eqref{equation:6.76}, \autoref{lemma:2} implies that		 \begin{align}\label{equation:4.29}
	\limsup\limits_{(L,K,L^\prime,K^\prime)\Rightarrow \infty}\limsup\limits_{N\rightarrow \infty} \tilde{d}\left(\max_{v\in V^{*}_{N,\delta}} \tilde{\psi}^N_{v,L^\prime,0,L,\sqrt{\sigma_{*}^2+4\alpha}}-m_N,\max_{v\in V^{*}_{N,\delta}} \tilde{S}^N_{v,L^\prime,\sigma_{*},L,0}-m_N \right)=0.
	\end{align}
	A combination of \eqref{equation:4.30}, \eqref{equation:4.31} and \eqref{equation:4.29}, and using the triangle-inequality, gives stochastic domination in one direction, i.e.
	\begin{align}
	\limsup\limits_{(L,K,L^\prime,K^\prime)\Rightarrow \infty}\limsup\limits_{N\rightarrow \infty}\tilde{d}\left(\max_{v\in V^{*}_{N,\delta}} \psi^N_v-m_N ,\max_{v\in V^{*}_{N,\delta}} S^N_v-m_N-4\alpha \right)=0.
	\end{align}
	For the proof of the other direction of stochastic domination, consider instead the Gaussian fields $\left\{\tilde{\psi}^N_{v,L^\prime,\sqrt{\sigma_{*}^2+4\alpha},L,0} \right\}_{v\in V_N}$ and $\left\{\tilde{S}^N_{v,L^\prime,0,L,\sigma_{*}} \right\}$. This switches the roles in \eqref{equation:6.76} and the rest of the proof carries out analogously, which concludes the proof of \autoref{lemma:asymp_equiv_centred_laws1}.
\end{proof}
	\subsection{Convergence in law of the auxiliary field}
	A key step in the proof of \autoref{thm:convergence_in_law} is to establish a precise right-tail estimate for the maximum of the auxiliary process, which is provided in the following proposition. Before we state it, we introduce additional notation and make a preliminary observation. For $a,b\in [0,1]$, we write $\mathcal{I}_{\sigma^2}(a,b)= \int_a^b \sigma^2(x) \mathrm{d}x$. Let $\{S^N_v\}_{v\in V_N}$ be the field defined in \eqref{equation:def_3-field_approximand}, and set $S^{N,f}_v\coloneqq S^N_v-S^{N,c}_v$. Recall the tail-bounds from \cite[(2.6) in Theorem~2.1]{paper1}. By \autoref{lemma:cov_estimates_3field} and applying Slepian's lemma, these carry over to $\{S^{N}_v\}_{v\in V_N}$.
In particular, \cite[(2.6) in Theorem~2.1]{paper1} implies that there are constants $c_\alpha, C_\alpha>0$ such that for $z\geq 0$,
\begin{align}\label{equation:c.1}
c_\alpha e^{-2z}\leq \mathbb{P}\left(\max_{v \in V_N}S^{N}_v\geq m_N+z\right)\leq C_\alpha e^{-2z}.
\end{align}

\begin{lemma}\label{lemma:localization:fine_field}
	Let $\gamma\in (1/2,1)$ and fix $A>0$.
	Then, for $z\in \mathbb{R}$,
	\begin{align}\label{equation:5.40}
		\mathbb{P}\left(\exists v \in V_N:\,S^{N}_v\geq m_N+z,\, S^{N,c}_v-2\log (2) \sigma^2(0) (k+l)\notin [-A(k+l)^\gamma, A(k+l)^\gamma]\right)\leq C e^{-\frac{A^2(k+l)^{2\gamma-1}}{2\log(2)\sigma^2(0)}}.
\end{align}
\begin{proof}
	 Denote by $\nu_{v,N}^c(\cdot)$ the density such that for any interval $I\subset \mathbb{R}$,
	\begin{align}\label{equation:vcn}
	\int_I \nu_{v,N}^c(y) \mathrm{d}y=\mathbb{P}\left(S^{N,c}_v-2\log (2) \sigma^2(0) (k+l) \in I\right).
	\end{align}
	For any $v\in V^{\delta}_N$, using a union bound the probability in \eqref{equation:5.40} is bounded from above by
	\begin{align}
	2^{2n} &\int_{[-A(k+l)^{\gamma},A(k+l)^{\gamma}]^c}\nu_{v,N}^c(x)
	\mathbb{P}\left(S^{N,f}_v\geq 2\log(2)\mathcal{I}_{\sigma^2}\left(\frac{k+l}{n},1\right)n-\log(n)/4+z-x\right)\mathrm{d}x \nonumber\\
	&= 2^{2n}\int_{[-A(k+l)^{\gamma},A(k+l)^{\gamma}]^c}\frac{\exp\left[-2\log(2)\sigma_1^2 (k+l)-2x-\frac{x^2}{2\log(2)\sigma^2(0) (k+l)}\right]}{\sqrt{2\pi \log(2)\sigma^2(0) (k+l)}}\nonumber\\
	&\quad\times\exp\left[-2\log(2)\mathcal{I}_{\sigma^2}\left(\frac{k+l}{n},1\right)n-2\left(z-x-\frac{\log(n)}{4}\right)-\frac{\left(z-x-\frac{\log(n)}{4}\right)^2}{2\log(2)\mathcal{I}_{\sigma^2}\left(\frac{k+l}{n},1\right)n}\right]\nonumber\\&\quad\times\frac{\sqrt{2\log(2)\mathcal{I}_{\sigma^2}\left(\frac{k+l}{n},1\right)n}}{2\log(2)\mathcal{I}_{\sigma^2}\left(\frac{k+l}{n},1\right)n-\frac{\log(n)}{4}+z-x}\mathrm{d}x.
	\end{align}
	The latter integral decays with $e^{-A^2(k+l)^{2\gamma-1}/(2\log(2)\sigma^2(0)}$, which allows to conclude the proof.
\end{proof}
\end{lemma}
Write $\bar{k}=k+l $ and $M_n(k,t)=2\log(2)\mathcal{I}_{\sigma^2}\left(\frac{k}{n},\frac{t}{n}\right)n-\frac{((t)\wedge (n-\bar{l})) \log(n)}{4(n-\bar{l})},$ for $t\in [k,n]$. Note that $m_N=M_n(0,n)$, for $n=\log_2 N$.
\begin{prop}\label{proposition:sharp_right_tail_estimate}
	Let $\{S^N_v\}_{v\in V_N}$ be the field defined in \eqref{equation:def_3-field_approximand}, and set $S^{N,f}_v\coloneqq S^N_v-S^{N,c}_v$. Then, there are constants $C_\alpha,c_\alpha>0$, depending only on $\alpha$, and constants $c_\alpha\leq \beta_{K^\prime,L^\prime}^{*}\leq C_\alpha$, such that
	\begin{align}\label{equation:prop_sharp_right_tail}
		\lim\limits_{z\rightarrow \infty}&\limsup\limits_{(L^\prime,K^\prime,N)\Rightarrow \infty} |e^{2\log(2)(\bar{k})(1-\sigma^2(0))}e^{-2\bar{k}^{\gamma}}e^{2z}
		\, \mathbbm{P}\left(\max_{v\in B_{N/KL,i}} S^{N,f}_v\geq M_n(\bar{k},n)-\bar{k}^{\gamma}+z \right)-\beta_{K^\prime,L^\prime}^{*}|=0.
	\end{align}
	In particular, $\{\beta_{K^{\prime},L^{\prime}}^{*}\}_{K^\prime, L^\prime\geq 0}$ depends on the variance parameters only through $\sigma(1)$.
\end{prop}
Note that, unlike previous tail estimates obtained in \cite[Theorem~2.1]{paper1}, the estimates in \autoref{proposition:sharp_right_tail_estimate} are precise estimates for the maximum far in front of the expected maximum. Nevertheless, the proofs are technically similar, i.e. both rely on a truncated second moment computation. The proof of \autoref{proposition:sharp_right_tail_estimate} is postponed to \autoref{subsection:proof_right_tail}, as we first want to use it to finish the proof of \autoref{thm:convergence_in_law}.
\autoref{proposition:sharp_right_tail_estimate} allows to construct the limiting law of $(\max\limits_{v \in V_N}S^N_v-m_{ N})_{N\geq 0}$, which is the contents of the following:
Partition $[0,1]^2$ into $R=(KL)^2$ disjoint, equal-sized boxes. Let $\{\beta_{K^\prime,L^\prime}^{*}\}_{K^\prime,L^\prime\geq 0}$ be given by \autoref{proposition:sharp_right_tail_estimate}. Then, there is a function, $\rho:\mathbbm{R}\rightarrow \mathbbm{R}$, that grows to infinity arbitrarily slowly, and such that
\begin{align}\label{equation:4.35}
	&\lim\limits_{z^\prime\rightarrow \infty}\limsup\limits_{(L^\prime,K^\prime,N)\Rightarrow \infty}\sup\limits_{z^\prime \leq z\leq \rho(K^\prime L^\prime)}\left| e^{2z}e^{-2\bar{k}^\gamma}e^{2\log (2)\bar{k}(1-\sigma^2(0))} \mathbbm{P}\left(\max_{v\in B_{N/KL,i}}S^{N,f}_v\geq M_n(\bar{k},n)+z-\bar{k}^{\gamma}\right)-\beta_{K^\prime,L^\prime}^{*}\right|=0.
\end{align}
Let $\{\varrho_{R,i}\}_{1\leq i \leq R}$ be independent Bernoulli random variables with
\begin{align}\label{equation:Bernoulli_rv}
	\mathbbm{P}\left(\varrho_{R,i}=1\right)=\beta_{K^\prime,L^\prime}^{*}e^{2\bar{k}^\gamma}2^{2\log(2)\bar{k}(\sigma^2(0)-1)}.
\end{align}
In addition, consider independent random variables $\{Y_{R,i}\}_{1\leq i\leq R}$ satisfying
\begin{align}\label{equation:4.37}
	\mathbbm{P}\left(Y_{R,i}\geq x\right)=e^{-2x}e^{-2\bar{k}^\gamma},\quad x\geq -\bar{k}^\gamma,
\end{align}
and let $\{Z_{R,i}\}_{1\leq i \leq R}$ be an independent Gaussian field with the same distribution as $\{S^{N,c}_v\}_{v\in V_N}$. Set
\begin{align}\label{equation:G_R,i}
G_{R,i}\coloneqq \varrho_{R,i}(Y_{R,i}+2\log(KL)(1-\sigma^2(0)))+ (Z_{R,i}-2\log(KL)),
\end{align} 
and
\begin{align}
	G^{*}_{K,L,K^\prime,L^\prime}\coloneqq \max_{\substack{1\leq i\leq R\\\varrho_{R,i}=1}} G_{R,i}.
\end{align}
Let $\bar{\mu}_{K,L,K^\prime,L^\prime}$ be the distribution of $G^{*}_{K,L,K^\prime,L^\prime}$. Note that it is independent of $N$, which is essential for the proof of convergence in law. The following theorem reduces the proof of convergence in law of $\max_{v\in V_N}S^N_v-m_{ N}$, to proving convergence of the sequence $\{\bar{\mu}_{K,L,K^\prime,L^\prime}\}_{K,L,K^\prime, L^\prime}$.
\begin{thm}\label{theorem:thm6.1}
	Let $\mu_N=\text{law of }\left(\max\limits_{v \in V_N}S^N_v-m_{ N} \right)$. Then,
	\begin{align}\label{equation:5..48}
		\limsup\limits_{(L,K,L^\prime,K^\prime)\Rightarrow \infty}\limsup\limits_{N\rightarrow \infty} d\left(\mu_N,\bar{\mu}_{K,L,K^\prime,L^\prime} \right)=0.
	\end{align}
	In particular, there exists $\mu_\infty$ such that $\lim\limits_{N\rightarrow \infty} d(\mu_N, \mu_\infty)=0$.
	\begin{proof}
		Denote by $\tau=\text{arg}\max\limits_{v\in V_N}S^N_v$ the (unique) particle achieving the maximal value. The correlation estimates in \autoref{lemma:cov_estimates_3field}, together with Slepian's lemma and \eqref{equation:expected_max}, imply that $\max_{v\in V_N}S^N_v- m_{N}$, as a sequence in $n$, is tight. Using this fact and the localization of $\{S^{N,c}_v\}_{v\in V_N}$ in \autoref{lemma:localization:fine_field}, one obtains
		\begin{align}\label{equation:varrho=1}
			\limsup\limits_{(L,K,L^\prime,K^\prime)\Rightarrow \infty}\limsup\limits_{N\rightarrow \infty} \mathbbm{P}\left(S^{N,f}_\tau\geq M_n(\bar{k},n)-\bar{k}^\gamma\right)=1.
		\end{align}
		Thus, assume that $S^{N,f}_\tau\geq M_n(\bar{k},n)-\bar{k}^\gamma$ holds. To exclude that $\max_{v\in V_N} S^{N,f}_v$ is too large, consider the event $\mathcal{E}=\cup_{i=1}^{R}\{\max_{v\in B_{N/KL,i}} S^{N,f}_v\geq M_n(\bar{k},n) +KL +\bar{k}^\gamma \}$. By a union and a Gaussian tail bound,
		\begin{align}
			\mathbb{P}\left(\mathcal{E}\right)&\leq 2^{2\bar{k}}\mathbb{P}\left(\max_{v\in B_{N/KL,i}} S^{N,f}_v\geq M_n(\bar{k},n) +KL +\bar{k}^\gamma\right)\leq 2^{2n}  \mathbb{P}\left(S^{N,f}_v\geq M_n(\bar{k},n) +KL +\bar{k}^\gamma \right) \nonumber\\&\leq C \exp\left[2\log(2) \sigma^2(0)(k+l)-2KL-2\bar{k}^\gamma\right].
		\end{align}
	 	Thus, one obtains
		\begin{align}
			\limsup\limits_{(L,K,L^\prime,K^\prime)\Rightarrow \infty}\limsup\limits_{N\rightarrow \infty}\mathbbm{P}\left(\mathcal{E}\right)=0.
		\end{align}
		Analogously, a union bound on the event $\mathcal{E}^\prime=\cup_{i=1}^R \{Y_{R,i}\geq KL+\bar{k}^\gamma \}$ yields
		\begin{align}
			\limsup\limits_{(L,K,L^\prime,K^\prime)\Rightarrow \infty}\limsup\limits_{N\rightarrow \infty} \mathbbm{P}\left(\mathcal{E}^\prime\right)=0.
		\end{align}
		As a next step, we couple the centred fine field, $M^{f}_{n,i}\coloneqq \max_{v\in B_{N/kl,i}}S^{N,f}_v-M_n(\bar{k},n)$, to the approximating process $G_{R,i}$ defined in \eqref{equation:G_R,i}.
		By \autoref{proposition:sharp_right_tail_estimate}, there are $\epsilon_{N,KL,K^\prime L^\prime}^{*}>0,$ satisfying $\limsup\limits_{(L,K,L^\prime,K^\prime,N)\Rightarrow \infty} \epsilon_{N,KL,K^\prime L^\prime}^{*}=0,$ and such that, for some $|\epsilon^\diamond|\leq \epsilon_{N,KL,K^\prime L^\prime}^{*}/4$,
		\begin{align}\label{equation:6.128}
			\mathbbm{P}\left(-A\bar{k}^\gamma+\epsilon^\diamond\leq M^f_{n,i}\leq KL+\bar{k}^\gamma\right)=\mathbbm{P}\left(\varrho_{R,i}=1,\, Y_{R,i}\leq KL+\bar{k}^\gamma\right)
		\end{align}
		and such that, for all $t$ with $-\bar{k}^\gamma-1\leq t\leq KL+\bar{k}^\gamma$,
		\begin{align}\label{equation:6.129}
			\mathbbm{P}\left(\varrho_{R,i}=1,Y_{R,i}\leq t-\epsilon_{N,KL,K^\prime L^\prime}^{*}/2\right)&\leq \mathbbm{P}\left(-\bar{k}^\gamma+\epsilon^\diamond\leq M^f_{n,i}\leq t\right)\nonumber\\&\leq \mathbbm{P}\left(\varrho_{R,i}=1, Y_{R,i}\leq t+ \epsilon^{*}_{N,KL,K^\prime L^\prime}/2\right).
		\end{align}
		Thus, there is a coupling of $\{M^f_{n,i}:\,1\leq i\leq R\}$ and $\{(\varrho_{R,i}, Y_{R,i}):\,1\leq i\leq R \}$, such that, on the event $\left(\mathcal{E}\cup \mathcal{E}^\prime\right)^c,$
		\begin{align}
				& \varrho_{R,i}=1,\, |Y_{R,i}-M^f_{n,i}|\leq \epsilon_{N,KL,K^\prime L^\prime}^{*}, \quad\text{ if } M^f_{n,i}\geq \epsilon_{N,KL,K^\prime L^\prime}^{*} \label{equation:coupling1}\\
				& |Y_{R,i}-M^f_{n,i}|\leq \epsilon_{N,KL,K^\prime L^\prime}^{*}, \qquad \qquad\,\,\,\, \text{ if } \varrho_{R,i}=1. \label{equation:coupling2}
		\end{align}
		Note that, for each $N$, one possibly needs a different coupling, since $M_{n,i}^f$ depends on $N$, whereas $(\varrho_{R,i}, Y_{R,i})$ does not.
		A short argument for the existence of such couplings is as follows: In the event $\mathcal{E}^c\cap \mathcal{E}^{\prime,c}$, \eqref{equation:6.128} becomes
		\begin{align}
			\mathbb{P}\left(-\bar{k}^\gamma+\epsilon^\diamond \leq M^f_{n,i}\right)=\mathbb{P}\left(\varrho_{R,i}=1\right).
		\end{align}
		By \eqref{equation:6.129} and since the random variables have distributions that are absolutely continuous with respect to the Lebesgue measure, there is an increasing function, $g:\mathbb{R}\rightarrow \mathbb{R}$, with $g(t)\in [t-\epsilon^{*}/2,t+\epsilon^{*}/2]$, for $-\bar{k}^\gamma-1\leq t\leq KL+\bar{k}^\gamma$, and such that
		\begin{align}
			\mathbbm{P}\left(\varrho_{R,i}=1,Y_{R,i}\leq g(t)\right)= \mathbbm{P}\left(-\bar{k}^\gamma+\epsilon^\diamond\leq M^f_{n,i}\leq t\right).
		\end{align}
		Let $-\bar{k}^\gamma-1=t_0<\dotsc<t_D=KL+\bar{k}^\gamma$ be an arbitrary partition. Define sets
		\begin{align}
			A_j&\coloneqq \{\omega: \varrho_{R,i}(\omega)=1, Y_{R,i}(\omega)\in[g(t_j),g(t_{j+1})) \},\\
			B_j &\coloneqq \{\omega: \epsilon^\diamond\leq M^f_{n,i}(\omega)\in[t_j,t_{j+1}) \}.
		\end{align}
		In particular, for any $0\leq j<D$, $\mathbb{P}\left(A_j\right)=\mathbb{P}\left(B_j\right)$.
		Define random variables $(\varrho_{R,i}^{\prime}, Y_{R,i}^{\prime})$, i.e for $\omega \in B_j\cap \left(\mathcal{E}\cup \mathcal{E}^\prime\right)^c$, set $Y_{R,i}^{\prime}(\omega)=g(M^f_{n,i}(\omega))$ and such that, for all $\omega \in \left(\mathcal{E}\cup \mathcal{E}^\prime\right)^c \cap \left(\cup_{j}B_j\right)$, $\varrho_{R,i}^{\prime}(\omega)=1$. For $\omega\in \mathcal{E}\cup \mathcal{E}^\prime,$ set $\varrho_{R,i}^{\prime}(\omega)=\varrho_{R,i}(\omega)$ and $Y_{R,i}^{\prime}(\omega)=Y_{R,i}(\omega)$. Then $(\varrho_{R,i}^{\prime},Y_{R,i}^{\prime})\overset{d}{=}(\varrho_{R,i},Y_{R,i})$, and $(\varrho_{R,i}^{\prime},Y_{R,i}^{\prime})$ additionally satisfies both \eqref{equation:coupling1} and \eqref{equation:coupling2}.
		Concerning the coarse field, one can couple such that $S^{N,c}_v= Z_{R,i}$, for $v\in B_{N/KL,i},$ $1\leq i \leq R$, simply as they have the same law.
		Thus, there are couplings, such that, outside an event of vanishing probability as $(L,K,L^\prime,K^\prime,N)\Rightarrow \infty$,
		\begin{align}\label{equation:6.94}
			\max_{v\in V_N}\left(S^N_v-m_N\right)- G^{*}_{K,L,K^\prime,L^\prime}\leq 2\epsilon_{N,KL,K^\prime L^\prime}^{*}.
		\end{align}
		Let $\tau^\prime=\text{arg} \max_{1\leq i \leq R}G_{R,i}$. In the following, we exclude the case that the maximum of $G_{R,i}$ is achieved at $i=\tau^\prime$ and when at the same time, $\varrho_{R,\tau^\prime}=0$. The first order of the maximum of $\{S^{N,c}_v\}_{v\in V_N}$ is given by $2\log(KL)\sigma(0)$ (see \cite{MR1880237}), which is of order $O(\log(KL))$ less than subtracted in \eqref{equation:G_R,i}, and so, $Z_{R,i}-2\log(KL)\rightarrow - \infty$, as $(L,K)\Rightarrow \infty$. Having this in mind, considering \eqref{equation:6.94} and since $(\max\limits_{v\in V_N}S^N_v-N)_{N\geq 0}$ is tight, it follows that
		\begin{align}\label{equation:6.95}
			\limsup\limits_{(L,K,L^\prime,K^\prime,N)\Rightarrow \infty}\mathbbm{P}\left(\varrho_{R,\tau^\prime}=1\right)=1.
		\end{align}
		By \eqref{equation:coupling1}, \eqref{equation:coupling2} and \eqref{equation:6.95}, there are couplings, such that outside a set with probability tending to $0$, as $(L,K,L^\prime,K^\prime,N)\Rightarrow \infty$, it holds that
		\begin{align}\label{equation:5..63}
			\left|\max_{v\in V_N}S^N_v-m_N-G_{K,L,K^\prime,L^\prime}^{*} \right|\leq 2\epsilon_{N,KL,K^\prime L^\prime}^{*},
		\end{align}
		which proves \eqref{equation:5..48}. Moreover, \eqref{equation:5..63} implies that $\mu_N$ is a Cauchy sequence and that there is $\mu_\infty$, such that $\lim\limits_{N\rightarrow \infty} d(\mu_N,\mu_\infty)=0,$ which concludes the proof of \autoref{theorem:thm6.1}.
	\end{proof}
\end{thm}
\begin{proof}[Proof of \autoref{thm:convergence_in_law}:]
Recall that $G^{*}_{K,L,K^\prime,L^\prime}$ is a random variable with law $\bar{\mu}_{K,L,K^\prime,L^\prime}$.  The goal is to construct a sequence of random variables, $\{D_{K,L}\}_{K,L\geq 0}$, which are measurable with respect to $\mathcal{F}^c\coloneqq \sigma\left(\{Z_{R,i}\}\right)_{i=1}^{R}$, with $R\coloneqq (KL)^2$, and so that, for any $x\in \mathbbm{R},$
	\begin{align}\label{equation:6.99}
		\limsup\limits_{(L,K,L^\prime,K^\prime)\Rightarrow \infty} \frac{\bar{\mu}_{K,L,K^\prime,L^\prime}((-\infty,x])}{\mathbbm{E}\left[\exp(-\beta^{*}_{K^\prime,L^\prime} D_{K,L}e^{-2 x})\right]}=	\liminf\limits_{(L,K,L^\prime,K^\prime)\Rightarrow \infty} \frac{\bar{\mu}_{K,L,K^\prime,L^\prime}((-\infty,x])}{\mathbbm{E}\left[\exp(-\beta^{*}_{K^\prime,L^\prime} D_{K,L}e^{-2 x})\right]}=1.
	\end{align}
	Regarding \eqref{equation:6.95}, assume $\varrho_{R,\tau^\prime}=1$. Moreover, let 
	\begin{align}\label{equation:4.59}
		S_{R,i}\coloneqq 2\log(KL)(1+\sigma^2(0))-Z_{R,i}, \quad \text{for }i=1,\dotsc,R.
	\end{align}
	For $x\in \mathbbm{R}$, it holds 
	\begin{align}\label{equation:6.101}
		\bar{\mu}_{K,L,K^\prime,L^\prime}((-\infty,x])&=\mathbbm{P}\left(G^{*}_{K,L,K^\prime,L^\prime}\leq x\right)\\
		&=\mathbbm{E}\left[\prod_{i=1}^{R}\left(1-\mathbbm{P}\left(\varrho_{R,i}\left( Y_{R,i}+2\log(KL)(1-\sigma^2(0))\right)>2\log(KL)-Z_{R,i}+x\right)\right)|\mathcal{F}^c \right].\nonumber
	\end{align}
	A union bound on $\mathcal{D}^{c}=\{\min_{1\leq i\leq R}2\log(KL)-Z_{R,i}\geq 0 \}^{c}$, shows that $\limsup\limits_{KL\rightarrow \infty} \mathbbm{P}\left(\mathcal{D}\right)=1.$ Thus, on the event $\mathcal{D}$ and using \eqref{equation:Bernoulli_rv}, \eqref{equation:4.37}, \eqref{equation:4.59}, one deduces
	\begin{align}\label{equation:6.121}
		\mathbbm{P}\left(\varrho_{R,i} Y_{R,i}>2\log(KL)\sigma^2(0)-Z_{R,i}+x|\mathcal{F}^c\right)=\beta_{K^\prime,L^\prime}^{*}e^{-2(S_{R,i}+x)}.
	\end{align}
	Note that \eqref{equation:6.121} tends to $0$, as $KL\rightarrow\infty$. Using the fact that $e^{-\frac{x}{1-x}}\leq 1-x\leq e^{-x}$, for $x<1$, and inserting for $x$ the probability in \eqref{equation:6.121}, it follows that there is a non-negative sequence $\{\epsilon_{K,L}\}_{K,L\geq 0}$, satisfying $\limsup\limits_{KL\rightarrow \infty}\epsilon_{K,L}=0$, and such that
	\begin{align}\label{equation:6.103}
		\exp\left(-(1+\epsilon_{K,L})\beta_{K^\prime,L^\prime}^{*}e^{-2(S_{R,i}+x)}\right)&\leq \mathbbm{P}\left(\varrho_{R,i}Y_{R,i}\leq 2\log(KL)\sigma^2(0) -Z_{R,i}+x|\mathcal{F}^c\right)\nonumber\\ &\leq
		\exp\left(-(1-\epsilon_{K,L})\beta_{K^\prime,L^\prime}^{*}e^{-2(S_{R,i}+x)} \right).
	\end{align}
	Plugging \eqref{equation:6.103} into \eqref{equation:6.101} yields \eqref{equation:6.99}. Combining \eqref{equation:6.99} with \autoref{theorem:thm6.1} implies that there is a constant $\beta^*$, such that
	\begin{align}\label{equation:conv_betaKL}
		\limsup\limits_{(K^\prime,L^\prime)\Rightarrow \infty}|\beta_{K^\prime,L^\prime}^{*}-\beta^{*}|=0.
	\end{align}
	Set
	\begin{align}\label{equation:4.64}
	D_{K,L}=\sum_{i=1}^{R}e^{-2S_{R,i}}.
	\end{align}
	Combining \eqref{equation:conv_betaKL} with \eqref{equation:6.99}, it follows that
	\begin{align}\label{equation:4.62}
			\limsup\limits_{(L,K,L^\prime,K^\prime)\Rightarrow \infty} \frac{\bar{\mu}_{K,L,K^\prime,L^\prime}((-\infty,x])}{\mathbbm{E}\left[\exp(-\beta^{*} D_{K,L}e^{-2 x})\right]}=	\liminf\limits_{(L,K,L^\prime,K^\prime)\Rightarrow \infty} \frac{\bar{\mu}_{K,L,K^\prime,L^\prime}((-\infty,x])}{\mathbbm{E}\left[\exp(-\beta^{*} D_{K,L}e^{-2 x})\right]}=1.
	\end{align}
	\autoref{theorem:thm6.1} and \eqref{equation:4.62} imply that $D_{K,L}$ converges weakly to a random variable $D$, as $(L,K)\Rightarrow \infty$. \eqref{equation:4.64} shows that $D_{K,L}$ depends solely on  $(KL)^2=R$.
	Moreover, as $\bar{\mu}_{K,L,K^\prime,L^\prime}$ is a tight sequence of laws, it follows that almost surely, $D>0$. This concludes the proof of \autoref{thm:convergence_in_law}.
\end{proof}
Note that the random variables $\{D_{K,L}\}_{K,L\geq 0}$, defined in \eqref{equation:4.64}, are the analogue of the ``McKean martingale'' in variable-speed BBM (see \cite[(1.14)]{MR3351476}).
	\appendix
	\section{Gaussian comparison and covariance estimates}\label{section:appendix}
	\begin{thm}[Slepian's Lemma, {\cite[Theorem~3.11]{MR2814399}}]\label{thm:slepian}
	Let $T= \{1,\dotsc,n \}$ and $X,Y$ be two centred Gaussian vectors. Assume further that it exist two subsets $A,B \subset T \times T$, so that
	\begin{align}
		& \mathbbm{E}[X_i X_j] \leq \mathbbm{E}[Y_i Y_j], \,\quad (i,j) \in A \\
		& \mathbbm{E}[X_i X_j] \geq \mathbbm{E}[Y_i Y_j], \, \quad(i,j) \in B \\
		& \mathbbm{E}[X_i X_j] = \mathbbm{E}[Y_i Y_j], \, \quad (i,j)\notin A \cup B.
	\end{align}
	Suppose $f: \mathbbm{R}^n \rightarrow \mathbbm{R}$ is smooth, with at most exponential growth at infinity of f and its first and second derivatives , and 
	\begin{align}
		& \partial_{ij} f \geq 0, \, \quad(i,j) \in A \\
		& \partial_{ij} f \leq 0, \, \quad(i,j) \in B.
	\end{align}
	Then, 
	\begin{align}
		\mathbbm{E}[f(X)]\leq \mathbbm{E}[f(Y)].
	\end{align}
\end{thm}
\begin{rem}
	We use Slepian's Lemma in a very particular setting: Assume that $\mathbbm{E}\left[X_i^2\right]=\mathbbm{E}\left[Y_i^2\right]$ and $\mathbbm{E}\left[X_iX_j\right]\geq \mathbbm{E}\left[Y_iY_j\right]$, for all $i,j \in T.$ Then, for any $x \in \mathbbm{R}$,
	\begin{align}
		\mathbbm{P}\left(\max_{i \in T}X_i >x\right)\leq \mathbbm{P}\left(\max_{i \in T}Y_i>x\right),
	\end{align}
	and
	\begin{align}
		\mathbbm{E}\left[\max_{i \in T}X_i\right]\leq \mathbbm{E}\left[\max_{i\in T}Y_i\right].
	\end{align} 
\end{rem}
\begin{thm}[Sudakov-Fernique, {\cite[Sudakov-Fernique]{MR0413238}}]\label{thm:sudakov}
	Let $I$ be an arbitrary set with cardinality $\left|I\right|=n,$ $\{X_i\}_{i \in I}, \{Y_i\}_{i \in I}$ be two centred Gaussian vectors. Define $\gamma_{ij}^{X} \coloneqq \mathbbm{E}[(X_i-X_j)^2]$, $\gamma_{ij}^{Y} \coloneqq \mathbbm{E}[(Y_i-Y_j)^2]$. Let $\gamma \coloneqq \max_{i,j} |\gamma_{ij}^X - \gamma_{ij}^Y|.$ Then,
	\begin{align}
	& \left|\mathbbm{E}[X^*] - \mathbbm{E}[Y^*]\right| \leq \sqrt{\gamma \log(n)}. \\
	& \text{If } \gamma_{ij}^X \leq \gamma_{ij}^Y \text{ for all } i,j \text{ then } \mathbbm{E}[X^*]\leq \mathbbm{E}[Y^*].
	\end{align}
\end{thm}
In particular, if $\{X_i\}_{i \in I}, \{Y_i\}_{i \in I}$ are independent centred Gaussian fields, then
\begin{equation}\label{equation:consequence_sudakov}
\mathbbm{E}\left[ \max_{i \in I} (X_i + Y_i)\right] \geq \mathbbm{E}\left[\max_{i \in I} X_i \right].
\end{equation}
	\subsection{Covariance estimates}\label{subsection:covariance_estimates}
	\begin{proof}[Proof of \autoref{lemma:cov_comp}]
	The proof of statement $i.$ is a simple adaptation of the proof of the analogue statement for finitely many scales \cite[Lemma~3.3]{paper1}.
	The third statement follows by a combination of $i.$ with $ii.$. In the following, we prove statement $ii.$.
	Let $u,v\in V_N^\delta$ and denote by $b_N(u,v)=1-\frac{\log_{+} \|u-v\|_2}{\log N}$ the ``branching scale''. By the Gibbs-Markov property of the DGFF, increments $\nabla \phi^N_u(s), \, \nabla \phi^N_v(s)$ beyond $b_N(u,v)$ are independent.
	By \eqref{equation:def_scale_dgff}, one has
	\begin{align}\label{equation:a.13}
	\mathbb{E}\left[\psi^N_{u}\psi^N_{v}\right]=\int_{0}^{1}\int_0^1 \sigma(s)\sigma(t)\mathbb{E}\left[\nabla\phi^N_u(s)\nabla\phi^N_v(t)\right]\mathrm{d}s\mathrm{d}t.
	\end{align}
	To compute the discrete gradients, it suffices to consider $\mathbb{E}\left[\phi^N_u(s)\phi^N_v(t)\right]$, for $s,t\in [0,1]$.
	Let $S$ be a simple random walk with hitting times $\tau_{\partial A}=\inf \{r\geq 0:\, S_r \in \partial A\}$, for $A\subset \mathbb{Z}^2$. Let $c:\partial
	[-\frac{1}{2},\frac{1}{2}]^2\rightarrow \mathbb{R}^2$ be the continuous function, encoding the relative position on the boundary, such that, for $x\in (0,1)$, $u\in \mathbb{Z}^2$ and $z\in \partial [xN+u]_{\lambda_{i}}$, $z=xN+u+c(z)N^{1-\lambda_{i}}.$  In particular, the function $c$ is in both components absolutely bounded away from zero by $1/2$ and from above by $\sqrt{1/2}$.
	For $0\leq s<t\leq 1$, we have
	\begin{align}\label{equation:a..14}
	\mathbb{E}\left[\phi^N_u(s)\phi^N_v(t)\right]&=\sum_{\substack{x\in \partial[u]_s\\y\in \partial [v]_t}}\mathbb{P}_u\left(S_{\tau_{\partial [u]_s}}=x\right)\mathbb{P}_v\left(S_{\tau_{\partial[v]_t}}=y\right)\mathbb{E}\left[\phi^N_{u+c(x)N^{1-s}}\phi^N_{v+c(y)N^{1-t}}\right]\nonumber\\
	&=\sum_{\substack{x\in \partial[u]_s\\y\in \partial [v]_t}}\mathbb{P}_u\left(S_{\tau_{\partial [u]_s}}=x\right)\mathbb{P}_v\left(S_{\tau_{\partial [v]_t}}=y\right)\left[-\mathfrak{a}\left(u-v +N^{1-s}(c(x)-c(y)N^{s-t})\right) \right.\nonumber
	\\& \left. \quad+\sum_{z\in \partial V_N}\mathbb{P}_{u+c(x)N^{1-s}}\left(S_{\tau_{\partial V_N}}=z\right) \mathfrak{a}(z-v-c(y)N^{1-t}) \right],
	\end{align}
	where $\mathfrak{a}$ denotes the Potential kernel, which satisfies the asymptotics
	\begin{align}\label{equation:possion_kernel}
		\mathfrak{a}(x)=\log \|x\|_2+c_0+O(\|x\|_2^{-2}),
	\end{align}
	as $\|x\|_2\rightarrow \infty$.
	Using this asymptotics and the approximate uniformity of the harmonic measure away from the boundary \cite[Lemma~B.5]{2016arXiv160600510B}, the second sum in is about $\log(N)+O(1)$, and the first is about $\log(N^{1-s})+O(1)$ if $s<t$ and if $\|u-v\|_2\ll N^{1-s}$, i.e. $b_N(u,v)\leq s-\epsilon_N$ with $\epsilon_N=4/\log N$.
	In particular,
	\begin{align}\label{equation:a..16}
	\int_{s+\epsilon_N}^{1}\mathbb{E}\left[\phi^N_u (s) \nabla \phi^N_v(t)\right] \mathrm{d}t=0,
	\end{align}
	and, if $\|u-v\|_2< N^{1-t}$,
	\begin{align}\label{equation:a..17}
	\int_{0}^{t-\epsilon_N}\mathbb{E}\left[\nabla \phi^N_u(s)\phi^N_v(t)\right]\mathrm{d}s=(t-\epsilon_N)\log(N)+O(1),
	\end{align}
	where the constant order term is uniform in $N$. \eqref{equation:a..16} and \eqref{equation:a..17} imply that the integral in \eqref{equation:a.13} concentrates on the diagonal.
	Then, by independence of increments beyond the branching scale,
	\begin{align}\label{equation:a..18}
	\mathbb{E}\left[\psi^N_u \psi^N_v\right]=\int_{0}^{1}\sigma^2(s)\mathbb{E}\left[\nabla \phi^N_u(s)\nabla \phi^N_v(s)\right]\mathrm{d}s&= \int_{0}^{b_N(u,v)-\epsilon_N} \sigma^2(s)\mathbb{E}\left[\nabla \phi^N_u(s)\nabla \phi^N_v(s)\right]\mathrm{d}s \nonumber\\ &\quad+
	\int_{b_N(u,v)-\epsilon_N}^{b_N(u,v)}\sigma^2(s)\mathbb{E}\left[\nabla \phi^N_u(s)\nabla \phi^N_v(s)\right]\mathrm{d}s.
	\end{align}
	By Cauchy-Schwarz, the second integral in \eqref{equation:a..18} is absolutely bounded by a constant $C$ which depends on $\sigma$ but is independent of $N$. To bound the first integral in \eqref{equation:a..18} with $s=t$, note that in \eqref{equation:a..14} there are $2\pi \|u-v\|_2$ many pairs, $x\in \partial [u]_s,\, y\in [v]_s$ that have distance less than $\|u-v\|_2$ at scale $b_N(u,v)-\epsilon_N$. By \cite[Lemma~B.5]{2016arXiv160600510B} the harmonic measures evaluate to approximately $1/4\|u-v\|_2$. Thus, the sum over these particles is at most of order $O\left(\frac{\log_+ \|u-v\|}{\|u-v\|_2}\right)=O(1)$. For summands $x\in \partial [u]_s,\, y\in [v]_s$ and $\|x-y\|_2\geq \|u-v\|_2$, we use \eqref{equation:possion_kernel} and \cite[Lemma~B.5]{2016arXiv160600510B}, to deduce that the first integral in \eqref{equation:a..18} equals
	\begin{align}
	\log N\int_{0}^{b_N(u,v)-\epsilon_N}\sigma^2(s)  \mathrm{d}s+ O(1)=\log N \mathcal{I}_{\sigma^2}\left(1-\frac{\log_{+}\left( \|u-v\|_2\right)}{\log N}\right)+O(1).
	\end{align}
	This concludes the proof of the extension.
\end{proof}
\begin{proof}[Proof of \autoref{lemma:convergence_cov_neardiag_offdiag}]
	We start with the proof of the first statement.
		First note that by \autoref{lemma:cov_comp} $ii.$, for $xN+u, xN+v\in V_N^\delta$, it holds that
		\begin{align}
				\mathbb{E}\left[\psi^N_{xN+u}\psi^N_{xN+v}\right]=\log(N)+O(1).
		\end{align}
		Thus, one has to show that, as $N\rightarrow \infty$, the constant order contribution may depend on $u,v$, but not on $x$ and apart from this, has fluctuations which vanish as $N\rightarrow \infty$.
		By \eqref{equation:def_scale_dgff}, one has
		\begin{align}\label{equation:a.14}
				\mathbb{E}\left[\psi^N_{xN+u}\psi^N_{xN+v}\right]&=\int_{0}^{1}\sigma^2(s) \mathbb{E}\left[\nabla\phi^N_{xN+u}(s)\nabla\phi^N_{xN+v}(s) \right]\mathrm{d}s=\int_{0}^{\lambda_0}\sigma^2(s)\mathbb{E}\left[\nabla\phi^N_{xN+u}(s)\nabla\phi^N_{xN+v}(s) \right]\mathrm{d}s\nonumber\\
				&\quad+\int_{\lambda_0}^{1-\lambda_1}\sigma^2(s)\mathbb{E}\left[\nabla\phi^N_{xN+u}(s)\nabla\phi^N_{xN+v}(s) \right]\mathrm{d}s+\int_{1-\lambda_1}^{1}\sigma^2(s)\mathbb{E}\left[\nabla\phi^N_{xN+u}(s)\nabla\phi^N_{xN+v}(s) \right]\mathrm{d}s.
		\end{align}
		We choose $\lambda_0,\lambda_1=O\left(\frac{\log \log N}{\log N}\right)$, such that
		\begin{align}\label{equation:restriction_param}
			\sigma^2(0)\lambda_0+\sigma^2(1)\lambda_1+\int_{\lambda_0}^{1-\lambda_{1}}\sigma^2(s)\mathrm{d}s=1.
		\end{align}
		Note that we have by assumptions $\|u-v\|_2\leq L$, for $L\ll N$ and thus, we can assume $b_N(xN+u,xN+v)>1-\lambda_{1}$.
		For the first integral in \eqref{equation:a.14}, we use a Taylor expansion of $\sigma$ at $0$, i.e. $\sigma(s)=\sigma(0)+\sigma^\prime(0)s+o(\sigma^\prime(0)s)$, for $s\geq 0$ small. Thus, the first integral becomes
		\begin{align}\label{equation:int1}
			\int_{0}^{\lambda_0} \sigma^2(0)\mathbb{E}\left[\nabla\phi^N_{xN+u}(s)\nabla\phi^N_{xN+v}(s)\right]\mathrm{d}s+O(\lambda_0^2\log N \sigma(0) \sigma^\prime(0))\qquad\qquad\qquad\qquad\nonumber\\
			=\sigma^2(0)\mathbb{E}\left[\phi^N_{xN+u}(\lambda_0)\phi^N_{xN+v}(\lambda_0)\right]+ O(\lambda_0^2\log N \sigma(0) \sigma^\prime(0)),
		\end{align}
		 where the error term vanishes as $N\rightarrow \infty$, since $\lambda_0^2 \log N=O\left(\frac{\log \log N}{\log N}\right).$
		Similarly, by a Taylor expansion of $\sigma$ at $1$, i.e. $\sigma(s)=\sigma(1)-\sigma^\prime(1)(1-s)+o(\sigma^\prime(1)(1-s))$, for $s<1$ close to one, the last integral in \eqref{equation:a.14} can be computed as
		\begin{align}\label{equation:int2}
			\int_{1-\lambda_1}^{1}\sigma^2(1)\mathbb{E}\left[\nabla \phi^N_{xN+u}(s)\nabla\phi^N_{xN+v}(s)\right]\mathrm{d}s + O(\lambda_1^2\log N \sigma(1)\sigma^\prime(1))\qquad\qquad\qquad\qquad\qquad\qquad\nonumber\\
			=\sigma^2(1)\mathbb{E}\left[\left(\phi^N_{xN+u}(1)-\phi^N_{xN+u}(1-\lambda_1) \right)\left(\phi^N_{xN+v}(1)-\phi^N_{xN+v}(1-\lambda_1) \right)\right]+O(\lambda_1^2\log N \sigma(1)\sigma^\prime(1)).
		\end{align}
		Similarly as in \eqref{equation:int1}, the error term vanishes as $N\rightarrow \infty$.
		In all three cases in \eqref{equation:a.14}, using \eqref{equation:int1} and \eqref{equation:int2}, it suffices to compute quantities of the form $\mathbb{E}\left[\phi^N_{xN+u}(s)\phi^N_{xN+v}(s)\right]$. The case when $s=0$ is trivial since, for any $v\in V_N$, $\phi^N_v(0)=0$, as the harmonic average of the value zero is zero.
		Note that by \cite[(B.5),(B.6),(B.7)]{2016arXiv160600510B} one has, for $v,w\in V_N$,
		\begin{align}\label{equation:exp_1}
		\mathbb{E}\left[\phi^N_v\phi^N_{w}\right]&= -\mathfrak{a}(v-w)+\sum_{z\in \partial V_N}\mathbb{P}_v\left(S_{\tau_{\partial V_N}}=w\right) \mathfrak{a}(z-w),
		\end{align}
		where $\mathfrak{a}$ denotes the potential kernel, with representation as in \eqref{equation:possion_kernel}.
		First, consider the case when $0<s<1$. Note that the discrete harmonic measure converges weakly to the harmonic measure associated to Brownian motion \cite[Lemma~1.23]{1712.09972}, i.e. to the  measure $\Pi(x,A)\coloneqq \mathbb{P}_x\left(B_{\tau_{\partial [0,1]^2}}\in A\right)$, where $(B_t)_{t\geq 0}$ is Brownian motion in $\mathbb{R}^2$ killed upon exiting $[0,1]^2$. Moreover, since the logarithm is continuous and bounded in a neighbourhood of $\partial [0,1]^2$, using \eqref{equation:exp_1} and the weak convergence of the discrete harmonic measure, one obtains
		\begin{align}\label{equation:a.20}
			\mathbb{E}\left[\phi^N_{xN+u}(s)\phi^N_{xN+v}(s)\right]=\sum_{\substack{z\in \partial [xN+u]_s\\ y\in \partial [xN+v]_{s}}} \mathbb{P}_{xN+u}\left(S_{\tau_{\partial [xN+u]_{s}}}=z\right) \mathbb{P}_{xN+v}\left(S_{\tau_{\partial [xN+v]_{\lambda_i}}}=y\right)\mathbb{E}\left[\phi^N_z \phi^N_y\right]\nonumber\\
			=\sum_{\substack{z\in \partial [xN+u]_{s}\\ y\in \partial  [xN+v]_{s}}}\mathbb{P}_{xN+u}\left(S_{\tau_{\partial [xN+u]_{s}}}=z\right) \mathbb{P}_{xN+v}\left(S_{\tau_{\partial [xN+v]_{s}}}=y\right)\left( -\mathfrak{a}(u-v+N^{1-s}(c(z)-(y)))\right. \nonumber\\
			\qquad \left. +\sum_{w\in \partial V_N} \mathbb{P}_{xN+u+N^{1-s}}\left(S_{\tau_{\partial{V}_N}}=w\right)\mathfrak{a}(w-xN-v-N^{1-s}c(y))\right)\nonumber\\
			=-\log N^{1-s}+\log N+f(x)+o(1)=s \log N+f(x)+o(1),
		\end{align}
		where $f(x)=\int_{z\in \partial [0,1]^2}\Pi(x,\mathrm{d}z) \log\|z-x\|_2$. In particular, $f$ is continuous. Using \eqref{equation:a.20} and \eqref{equation:int1}, the first integral in \eqref{equation:a.14} can be rewritten as
		\begin{align}\label{equation:a..26}
			\sigma^2(0) \left(\lambda_0 \log N+ f(x)\right) + o(1).
		\end{align}
		For the remaining case, $s=1$, call $e_i$ the $i-$th unit vector. By \eqref{equation:exp_1} and using weak convergence of the discrete harmonic measure \cite[Lemma~1.23]{1712.09972},
		\begin{align}\label{equation:a.21}
		\mathbb{E}\left[\phi^N_{xN+u}(1)\phi^N_{xN+v}(1) \right]=\mathbb{E}\left[\phi^N_{xN+u} \phi^N_{xN+v} \right]
		=\log N + f(x) - \mathfrak{a}(u,v)+o(1).
		\end{align}
		Using \eqref{equation:int2} and \eqref{equation:a.21} allows to rewrite the third integral in \eqref{equation:a.14} as 
		\begin{align}\label{equation:a..28}
			\sigma^2(1)\left(\lambda_1 \log N - \mathfrak{a}(u,v)\right)+o(1).
		\end{align}
		Inserting \eqref{equation:a..26}, \eqref{equation:a..28} into \eqref{equation:a.14}, using \eqref{equation:a.20}, \eqref{equation:restriction_param} and $\mathcal{I}_{\sigma^2}(1)=1$, one obtains,
		\begin{align}
			\mathbb{E}\left[\psi^N_{xN+u}\psi^N_{xN+v}\right]=\log N+\sigma(0)^2 f(x)+\sigma(1)^2 g(u,v)+o(1),
		\end{align}
		with $g(u,v)=-\mathfrak{a}(u,v)$ and where $o(1)\rightarrow 0$, as $N\rightarrow \infty$. This concludes the proof of statement $i.$ in \autoref{lemma:convergence_cov_neardiag_offdiag}.
\par The covariances in the off-diagonal case, i.e. when $x\neq y \in (0,1)^2$, $\|x-y\|_2\geq 1/L$, can be computed similarly, now by Taylor expansion of the variance $\sigma(s)$ at $0$. First note that, for $\lambda=\frac{\log \log N}{\log N}$ and $N$ large enough, $\lambda> b_N(xN,yN)$. Thus,
\begin{align}\label{equation:a..41}
\mathbb{E}\left[\psi^N_{xN}\psi^N_{yN}\right]=
\int_{0}^{\lambda}\sigma^2(s)\mathbb{E}\left[\nabla \phi^N_{xN}(s)\nabla\phi^N_{yN}(s)\right]\mathrm{d}s=
\sigma^2(0)\mathbb{E}\left[\phi^N_{xN}(\lambda) \phi^N_{yN}(\lambda)\right] +O(\sigma(0)\sigma^\prime(0)\lambda^2 \log N).
\end{align}
By choice of $\lambda$, $O\left(\sigma(0) \sigma^\prime(0) \lambda \log N\right)=O\left(\sigma(0) \sigma^\prime(0) \frac{\log \log N}{\log N}\right)=o(1)$.
\begin{align}\label{equation:a..42}
\sigma^2(0)\mathbb{E}\left[\phi^N_{xN}(\lambda)\phi^N_{yN}(\lambda)\right]=\sigma^2(0)\sum_{\substack{u\in \partial [xN]_{\lambda}\\ v\in \partial  [xN]_{\lambda}}} \mathbb{P}_{xN}\left(S_{\tau_{\partial [xN]_{\lambda}}}=u\right)\mathbb{P}_{yN}\left(S_{\tau_{\partial [yN]_{\lambda_1}}}=v\right)\mathbb{E}\left[\phi^N_u \phi^N_v\right].
\end{align}
Using \eqref{equation:exp_1} and previous notation allows to reformulate \eqref{equation:a..42} as
\begin{align}\label{equation:a.23}
	\sigma^2(0)\sum_{\substack{u\in \partial [xN]_{\lambda}\\ v\in \partial  [xN]_{\lambda}}}\mathbb{P}_{xN}\left(S_{\tau_{\partial [xN]_{\lambda}}}=u\right)\mathbb{P}_{yN}\left(S_{\tau_{\partial [yN]_{\lambda}}}=v\right)\left(-\mathfrak{a}(N(x-y+ N^{-\lambda}(c(u)-c(v))))\right. \nonumber\\
	\left. + \sum_{w\in \partial V_N}\mathbb{P}_{xN}\left(S_{\tau_{\partial V_N}}=w\right)\mathfrak{a}(w-yN) \right).
\end{align}
Using \eqref{equation:possion_kernel}, we rewrite \eqref{equation:a.23}
\begin{align}\label{equation:a.25}
	\sigma^2(0)\sum_{\substack{u\in \partial [xN]_{\lambda}\\ v\in \partial  [xN]_{\lambda}}}\mathbb{P}_{xN}\left(S_{\tau_{\partial [xN]_{\lambda}}}=u\right)\mathbb{P}_{yN}\left(S_{\tau_{\partial [yN]_{\lambda}}}=v\right)\left(-\log N -\log \|x-y\|_2-c_0+o(1)\right. \nonumber\\
	\left. + \sum_{w\in \partial V_N}\mathbb{P}_{xN}\left(S_{\tau_{\partial V_N}}=w\right)(\log N +\log\|c(w)-y\|_2 +c_0 +o(1) ) \right)\nonumber\\
	=\sigma^2(0) h(x,y)+o(1),
\end{align}
where $h(x,y)=-\log \|x-y\|_2+\int_{\partial [0,1]^2}\Pi(x,\mathrm{d}z)\log\|z-y\|_2,$ by the weak convergence of the harmonic measure to $\Pi$. In particular, $h$ is continuous on $[0,1]^2\setminus\{(x,x):\, x\in [0,1] \}$. This concludes the proof of the second statement and thus, of \autoref{lemma:convergence_cov_neardiag_offdiag}.
	\end{proof}
\begin{proof}[Proof of \autoref{lemma:cov_estimates_3field}]: We start with the proof of $(i)$.
		Let $i^\prime$ be such that $u,v\in B_{L^\prime,i}\subset B_{K^\prime L^\prime,i^\prime}$. By \eqref{equation:def_3-field_approximand}, one has
		\begin{align}
			S^N_u-S^N_v&= \left(S^{N,c}_u-S^{N,c}_v\right)+\left(S^{N,m}_u-S^{N,m}_v\right)+\left(S^{N,b}_u-S^{N,b}_v\right)+\Phi_{i^\prime}\left(a_{K^\prime L^\prime,\bar{u}}-a_{K^\prime L^\prime,\bar{v}}\right)\nonumber\\& = S^{N,b}_u -S^{N,b}_v +\Phi_{i^\prime}\left(a_{K^\prime L^\prime,\bar{u}}-a_{K^\prime L^\prime,\bar{v}}\right).
		\end{align}
			In particular, by \eqref{equation:5.18}, $|a_{K\prime L^\prime, \bar{u}}-a_{K^\prime L^\prime, \bar{v}}| \leq \epsilon_{N,KL,K^\prime L^\prime}$, and so
		\begin{align}\label{equation:aa.46}
			&\left|\mathbbm{E}\left[\left(S^N_u-S^N_v\right)^2\right]- \mathbbm{E}\left[\left(\psi^N_u-\psi^N_v\right)^2\right] \right|\nonumber\\ &\qquad\leq 4\epsilon_{N,KL,K^\prime L^\prime}+ \left|\sigma^2(1)\mathbbm{E}\left[\left(\phi^{K^\prime L^\prime}_{u-v_{K^\prime L^\prime,i^\prime}}-\phi^{K^\prime L^\prime}_{v-v_{K^\prime L^\prime,i^{\prime}}}\right)^2\right]-\mathbbm{E}\left[\left(\psi^N_u-\psi^N_v\right)^2\right] \right|.\qquad\qquad
		\end{align}
			Using the tower property of conditional expectation, conditioning $\{\psi^N_v\}_{v\in V_N}$ on $\sigma\left(\phi^N_w:w\in[v_{K^\prime L^\prime, i^\prime}]_{K^\prime L^\prime}^c\right)$ and using \eqref{equation:a.21} and \autoref{lemma:convergence_cov_neardiag_offdiag} $ii.$, it follows that
		\begin{align}\label{equation:a.27}
			\limsup\limits_{(L,K,L^\prime,K^\prime,N)\Rightarrow \infty} \sup_{\substack{u,v \in B_{L^\prime,i}\cap V^*_{N,\delta}\\  1\leq i\leq (N/L^\prime)^2}} \left|\sigma^2(1)\mathbbm{E}\left[\left(\phi^{K^\prime L^\prime}_{u-v_{K^\prime L^\prime,i^\prime}}-\phi^{K^\prime L^\prime}_{v-v_{K^\prime L^\prime,i^{\prime}}}\right)^2\right]-\mathbbm{E}\left[\left(\psi^N_u-\psi^N_v\right)^2\right] \right|=0.
		\end{align}
		Statement $i.$ follows from \eqref{equation:a.27} together with \eqref{equation:aa.46}. Next, we prove $ii.$. Let $i^\prime\neq j^\prime$ be such that $u\in B_{N/KL,i^\prime},\,v\in B_{N/KL,j^\prime}$ and assume without loss of generality that $N\gg K^\prime \gg L^\prime \gg K \gg L\gg 1/\delta.$ Since vertices $u$ and $v$ belong to distinct boxes of side length $N/KL$ and thus, also to distinct $K^\prime L^\prime-$boxes, both $\mathbb{E}\left[S^{N,m}_u S^{N,m}_v\right]=0$ and $\mathbb{E}\left[S^{N,b}_u S^{N,b}_v\right]=0$. Using these observations, scaling the DGFF from $V_{KL}$ to $V_N$ and by \eqref{equation:exp_1},
		\begin{align}\label{eq:equation6.67}
			\mathbbm{E}\left[S^N_uS^N_v\right]&=\mathbbm{E}\left[S^{N,c}_uS^{N,c}_v\right]=\sigma^2(0)\mathbbm{E}\left[\phi^{KL}_{w_{i^\prime}}\phi^{KL}_{w_{j^\prime}}\right]
			= \sigma^2(0) \mathbbm{E}\left[\phi^N_{v_{N/KL,i^\prime}}\phi^N_{v_{N/KL,j^\prime}} \right]+o(1).
		\end{align}
		Since $\|\frac{v_{N/KL,i^\prime}-u}{N}\|_2,\|\frac{v_{N/KL,j^\prime}-v}{N}\|_2=O\left(\frac{1}{KL}\right)
		$, \cite[Lemma B.14]{2016arXiv160600510B} implies
		\begin{align}\label{equation:a.29}
			\limsup\limits_{(L,K,L^\prime,K^\prime,N)\Rightarrow \infty}\sup_{\substack{u\in B_{N/KL,i^\prime}\cap V_{N,\delta}^*\\v\in B_{N/KL,j^\prime}\cap V_{N,\delta}^*,\, i^\prime\neq j^\prime}} \left|\mathbbm{E}\left[S^N_uS^N_v\right]- \sigma^2(0) \mathbbm{E}\left[\phi^N_u\phi^N_v\right]\right|=0.
		\end{align}
			On the other hand, the vertices $u,v$ are at distance of order $N/KL$ away from each other. Since considering limits of the form $(L,K,L^\prime,K^\prime,N)\Rightarrow \infty$, one can assume that $N/KL \gg N^{1-\lambda_1}$, and thus $\mathbb{E}\left[\phi^N_u \phi^N_v\right]=\mathbb{E}\left[\phi^N_u(\lambda_1)\phi^N_v(\lambda_1)\right].$ Therefore, by a Taylor expansion of $\sigma$ at $0$ as in \eqref{equation:a..41},
		\begin{align}\label{eq:equation6.69}
			\left|\mathbbm{E}\left[\psi^N_u\psi^N_v\right]-\sigma^2(0) \mathbbm{E}\left[\phi^N_u \phi^N_v\right]\right|&= \left|\sigma^2(0) \mathbb{E}\left[\phi^N_u(\lambda_1)\phi^N_v(\lambda_1)\right]-\sigma^2(0) \mathbbm{E}\left[\phi^N_u \phi^N_v\right]\right|+o(1)\rightarrow 0,
		\end{align}
		as $N\rightarrow \infty$.
		\eqref{equation:a.29} together with \eqref{eq:equation6.69} implies statement $ii.$. Note that for statement $iii.$, one has $\|u-v\|_2=O(N/L)$. This allows to approximate as in \eqref{eq:equation6.69}. Note that in this case, there is a constant $L\geq c(u,v)>0$, such that the leading order of the first covariance is given by $\log(\|u-v\|_2+N^{1-\lambda_1})-\log(\|u-v\|_2)=\log\left(1+\frac{cL}{N^{\lambda_1}}\right)$.
		In the following, we distinguish three cases:
			\begin{enumerate}
				\item \label{case:a1} $u,v\in B_{K^\prime L^\prime,i}$ but $u\in B_{L^\prime,i^\prime}$ and $v\in B_{L^\prime,j^\prime}$
				\item \label{case:a12} $u,v \in B_{N/KL,i}$, but $u \in B_{K^\prime L^\prime,\tilde{i}}$ and $v\in B_{K^\prime L^\prime, \tilde{j}}$
				\item \label{case:a2} $u\in B_{N/KL,i}\cap B_{L^\prime,i^\prime}$ and $v\in B_{N/KL,j}\cap B_{L^\prime,j^\prime}.$
			\end{enumerate}
			In case (\ref{case:a1}), $S^{N,c}_u=S^{N,c}_v$ and $S^{N,m}_u=S^{N,m}_v$ and so, using notation from the proof of \autoref{lemma:convergence_cov_neardiag_offdiag}, by \eqref{equation:exp_1}, \eqref{eq:properties_a_N,v}, \eqref{eq:a_N,bar(v)} and as in \eqref{equation:a.21},
		\begin{align}\label{equation:a.55}
			\mathbbm{E}\left[S^N_uS^N_v\right]&=\mathrm{Var}\left[S^{N,c}_uS^{N,c}_v\right]+ \mathrm{Var}\left[S^{N,m}_u \right]+\mathbbm{E}\left[S^{N,b}_u S^{N,b}_v\right]+a_{K^\prime L^\prime,\bar{u}}a_{K^\prime L^\prime, \bar{v}}+o(1)\nonumber\\
			&=\log N +\sigma^2(0)f\left(\frac{u}{N}\right)+\sigma^2(1) \left(-\mathfrak{a}(u-v)+\int_{\partial [0,1]^2}\Pi\left(\frac{u}{N},\mathrm{d}z\right)\mathfrak{a}\left(z-\frac{v}{K^\prime L^\prime}\right)\right)\nonumber\\ &\quad+ a_{K^\prime L^\prime,\bar{u}}a_{K^\prime L^\prime, \bar{v}}+o(1).
		\end{align}
		Since $u,v\in V^{*}_{N,\delta}$ are away from the boundary, the integral in \eqref{equation:a.55} is bounded by a constant $C_\delta$, depending on $\delta$. Thus, \eqref{equation:a.55} can be written as $\log N-\sigma^2(1) \log_+ \|u-v\|_2 +O(1)$, where the constant order term is bounded by $8\alpha +C_\delta$.
		By \autoref{lemma:cov_comp} $ii.$, $\mathbb{E}\left[\psi^N_u \psi^N_v\right]=\log N - \sigma^2(1)\log_{+}\|u-v\|_2+ O(1)$, where the constant order term is bounded by $\alpha$. Thus, statement $ii.$ follows in case (\ref{case:a1}).
			In case (\ref{case:a12}), $\mathbb{E}\left[S^{N,b}_uS^{N,b}_v\right]=0.$ Thus, there is a constant $c_1>0$, such that
			\begin{align}\label{equation:a.56}
				\mathbb{E}\left[S^N_u S^N_v\right]= \mathbb{E}\left[S^{N,c}_u S^{N,c}_v\right]+\mathbb{E}\left[S^{N,m}_u S^{N,m}_v\right]+c_1.
			\end{align}
			To estimate the first covariance in \eqref{equation:a.56}, apply \eqref{equation:exp_1} and for the second, note that $\{S^{N,m}_v\}_{v\in V_N}$ is a MIBRW, and thus, using \autoref{lemma:cov_comp} $i.$ and $ii.$, statement $ii.$ follows, in case (\ref{case:a12}).
			In case (\ref{case:a2}), $\mathbb{E}\left[S^{N,m}_uS^{N,m}_v\right]=0$ and $\mathbb{E}\left[S^{N,b}_uS^{N,b}_v\right]=0$. By scaling the DGFF as in \eqref{eq:equation6.67} and using \eqref{equation:exp_1},
		\begin{align}
			\mathbbm{E}\left[S^N_uS^N_v\right]=\mathbbm{E}\left[S^{N,c}_uS^{N,c}_v\right]=\sigma^2(0)\left(\log(N)-\log_{+}(\|u-v\|_2)\right)+c+o(1),
		\end{align}
			where $c$ is a bounded constant depending on $\delta$ and where the error $o(1)$ vanishes as $N\rightarrow \infty$.
			The same reasoning applied to $\mathbb{E}\left[\psi^N_u \psi^N_v\right]$ as in \eqref{eq:equation6.69}, implies the claim in this remaining case and thereby concludes the proof \autoref{lemma:cov_estimates_3field}.
\end{proof}
	\section{Proof of \autoref{lemma:1} and \autoref{lemma:2}}\label{section:justification_approximation}
	We prove \autoref{lemma:1} in the case of the scale-inhomogeneous DGFF. The proof for the approximating field, $\{S^N_v\}_{V\in V_N}$, is essentially identical. This is due to \autoref{lemma:cov_estimates_3field}, which allows to use Gaussian comparison to reduce the proof to the one we provide.

\begin{lemma}\label{lemma:adding_gaussiantailtype_rv}
	Let $\{g^N_v:\,u\in V_N\}$ be a collection of random variables, independent of the centred Gaussian field, $\{\bar{\psi}^N_u:\,u\in V_N \}$, and the 2d scale-inhomogeneous DGFF, $\{\psi^N_u:\, u\in V_N\}$, such that 
	\begin{align}\label{eq:assumptionlemma1}
	\mathbbm{P}\left(g^N_u\geq 1+y \right)\leq e^{-y^2}\quad \forall u \in V_N.
	\end{align}
	Assume further that there is some $\delta>0$, such that, for all $v,w\in V_N$, $\mathbb{E}\left[\bar{\psi}^N_v \bar{\psi}^N_w \right]-\mathbb{E}\left[\psi^N_v \psi^N_w\right]| \leq \delta$.
	Then, there is a constant $C=C(\alpha)$, such that, for any $\epsilon>0,$ $N\in \mathbbm{N}$ and $x\geq -\sqrt{\epsilon}$,
	\begin{align}
	\mathbbm{P}\left(\max_{v\in V_N}\left(\bar{\psi}^N_v+\epsilon g^N_v \right)\geq m_N+x \right)\leq \mathbbm{P}\left(\max_{v\in V_N} \bar{\psi}^N_v\geq m_N+x-\sqrt{\epsilon} \right)\left(Ce^{-C^{-1}\epsilon^{-1}}\right).
	\end{align}
	\begin{proof}
		Let $\Gamma_y \coloneqq \{v\in V_N:\, y/2\leq \epsilon g^N_v\leq y \}$. Then,
			\begin{align}\label{equation:b.15}
			\mathbb{P}\left(\max_{v\in V_N} \left(\bar{\psi}^N_v +\epsilon g^N_v \right) \geq m_N +x \right)  \leq & \mathbbm{P}\left(\max_{v\in V_N}\bar{\psi}^N_v\geq m_N+x-\sqrt{\epsilon} \right)\nonumber\\
			&+\sum_{i=0}^{\infty}\mathbbm{E}\left[\mathbbm{E}\left[\mathbbm{1}_{\max_{v\in \Gamma_{2^i \sqrt{\epsilon}}}\bar{\psi}^N_v\geq m_N+x-2^i \sqrt{\epsilon}} \bigg|\Gamma_{2^i \sqrt{\epsilon}}\right]\right].
			\end{align}
			By \autoref{proposition:max_over_A}, the last sum in \eqref{equation:b.15} can be bounded from above by
			\begin{align}\label{eq:0001}
			\tilde{c} e^{-2x}\sum_{i=0}^{\infty}\mathbbm{E}\left[|\Gamma_{2^i \sqrt{\epsilon}}|/|V_N| \right]e^{2^{i+1} \sqrt{\epsilon}},
			\end{align}
			with $\tilde{c}>0$ being a finite constant.
			By assumption \eqref{eq:assumptionlemma1}, one has
			\begin{align}
			\mathbbm{E}\left[|\Gamma_{2^i \sqrt{\epsilon}}|/|V_N| \right]\leq e^{-4^i (C\epsilon)^{-1}}.
			\end{align}
			Thus, \eqref{eq:0001} is bounded from above by $\tilde{c} e^{-2x}e^{-(C\epsilon)^{-1}}$. This concludes the proof of \autoref{lemma:adding_gaussiantailtype_rv}.
	\end{proof}
\end{lemma}
\begin{prop}\label{proposition:prop6.3}
	Let $\{\varphi^N_v\}_{v\in V_N}, \{\tilde{\varphi}^N_v\}_{v\in V_N}$ be two independent centred Gaussian fields satisfying the covariance estimates in \autoref{lemma:cov_estimates_3field}, and let $\{g_B:B\subset V_N\}$ be a family of independent standard Gaussians. Moreover, let $\tilde{\sigma}=(\tilde{\sigma}_1,\tilde{\sigma}_2)\in \mathbb{R}_+^2$ and $\{\varphi^{N,r,\tilde{\sigma}}_v:\,v\in V_N\}$ and $\{\varphi^{N,\tilde{\sigma},*}_v:\, v\in V_N \}$ be two centred Gaussian fields, given by
	\begin{align}
		\varphi^{N,r_1,r_2,\tilde{\sigma}}_v=\varphi^N_v+\tilde{\sigma}_1 g_{B_{v,r_1}}+\tilde{\sigma}_2 g_{B_{v,N/r_2}},
	\end{align}
	and
	\begin{align}\label{eq:6.28}
		\varphi^{N,\tilde{\sigma},*}_v=\varphi^N_v+\sqrt{\frac{\|\tilde{\sigma}\|_2^2}{\log N}} \tilde{\varphi}^N_v,
	\end{align}
	for $v \in V_N$.
	Set $M_{N,r_1,r_2,\tilde{\sigma}}=\max\limits_{v\in V_N} \varphi^{N,r_1,r_2,\tilde{\sigma}}_v,$ and likewise, $M_{N,\tilde{\sigma},*}=\max\limits_{v\in V_N} \varphi^{N,\tilde{\sigma},*}_v$.
	Then, for any fixed $\tilde{\sigma}\in (0,\infty)^2$,
	\begin{align}\label{equation:0002}
		\lim\limits_{r_1,r_2\rightarrow \infty}\limsup\limits_{N\rightarrow \infty}d\left(M_{N,r_1,r_2,\tilde{\sigma}}-m_N,M_{N,\tilde{\sigma},*}-m_N\right)=0.
	\end{align}
	\begin{proof}
		Partition $V_N$ into boxes of side length $N/r_2$ and denote by $\mathcal{B}$ the collection of these boxes. Fix arbitrary $\delta>0$, for $B\in \mathcal{B}$ denote by $B_\delta$ the box with the same centre as $B$, but with side length $(1-\delta)N/r_2$.
		The union of such restricted boxes, we call $V_{N,\delta}=\bigcup\limits_{B \in \mathcal{B}}B_\delta.$ The maxima over these sets, we denote by $M_{N,r_1,r_2,\tilde{\sigma},\delta}=\max\limits_{v\in V_{N,\delta}}  \varphi^{N,r_1,r_2,\tilde{\sigma}}_v$ and $M_{N,\tilde{\sigma},*,\delta}=\max\limits_{v\in V_{N,\delta}} \varphi^{N,\tilde{\sigma},*}_v$.
		By \autoref{proposition:max_over_A},
		\begin{align}
			\lim\limits_{\delta \rightarrow 0}\lim\limits_{N\rightarrow \infty} \mathbbm{P}\left(M_{N,r_1,r_2,\tilde{\sigma},\delta}\neq M_{N,r_1,r_2,\tilde{\sigma}} \right)= \lim\limits_{\delta\rightarrow 0}\lim\limits_{N\rightarrow \infty} \mathbbm{P}\left(M_{N,\tilde{\sigma},*,\delta}\neq M_{N,\tilde{\sigma},*} \right)=0.
		\end{align}
		Thus, it suffices to show equation \eqref{equation:0002} with $M_{N,r_1,r_2,\tilde{\sigma},\delta}-m_N$ and $M_{N,\tilde{\sigma},*,\delta}-m_N$.
		Next, we show that the main contribution to the maximum is given by $\{\varphi^N_v\}_{v\in V_N}$, while the perturbation fields only have a negligible influence.
		For $B\in \mathcal{B}$, let $z_b\in B$ the maximizing element, i.e. $\max_{v\in B_\delta} \varphi^N_v=\varphi^N_{z_B}.$
		The claim is that
		\begin{align}\label{equation:0003}
			\lim\limits_{r_1,r_2 \rightarrow \infty}\limsup\limits_{N\rightarrow \infty}\, &\mathbbm{P}\left(|M_{N,r_1,r_2,\tilde{\sigma},\delta}-\max_{B\in \mathcal{B}}\varphi^{N,r_1,r_2,\tilde{\sigma}}_{z_B}|\geq \frac{1}{\log n} \right)\nonumber\\&=
			\limsup\limits_{N\rightarrow \infty} \mathbbm{P}\left(|M_{N,\tilde{\sigma},*,\delta}-\max_{B\in \mathcal{B}}\varphi^{N,\tilde{\sigma},*}_{z_B}|\geq \frac{1}{\log n} \right)=0.\qquad\qquad\qquad
		\end{align}
		We first show how \autoref{proposition:prop6.3} follows from \eqref{equation:0003}.
		Assuming \eqref{equation:0003}, conditioning on the positions of the maximum, $\{z_B\}_{B \in \mathcal{B}}$, one deduces that the centred Gaussian field $\left\{\sqrt{\|\tilde{\sigma}\|^2_2/\log N}\tilde{\varphi}^N_{z_B} \right\}_{B\in \mathcal{B}}$ has pairwise correlations of order at most $O(1/\log N)$. Thus, the conditional covariance matrices of $\left\{\sqrt{\frac{\|\tilde{\sigma}\|^2_2}{\log(N)}}\tilde{\varphi}^N_{z_B} \right\}_{B\in \mathcal{B}}$ and $\{\tilde{\sigma}_1 g_{B_{z_B,r_1}}+\tilde{\sigma}_2 g_{B_{z_B,N/r_2}} \}_{B\in \mathcal{B}}$ are within $O(1/\log N)$ of each other entry-wise. In combination with \eqref{equation:0003} this proves \autoref{proposition:prop6.3}.
		It remains to prove \eqref{equation:0003}. Suppose that on the contrary, either of the events considered in the probabilities in \eqref{equation:0003} occurs. By \eqref{equation:expected_max} and Gaussian comparison, we know that $E_1=E_1(C)=\{\omega:M_{N,r_1,r_2,\tilde{\sigma},\delta}\notin (m_N-C,m_N+C) \}\cup \{M_{N,\tilde{\sigma},*,\delta}\notin (m_N-C,m_N+C) \}$ has a probability tending to $0$, i.e. $\lim\limits_{C\rightarrow \infty} \limsup\limits_{N\rightarrow \infty} \mathbbm{P}\left(E_1\right)=0.$ Moreover,  \autoref{thm:prob_dist_maximalparticles} implies that also the event $E_2=\{\omega:\exists u,v\in V_N: \|u-v\|_2 \in (r,N/r)\text{ and } \min(\varphi^N_u,\varphi^N_v)>m_N -c\log \log r  \}$ cannot occur, i.e. $\lim\limits_{r\rightarrow \infty}\limsup\limits_{N\rightarrow \infty} \mathbbm{P}\left(E_2\right)=0.$ Note that \autoref{thm:prob_dist_maximalparticles} is stated only for the scale-inhomogeneous DGFF. However, using the covariance assumptions and Gaussian comparison, it is possible to replace $\{\psi^N_v\}_{v\in V_N}$ with $\{\varphi^N_v\}_{v\in V_N}$ throughout the proof of \autoref{thm:prob_dist_maximalparticles}. This allows to assume the event $E_1^c\cap E^c_2$. To show \eqref{equation:0003}, we consider the following events:
		\begin{itemize}
			\item[--]
			$
				E_3=\tilde{E}_3\cup E^{*}_3, \text{ where } \tilde{E}_3=\{\omega:\exists v\in V_N:\, \varphi^{N,r_1,r_2,\tilde{\sigma}}_=M_{N,r_1,r_2,\tilde{\sigma},\delta},\, \varphi^N_v\leq m_N-c\log \log r \} \text{ and }E^{*}_3=\{\omega: \exists v\in V_N: \, \varphi^{N,\tilde{\sigma},*}_v=M_{N,\tilde{\sigma},*,\delta},\, \varphi^N_v\leq m_N-c\log \log r \}.
			$
			\item[--]
			$
				E_4=\{\omega:\exists v\in B, B\in \mathcal{B}:\, \varphi^N_v\geq m_N-c\log \log r \text{ and }  \sqrt{\frac{\|\tilde{\sigma}\|^2_2}{\log N}}\left(\tilde{\varphi}^N_v-\tilde{\varphi}^N_{z_B}\right)\geq 1/\log n \}.
			$
		\end{itemize}
		E3: Let $\Gamma_x=\{v\in V_N:\, \varphi^{N,r_1,r_2,\tilde{\sigma}}_v-\varphi^N_v \in (x,x+1)\}.$ The idea is that, by localizing and conditioning on the difference of the two Gaussian fields through the set $\Gamma_x$, one can use \autoref{proposition:max_over_A} to bound $\max\limits_{v \in \Gamma_x}\varphi^N_v$ from above, i.e.
			\begin{align}
				\mathbb{P}\left(E^c_1\cap \tilde{E}_3 \right)&\leq \mathbb{P}\bigg(\max_{x\geq c\log(n)-C}\max_{v \in \Gamma_x} \varphi^{N,r_1,r_2,\tilde{\sigma}}_v\geq m_N-C \bigg)\leq \sum_{x\geq c\log(n)-C} \mathbbm{P}\bigg(\max_{v \in \Gamma_x}\varphi^{N,r_1,r_2,\tilde{\sigma}}_v\geq m_N-C \bigg) \nonumber
				\\ &\leq \sum_{x\geq c\log(n)-C} \mathbbm{E}\left[\mathbb{P}\left(\max_{v \in \Gamma_x} \varphi^N_v \geq m_N-x-C| \Gamma_x \right) \right] \leq\tilde{c} \sum_{x\geq c\log(n)-C} \mathbbm{E}\left[|\Gamma_x|/|V_N| \right]e^{2x}.
			\end{align}
			By a first moment bound for Gaussian random variables, one has
			\begin{align}
				\mathbbm{E}\left[|\Gamma_x|^{1/2}/|V_N|^{1/2}\right]\leq &\mathbbm{E}\left[ |\{v\in V_N: \tilde{\sigma}_1 g_{B_{v,r_1}}+\tilde{\sigma}_2 g_{B_{v,N/r_2}}\in (x,x+1) \}|^{1/2} \right]/|V_N|^{1/2}\nonumber\\
				 \leq &\mathbbm{P}\left(\tilde{\sigma}_1 g_{B_{v,r_1}}+\tilde{\sigma}_2 g_{B_{v,N/r_2}}\in (x,x+1)\right)^{1/2} \leq e^{-c^{'}x^2}/c^{'},
			\end{align} 
			for some constant $c^{\prime}=c^{'}(\sigma,\tilde{\sigma})>0$.
			Thus,
			\begin{align}
				\limsup\limits_{C\rightarrow \infty}\limsup\limits_{r\rightarrow \infty}\limsup\limits_{N\rightarrow \infty} \mathbbm{P}\left(E^c_1(C)\cap \tilde{E}_3 \right)=0.
			\end{align}
			In the same way, one can prove an analogue estimate for $E^{*}_3$ in place of $\tilde{E}_3$, which gives
			\begin{align}
			\limsup\limits_{C\rightarrow \infty}\limsup\limits_{r\rightarrow \infty}\limsup\limits_{N\rightarrow \infty} \mathbbm{P}\left(E^c_1(C)\cap E_3 \right)=0.
			\end{align}\\
		E4: Let $\Gamma^{'}_r=\{v\in V_N :\, \varphi^N_v\geq m_N-c\log\log r\}$. In $V_N$, there can be at most $r^2$ particles at minimum distance $N/r$, and around each of these, one can find approximately $r^2$ particles in $V_N$ which are within distance $r$. Thus, on $E^c_2$, one has $|\Gamma^{'}_r|\leq 2r^4.$
			Further, for each $v\in B\cap\Gamma^{'}_r$ and in the event of $E^c_2$, one has $\|v-z_B\|_2\leq r$. Thus, by independence between the Gaussian fields $\{\varphi^N_v\}_{v\in V_N}$ and $\{\varphi^{N,'}_v \}_{v\in V_N}$, and using 2nd order Chebychev's inequality,
			\begin{align}
				&\mathbbm{P}\left( \sqrt{\frac{\|\tilde{\sigma}\|^2_2}{\log N }}\left(\varphi^{N,'}_v- \varphi^{N,'}_{z_B}\right)\geq \frac{1}{\log\log N} \right)\leq \frac{(\tilde{c}(\sigma,\tilde{\sigma}) \log r+c_1)\left(\log\log N\right)^2}{\log N},
			\end{align}
			where $\tilde{c},c_1>0$ are finite constants.
			Therefore, and by a union bound,
			\begin{align}
				\limsup\limits_{r\rightarrow \infty}&\limsup\limits_{N\rightarrow \infty}\mathbbm{P}\left(E_4\cap E^c_2 \right)\leq \limsup\limits_{r\rightarrow \infty}\limsup\limits_{N\rightarrow \infty} 2r^4[\tilde{c}(\sigma,\tilde{\sigma}) \log r+c_1] \frac{\left(\log\log N\right)^2}{\log N} =0.
			\end{align}
		This concludes the proof of equation \eqref{equation:0003} and thereby, the proof of \autoref{proposition:prop6.3}.
	\end{proof}
\end{prop}
\begin{proof}[Proof of \autoref{lemma:1}:]
	We prove \autoref{lemma:1} in the case of the scale-inhomogeneous DGFF. \autoref{lemma:1} for the approximating field follows from Gaussian Define $\bar{\psi}^{N,\tilde{\sigma}}=\left(1+\frac{\|\tilde{\sigma}\|^2_2}{2\log N} \right)\psi^N_v $, for $v\in V_N$, and set $M_N=\max_{v \in V_N} \psi^N_v$ and $\bar{M}_{N,\tilde{\sigma}}= \max\limits_{v\in V_N} \bar{\psi}^{N,\tilde{\sigma}}.$ One has $\bar{M}_{N,\tilde{\sigma}}=\left(1+\frac{\|\tilde{\sigma}\|^2_2}{\log N}\right) M_N$. Using \eqref{equation:expected_max}, this gives us both
	\begin{align}
		\mathbbm{E}\left[\bar{M}_{N,\tilde{\sigma}} \right]= \mathbbm{E}\left[M_N\right]+2\|\tilde{\sigma}\|_2^2+o(1),
	\end{align}
	and
	\begin{align}
		\lim\limits_{N\rightarrow \infty} d\left(M_N-\mathbbm{E}\left[M_N\right],\bar{M}_{N,\tilde{\sigma}}-\mathbbm{E}\left[\bar{M}_{N,\tilde{\sigma}} \right] \right)=0.
	\end{align}
	Further, let $\{\psi^{N,\tilde{\sigma},*}_v:\,v\in V_N \}$ be defined as in \eqref{eq:6.28} and set $M_{N,\tilde{\sigma},*}= \max_{v\in V_N} \psi^{N,\tilde{\sigma},*}_{v}.$
	In the distributional sense, $\left\{\bar{\psi}^{N,\tilde{\sigma}}_v\right\}_{v\in V_N}$ can be considered as a sum of $\left\{\psi^{N,\tilde{\sigma},*}_v\right\}_{v\in V_N}$ and an independent centred Gaussian field with variances of order $O((1/\log N )^3)$.
	Thus, by Gaussian comparison, it follows that
	\begin{align}
		\mathbbm{E}\left[\bar{M}_{N,\tilde{\sigma}}\right]=\mathbbm{E}\left[M_{N,\tilde{\sigma},*}\right]+o(1),
	\end{align}
	as well as
	\begin{align}\label{eq:6.46}
		\lim\limits_{N\rightarrow \infty} d\left(\bar{M}_{N,\tilde{\sigma}}-\mathbbm{E}\left[\bar{M}_{N,\tilde{\sigma}}\right],M_{N,\tilde{\sigma},*}-\mathbbm{E}\left[M_{N,\tilde{\sigma},*}\right]\right)=0.
	\end{align}
	By \eqref{eq:6.46}, \autoref{proposition:prop6.3}, and using the triangle inequality, one concludes the proof of \autoref{lemma:1}.
\end{proof}
\begin{proof}[Proof of \autoref{lemma:2}]
	Recall that we want to prove asymptotic stochastic domination. The basic idea is to use Slepian's Lemma. Let $\Phi,\,\{\Phi^N_v\}_{v\in V_N}$ be independent standard Gaussian random variables and for some $\epsilon^{*}>0$, set
	\begin{align}
		\psi^{N,lw,\epsilon^{*}}_v&=\left(1-\frac{\epsilon^{*}}{\log N }\right)\psi^N_v+\epsilon^{N,\prime}_v \Phi\\
		\bar{\psi}^{N,up,\epsilon^{*}}_v&=\left(1-\frac{\epsilon^{*}}{\log N}\right)\bar{\psi}^N_v+\epsilon^{N,\prime\prime}_v \Phi^N_v,
	\end{align}
	where $\epsilon^{N,\prime}_v=\epsilon^{N,\prime}_v(\epsilon,\epsilon^{*})$ and $\epsilon^{N,\prime\prime}_v=\epsilon^{N,\prime\prime}_v(\epsilon,\epsilon^{*})$ are chosen such that 
	\begin{align}\label{equation:6.4_1}
		\mathrm{Var}\left[\psi^{N,lw,\epsilon^{*}}_v\right]&=\left(1-\frac{\epsilon^{*}}{\log N}\right)^2\mathrm{Var}\left[\psi^N_v\right]+(\epsilon^{N,\prime}_v)^2\overset{}{=}\mathrm{Var}\left[\psi^N_v\right]+\epsilon
		\end{align}
		and
		\begin{align}\label{equation:6.49_2}
		\mathrm{Var}\left[\bar{\psi}^{N,up,\epsilon^{*}}_{v}\right]&=\left(1-\frac{\epsilon^{*}}{\log N}\right)^2\mathrm{Var}\left[\bar{\psi}^N_v\right]+(\epsilon^{N,\prime\prime}_v)^2 \overset{}{=}\mathrm{Var}\left[\psi^N_v\right]+\epsilon.
	\end{align}
	Solving for $\epsilon^{N,\prime}_v$ in \eqref{equation:6.4_1}, gives
	\begin{align}\label{equation:eps_prime}
		(\epsilon^{N,\prime}_v)^2=\frac{\epsilon^{*}}{\log N}\mathrm{Var}\left[\psi^N_v\right]+\epsilon.
	\end{align}
	Moreover, for $u\neq v\in V_N$,
	\begin{align}\label{equation:cov_lwup-comp1}
		\mathbbm{E}\left[\psi^{N,lw,\epsilon^{*}}_u \psi^{N,lw,\epsilon^{*}}_v\right]&=\left(1-\frac{\epsilon^{*}}{\log N}\right)^2\mathbbm{E}\left[\psi^N_u \psi^N_v\right]+ \epsilon^{N,\prime}_u\epsilon^{N,\prime}_v
	\end{align}
	and by \eqref{equation:6.49_2},
	\begin{align}\label{equation:cov_lwup-comp2}
		\mathbbm{E}\left[\bar{\psi}^{N,up,\epsilon^{*}}_u \bar{\psi}^{N,up,\epsilon^{*}}_v\right]&=\left(1-\frac{\epsilon^{*}}{\log N}\right)^2\mathbbm{E}\left[\bar{\psi}^N_u \bar{\psi}^N_v\right]\leq \left(1-\frac{\epsilon^{*}}{\log N}\right)^2 \mathbbm{E}\left[\psi^N_u \psi^N_v\right]+ \epsilon \left(1-\frac{\epsilon^{*}}{\log N}\right)^2.
	\end{align}
	We want that, for all $u,v\in V_N$, $\mathbbm{E}\left[\psi^{N,lw,\epsilon^{*}}_u \psi^{N,lw,\epsilon^{*}}_v\right]\geq \mathbbm{E}\left[\bar{\psi}^{N,up,\epsilon^{*}}_u \bar{\psi}^{N,up,\epsilon^{*}}_v\right].$
	Considering \eqref{equation:cov_lwup-comp1} and \eqref{equation:cov_lwup-comp2}, this holds, provided
	\begin{align}\label{equation:6.55}
		\epsilon^{N,\prime}_u\epsilon^{N,\prime}_v\geq \epsilon \left(1-\frac{\epsilon^{*}}{\log N}\right)^2.
	\end{align}
	Combining \eqref{equation:6.55} with \eqref{equation:eps_prime} and as $\epsilon \rightarrow 0$, one sees that it is possible to choose first $\epsilon^{*}(\epsilon)$ and then both $\{\epsilon^{N,\prime}_v(\epsilon,\epsilon^{*})\}_{v\in V_N}$ and $ \{\epsilon^{N,\prime\prime}_v(\epsilon,\epsilon^{*})\}_{v\in V_N}$, such that $\epsilon^{*}\rightarrow 0$, and that at the same time, all requirements \eqref{equation:6.4_1}, \eqref{equation:6.49_2} and \eqref{equation:6.55} hold. Observe further, that in this case, by \eqref{equation:6.4_1} and \eqref{equation:6.49_2}, $\max_{v\in V_N}\epsilon^{N,\prime}_v\rightarrow 0$, as well as $\max_{v\in V_N}\epsilon^{N, \prime \prime}_v\rightarrow 0$. With this choice, one can apply Slepian's lemma to obtain
	\begin{align}
		\tilde{d}\left(\max_{v\in V_N}\psi^{N,lw,\epsilon^{*}}_v-m_N,\max_{v\in V_N}\bar{\psi}^{N,up,\epsilon^{*}}_v-m_N\right)=0.
	\end{align}
	As $\epsilon\rightarrow 0,$ the distribution of the Gaussian field $\{\psi^{N,lw,\epsilon^{*}}_v\}_{v\in V_N}$ tends to that of $\{\psi^N_v\}_{v\in V_N}.$ Applying \autoref{lemma:adding_gaussiantailtype_rv} to $\{\bar{\psi}^{N,up,\epsilon^{*}}_v\}_{v\in V_N}$, one deduces
	\begin{align}
		\mathbbm{P}\left(\max_{v\in V_N}\bar{\psi}^{N,up,\epsilon^{*}}_v-m_N\geq x \right)&\leq \mathbbm{P}\left(\max_{v\in V_N}\bar{\psi}^N_v-m_N\geq x-\sqrt{\max_{w \in V_N}\epsilon^{N,\prime\prime}_w} \right) \left(Ce^{-(C\max_{w \in V_N}\epsilon^{N,\prime\prime}_w)^{-1}}\right).
	\end{align} 
	Since $\max\limits_{w \in V_N}\epsilon^{N,\prime\prime}_w\rightarrow 0$, as $\epsilon\rightarrow 0$, this allows to conclude the proof of \eqref{equation:4.23}. \eqref{equation:4.25} can be proved in the same way, by switching the roles of $\{\psi^N_v\}_{v\in V_N}$ and $\{\bar{\psi}^N_v\}_{v\in V_N}$ in the proof above. Further details are omitted.
\end{proof}
	\section{Proof of \autoref{proposition:sharp_right_tail_estimate}}\label{subsection:proof_right_tail}
	We outline the strategy of the proof: First, we localize the position of $S^{N,m}_v$, for particles $v\in V_N$ that satisfy $S^{N}_v\geq m_N+z$.
This reduces the computation of the asymptotic right-tail distribution to the computation of an expectation of a sum of indicators, which is significantly simpler, as it essentially boils down to computing a single probability.
In the second step, we prove that the asymptotic behaviour of the right-tail of the maximum of the auxiliary field does not depend on the parameter $N$, so that any possible constant also depends only on the remaining parameters, $K^\prime,L^\prime$ and $z$.
In the third step, we investigate how the limit scales in $z$, which allows us to factorize the dependence on the variable $z$ in the above obtained constants, reducing the dependence of the constants to the parameters, $K^\prime,L^\prime$. We further show that the constants can be bounded uniformly from below and from above, which then concludes the proof.
	Recall that $S^{N,f}_v=S^{N}_v-S^{N,c}_v$, for $v\in V_N$. 
	For the entire proof, fix the index $i$ along with a box $B_{N/KL,i}$. The field $\{S^{N,f}_v\}_{v\in B_{N/KL,i}}$ is constructed in such a way (see \eqref{equation:def_3-field_approximand}), that it is independent of the integers $K,L$ and $i$. In particular, the sequence $\{\beta_{K^{\prime},L^{\prime}}^{*}\}_{K^{\prime} L^{\prime}}$ does not depend on these. 
For a fixed $v\in B_{N/KL,i}$, and for $S^{N,m}_{v}$, consider $X^N_v$ as the associated variable speed Brownian motion. To be more precise, recall the definition of $S^{N,m}_v$ in \eqref{equation:S^N,m_v}. To each Gaussian random variable $b^{N}_{i,j,B}$ in \eqref{equation:S^N,m_v}, associate an independent Brownian motion $b^N_{i,j,B}(t)$ that runs for $2^{-2j}$ time with rate $\sigma\left(\frac{n-j}{n}\right)$ and ends at the value of $\sigma\left(\frac{n-j}{n}\right)b^N_{i,j,B}$. Each variable speed Brownian motion, $\{X^N_v(t)\}_{0\leq t\leq n-k-l-k^{\prime} -l^{\prime}}$, is defined by concatenating the Brownian motions associated to earlier times, which correspond to larger scales. Until the end of the proof, in order to shorten notation, simply write $\bar{N}=N/KL$, $n^{*}=n-k-l-k^{\prime}-l^{\prime}$ and analogously, $\bar{n}=n-k-l$ as well as $\bar{l}=l^{\prime}+k^{\prime}$, $\bar{k}=k+l$. As in \eqref{equation:S^NB}, we consider the partitioning of $B_{N/KL,i}$ into a collection of $K^\prime L^\prime$-boxes $\mathcal{B}_{K^{\prime} L^{\prime}}$ and refer to $B_{K^{\prime} L^{\prime}}(v)\in \mathcal{B}_{K^{\prime} L^{\prime}}$ as the unique $K^{\prime} L^{\prime}-$box that contains $v$. The set of all left bottom corners of these $K^{\prime} L^{\prime}-$boxes is called $\Xi_{\bar{N}}$. We further write $M_n(k,t)=2\log(2)\mathcal{I}_{\sigma^2}\left(\frac{k}{n},\frac{t}{n}\right)n-\frac{((t)\wedge (n-\bar{l})) \log(n)}{4(n-\bar{l})}$, for $t\in [k,n]$.
Let
			\begin{align}\label{equation:E_v,N(z)}
				E_{v,N}(z)&=\left\{X^{N}_v(t)- M_n(\bar{k},t) \in [-i^\gamma(t,n^*),\max(i^\gamma(t,n^*),z)],\, \forall 0\leq t\leq n^{*},\right. \nonumber\\ &\quad\,\,
				\left. \max_{u \in B_{K^\prime L^\prime}(v)}Y^N_u\geq 2\log(2)\mathcal{I}_{\sigma_2}\left(\frac{\bar{k}}{n},1\right)n-\log(n)/4-\bar{k}^{\gamma}+z-X^{N}_v(n^*)  \right\},
			\end{align}
where $Y^N_u\overset{\text{law}}{\sim} S^N_u-S^{N,c}_u-S^{N,m}_u=S^{N,f}_u-S^{N,m}_u$ is an independent Gaussian field.
 The first restriction is that all particles have to stay within a tube around $2\log(2)\mathcal{I}_{\sigma^2}\left(\frac{\bar{k}}{n},\frac{\bar{k}+t}{n}\right)n$, which is due to \autoref{proposition:position_at_variance_change}. Moreover, it ensures that at the beginning, particles cannot be too large. The second event ensures that there are particles reaching the relevant level.
We consider the number of particles satisfying the event $E_{v,N}(z)$, namely
			\begin{align}
				\Lambda_N(z)\coloneqq \sum_{v\in \Xi_{\bar{N}} } \mathbbm{1}_{E_{v,N}(z)}
			\end{align}
and claim that
			\begin{align}\label{equation:6.87}
				\limsup\limits_{z\rightarrow \infty}\limsup\limits_{(L,K,L^\prime,K^\prime,N)\Rightarrow \infty} \left|\frac{\mathbbm{P}\left(\max\limits_{v\in B_{N/KL,i}} S^{N,f}_v\geq M_n(\bar{k},n)+z-k^{\gamma}\right)}{\mathbbm{E}\left[\Lambda_N(z)\right]} \right|=1.
			\end{align}
			This reduces the analysis to compute the asymptotics of the expectation, which is much simpler, as this only needs precise right-tail asymptotic of a single vertex. We start proving the claim \eqref{equation:6.87}. By a first moment bound and using \autoref{lemma:localization_MIBRW}, one obtains
			\begin{align}\label{equation:c.17}
				\limsup\limits_{z\rightarrow \infty}\limsup\limits_{(L,K,L^\prime,K^\prime,N)\Rightarrow \infty} \mathbbm{P}\left(\max_{v\in B_{N/KL,i}} S^{N,f}_v\geq M_n(\bar{k},n)+z-\bar{k}^{\gamma}\right)\leq \mathbbm{E}\left[\Lambda_N(z)\right],
			\end{align}
			which implies that the quotient is bounded from above by $1$.
			In order to obtain equality, one shows
			\begin{align}\label{equation:6.89}
				\limsup\limits_{z\rightarrow \infty}\limsup\limits_{(L,K,L^\prime,K^\prime,N)\Rightarrow \infty}\mathbbm{E}\left[\Lambda_N(z)^2\right]/\mathbbm{E}\left[\Lambda_N(z)\right]=1.
			\end{align}
			Assuming \eqref{equation:6.89} and using the Cauchy-Schwarz inequality, one has
			\begin{align}
				\mathbbm{P}\left(\max_{v\in B_{N/KL,i}} S^{N,f}_v\geq M_n(\bar{k},n)+z\right)\geq \mathbbm{E}\left[\Lambda_N(z)\right],
			\end{align}
			which, together with \eqref{equation:c.17}, then implies \eqref{equation:6.87}.
			Thus, we turn to the proof of equation \eqref{equation:6.89}. First, decompose the second moment along the branching scale, $b_N(v,w)=\max \{\lambda\geq 0: \, [v]_\lambda \cap [w]_\lambda \neq \emptyset \}$, beyond which increments are independent, i.e.
			\begin{align}\label{equation:c.20}
				\mathbbm{E}\left[\Lambda_N(z)^2\right]&= \mathbbm{E}\left[\Lambda_N(z)\right] + \sum_{v,w\in \Xi_{\bar{N}} } \mathbbm{P}\left(E_{v,N}(z)\cap E_{w,N}(z)\right)\nonumber\\
				&= \mathbbm{E}\left[\Lambda_N(z)\right]+\sum_{t_s=0}^{n^{*}-1}\sum_{v,w:d(v,w)=t_s}\mathbbm{P}\left(E_{v,N}(z)\cap E_{w,N}(z)\right)
			\end{align}
			Note that, for $v\in \Xi_{\bar{N}}$ fixed, there are $2^{2(n^*-(\bar{k}+t_s))}$ many $w\in \Xi_{\bar{N}}$ with $d(v,w)=t_s$.
			The probabilities in \eqref{equation:c.20} can be bounded from above by
			\begin{align}\label{equation:7.115}
				\mathbbm{P}\left(E_{v,N}(z)\cap E_{w,N}(z)\right)& \leq\sum_{\substack{x_s\in [-i^\gamma(\bar{k}+t_s,n^*),\max(i^\gamma(\bar{k}+t_s,n^*),z)]\\x_1,x_2 \in [-\bar{l}^{\gamma},\bar{l}^{\gamma}]}} \mathbbm{P}\left(X^N_v(t_s)-M_n(\bar{k},t_s)\in  [x_s-1,x_s]\right)\nonumber
				\\
				&\quad\times \mathbbm{P}\left(X^N_v(n^{*})-X^N_v(t_s)-M_n(t_s,n-\bar{l})+\bar{k}^{\gamma}+x_s\in [x_1-1,x_1]\right)\nonumber\\ 
				&\quad\times
				\mathbbm{P}\left(\max_{u\in B_{K^\prime L^\prime}(v)} Y^N_u\geq 2\log(2)\sigma^2(1) \bar{l}+z-x_1 \right)\nonumber\\ 
				&\quad\times
				\mathbbm{P}\left(X^N_w(n^{*})-X^N_w(t_s)-M_n(t_s,n-\bar{l})+\bar{k}^{\gamma}+x_s\in[x_2-1,x_2]\right)\nonumber\\
				&\quad\times
				\mathbbm{P}\left(\max_{u\in B_{K^\prime L^\prime}(w)} Y^N_u\geq 2\log(2)\sigma^2(1) \bar{l}+z-x_2 \right)
			\end{align}
			Similarly, one can expand $\mathbb{E}\left[\Lambda_N(z)\right]$, i.e.
			\begin{align}\label{equation:7.116}
				\mathbb{E}\left[\Lambda_N(z)\right]&=2^{2n^*}\sum_{\substack{x_s\in [-i^\gamma(\bar{k}+t_s,n^*),\max(i^\gamma(\bar{k}+t_s,n^*),z)]\\x_1,x_2 \in [-\bar{l}^{\gamma},\bar{l}^{\gamma}]}} \mathbbm{P}\left(X^N_v(t_s)-M_n(\bar{k},t_s)\in  [x_s-1,x_s]\right)\nonumber
				\\&\qquad
				\times \mathbbm{P}\left(X^N_v(n^{*})-X^N_v(t_s)-M_n(\bar{k}+t_s,n-\bar{l})+\bar{k}^{\gamma}+x_s\in [x_1-1,x_1]\right)\nonumber\\ &\qquad\times
				\mathbbm{P}\left(\max_{u\in B_{K^\prime L^\prime}(v)} Y^N_u\geq 2\log(2)\sigma^2(1) \bar{l}+z-x_1 \right).
			\end{align}
			For each summand, there is an additional factor appearing in \eqref{equation:7.115} compared to \eqref{equation:7.116}. If one can show that all these vanish uniformly over $x_s$, when summing over $t_s$ and then taking the limits, $(z,\bar{L},N)\Rightarrow \infty$, one obtains \eqref{equation:6.89}, and thereby \eqref{equation:6.87}.
			Thus, one needs to estimate the additional factors,
			\begin{align}\label{equation:7.118}
				&\sum_{x_2\in [-\bar{l}^{\gamma},\bar{l}^{\gamma}]} \mathbb{P}\left(X^N_w(n^{*})-X^N_w(t_s)-M_n(\bar{k}+t_s,n-\bar{l})+\bar{k}^{\gamma}+x_s\in[x_2-1,x_2]\right)\nonumber\\&\qquad
				\times \mathbb{P}\left(\max_{u\in B_{K^\prime L^\prime}(w)} Y^N_u\geq 2\log(2)\sigma^2(1) \bar{l}+z-x_2\right)\nonumber\\&\quad
				\leq 2^{-2(n^{*}-(\bar{k}+t_s))} \sum_{x_2 \in [- \bar{l}^{\gamma}, \bar{l}^\gamma]} \frac{2\log(2)\bar{l}\sigma(1)+\frac{z-x_2}{\sigma(1)}}{\sqrt{2\pi \log(2)\mathcal{I}_{\sigma^2}\left(\frac{\bar{k}+t_s}{n},\frac{n-\bar{l}}{n}\right)n}\sqrt{\bar{l}\log 2}} \exp\left[-2\log(2)(\bar{k}+t_s-\mathcal{I}_{\sigma^2}\left(\frac{\bar{k}+t_s}{n}\right)n)\right]\nonumber\\&\qquad
				\times\exp\left[-2\log(2)\bar{l}-2\left(z-x_s-\frac{n-\bar{k}-\bar{l}-t_s}{4(n-\bar{k}-\bar{l})}\log(n)-\bar{k}^{\gamma}\right)-\frac{\left(\frac{z-x_2}{\sigma(1)}\right)^2}{2\log(2)\bar{l}}\right]\nonumber\\&\qquad
				\times \exp\left[-\frac{\left(x_2-x_s-\frac{n-\bar{k}-\bar{l}-t_s}{4(n-\bar{k}-\bar{l})}\log(n)-\bar{k}^{\gamma}\right)^2}{2\log(2)\mathcal{I}_{\sigma^2}\left(\frac{\bar{k}+t_s}{n},\frac{n-\bar{l}}{n}\right)n}\right].
			\end{align}
			Note that there are $2^{2(n^*-(\bar{k}+t_s))}$ vertices $w\in \Xi_{\bar{N}}$ with $d(v,w)=t_s$, for fixed $v\in\Xi_{\bar{N}}$ , which cancels with the prefactor in \eqref{equation:7.118} when taking the sum in \eqref{equation:c.20}.
			To show that the sum in $t_s$ is finite, first note that the relevant term in \eqref{equation:7.118} is given by $\exp\left[-2\log(2)(\bar{k}+t_s-\mathcal{I}_{\sigma^2}\left(\frac{\bar{k}+t_s}{n}\right)n)\right].$
			Recall the assumption $\mathcal{I}_{\sigma^2}(x)<x$, for $x\in (0,1)$. In particular, for any $\delta>0$, there exists $\epsilon>0$ such that $\mathcal{I}_{\sigma^2}(x)<x-\epsilon$, for $x\in (\delta,1-\delta)$. Since one is interested in the limit, as $(z,K^\prime, L^\prime,N)\Rightarrow \infty$, it is possible to assume $\frac{\bar{k}(1-\sigma^2(0))}{n}<\epsilon/2$ and $\frac{\bar{l}(\sigma^2(1)-1)}{n}<\epsilon/2$. In this case it holds, for $t_s \in (0,n-\bar{k}-\bar{l})$,
			\begin{align}\label{equation:c.25}
				\mathcal{I}_{\sigma^2}\left(\frac{\bar{k}+t_s}{n}\right)< \frac{\bar{k}+t_s}{n} -\epsilon/2.
			\end{align}
			Using \eqref{equation:c.25} in \eqref{equation:7.118}, implies that \eqref{equation:7.118} is summable in $t_s\in (0,n-\bar{k}-\bar{l})$, when considering limits $(z,K^\prime, L^\prime,N)\Rightarrow \infty$. The sum in $x_2$ in \eqref{equation:7.118} is bounded by its number of summands, i.e. one gets a prefactor of leading order $4\log(2)\bar{l}^{\gamma+1/2}\sigma(1)$, where one can choose $\gamma\in (\frac{1}{2},1)$. Note that there is still the term $\exp\left[-2\log(2)\bar{l}\right]$ which ensures that \eqref{equation:7.118} tends to zero, as $(z,K^\prime, L^\prime,N)\Rightarrow \infty$. Altogether, this proves \eqref{equation:6.89}.
In the second step, we show that it is possible to choose the sequence of constants independently of $N$. More explicitly, in the following, we show that there are constants $\beta_{K^\prime,L^\prime,z}>0$, such that
			\begin{align}\label{equation:scale_step3}
				\lim\limits_{z\rightarrow \infty}\limsup\limits_{(L^\prime,K^\prime,N)\Rightarrow \infty}\frac{\mathbbm{E}\left[\Lambda_N(z)\right]}{\beta_{K^\prime,L^\prime,z}}=\lim\limits_{z\rightarrow \infty}\liminf\limits_{(L^\prime,K^\prime,N)\Rightarrow \infty}\frac{\mathbbm{E}\left[\Lambda_N(z)\right]}{\beta_{K^\prime,L^\prime,z}}
				=e^{2\log(2)\bar{k}(\sigma^2(0)-1)}e^{2\bar{k}^\gamma}.
			\end{align}
Since $X^N_v(n^{*})\sim \mathcal{N} \left( 0,\log(2)\mathcal{I}_{\sigma^2}\left(\frac{\bar{k}}{n},\frac{n^{*}}{n}\right)n \right)$, and using \autoref{lemma:localization_MIBRW}, which allows to ignore the restriction to stay below the maximum at all times, $\mathbbm{E}\left[\Lambda_N(z)\right]$ reads
			\begin{align}\label{equation:6.112}
				&2^{2(n-\bar{k}-\bar{l})}\mathbbm{P}\left(X^N_v(\bar{n})-M_n(\bar{k},n-\bar{l})\in [-\bar{l}^\gamma,\bar{l}^\gamma],\max_{u \in B_{K^\prime L^\prime}(v) }Y^N_u \geq M_n(\bar{k},n)-X^N_v(n^{*})-\bar{k}^{\gamma}+z\right)\nonumber\\
				&\quad=\int\limits_{-\bar{l}^\gamma}^{\bar{l}^\gamma}\frac{2^{2(n-\bar{k}-\bar{l})}}{\sqrt{2\pi \log(2) \mathcal{I}_{\sigma^2}\left(\frac{\bar{k}}{n},\frac{n-\bar{l}}{n}\right)n}}\exp\left[-\frac{\left(M_n(\bar{k},n-\bar{l})+x\right)^2}{2\log(2)\mathcal{I}_{\sigma^2}\left(\frac{\bar{k}}{n},\frac{n-\bar{l}}{n}\right)n}\right] \nonumber\\&\qquad \times
				\mathbbm{P}\left(\max_{u \in B_{K^\prime L^\prime}(v) }Y^N_u\geq 2\log(2)\bar{l}\sigma^2(1)+z-\bar{k}^{\gamma}-x\right)\mathrm{d}x\nonumber\\&\quad
				=\int\limits_{-\bar{l}^\gamma}^{\bar{l}^\gamma}\frac{2^{2\bar{k}(\sigma^2(0)-1)}2^{2\bar{l}(\sigma^2(1)-1)}\sqrt{n}}{\sqrt{2\pi \log(2) \mathcal{I}_{\sigma^2}\left(\frac{\bar{k}}{n}\frac{n-\bar{l}}{n}\right)n}} \exp\left[-2x-\frac{\left(x-\frac{\log(n)}{4}\right)^2}{2\log(2)(n-\sigma^2(0)\bar{k}-\sigma^2(1)\bar{l})}\right] \nonumber\\ &\qquad
				\times \mathbbm{P}\left(\max_{u \in B_{K^\prime L^\prime}(v) }Y^N_u\geq 2\log(2)\bar{l}\sigma^2(1)+z-\bar{k}^{\gamma}-x\right)\mathrm{d}x.
			\end{align}
			By definition of $S^N_u$ (see \eqref{equation:def_3-field_approximand}), $\max_{u \in B_{K^\prime L^\prime}(v)} Y^N_u$ has the same law as $\max_{u\in V_{K^\prime L^\prime}}S^{N,b}_u+a_{K^\prime L^\prime,\bar{u}}\Phi_j$ and is therefore independent of $N$ (cp. \eqref{equation:S^NB} and \eqref{equation:def_3-field_approximand}). Note further that $\frac{\sqrt{n}}{\sqrt{\mathcal{I}_{\sigma^2}\left(\frac{\bar{k}}{n},\frac{n-\bar{l}}{n}\right)n}}\overset{n\rightarrow \infty}{\rightarrow} 1$, and by Borell's inequality for Gaussian processes (see \cite[Lemma~3.1]{MR2814399}),
			\begin{align}\label{equation:c.15}
				\mathbb{P}\left(\left|\max_{u\in B_{K^\prime L^\prime}(v)} Y^N_u\geq 2\log (2) \bar{l}\sigma^2(1) +z-x-\bar{k}^\gamma\right|\right)\leq C 2^{-2\bar{l}(\sigma(1)-1)^2} \bar{l}^{-\frac{3}{2}(\sigma(1)-1)} \exp\left[-2\frac{\sigma(1)-1}{\sigma(1)}(z-\bar{k}^\gamma)\right].
			\end{align} 
			As $\sigma(1)>1$, \eqref{equation:c.15}, together with \eqref{equation:6.112}, implies \eqref{equation:scale_step3} and thus, the third claim. In particular, one can read off \eqref{equation:6.112} that the sequence $\{\beta_{K^\prime,L^\prime,z}\}$ depends only on the very last variance parameter and on $\bar{k}^\gamma$.
			In the last step, we analyse how the right tail probability scales in $z$, namely we want to show
			\begin{align}\label{equation:c..15}
			&\lim\limits_{z_1,z_2\rightarrow \infty}\limsup\limits_{(\bar{L},N)\Rightarrow \infty}\frac{e^{-2z_2}\mathbbm{E}\left[\Lambda_N(z_1)\right]}{e^{-2z_1}\mathbbm{E}\left[\Lambda_N(z_2)\right]}=\lim\limits_{z_1,z_2\rightarrow \infty}\liminf\limits_{(\bar{L},N)\Rightarrow \infty}\frac{e^{-2z_2}\mathbbm{E}\left[\Lambda_N(z_1)\right]}{e^{-2z_1}\mathbbm{E}\left[\Lambda_N(z_2)\right]}=1.
			\end{align}
			For $v\in V_N$, set $\nu_{v,N}(\cdot)$ be the density, such that for any interval $I\subset \mathbb{R}$,
			\begin{align}
				\int_I \nu_{v,N}(y)\mathrm{d}y=\mathbb{P}\left(X^N_v(n^{*})\in I + M_n(\bar{k},n-\bar{l}) \right).
			\end{align}
			Using this notation, we can rewrite
			\begin{align}\label{equation:c..17}
				\mathbb{P}\left(E_{v,N}(z)\right)=\int\limits_{-\bar{l}^\gamma}^{\bar{l}^\gamma}\nu_{v,N}(z+x)\mathbbm{P}\left(\max\limits_{u \in B_{K^\prime L^\prime}(v)}Y^N_u\geq 2\log(2)\bar{l}\sigma^2(1)-\bar{k}^{\gamma}-x\right)\mathrm{d}x.
			\end{align}
			Note that in \eqref{equation:c..17} only $\nu_{v,N}(z+x)$ depends on $z$.
			For $z_1,z_2>0$, one has to compute the quotient $\mathbb{E}\left[\Lambda_N(z_1)\right]/\mathbb{E}\left[\Lambda_N(z_2)\right]$, for which we use the reformulation in \eqref{equation:c..17}.
			The strategy is to compute the asymptotic limit of the integral involving $z_1$ in terms of the integral involving $z_2$ and an additional correction factor.
			As $\bar{l}\rightarrow \infty$, prior to $z_1,z_2\rightarrow \infty$, there is no need to shift the limits of the integrals.
			For the remaining factors in both integrals, one obtains the relative density with respect to $z_1,z_2$, i.e.
			\begin{align}
				\frac{\nu_{v,N}(z_1+x)}{\nu_{v,N}(z_2+x)}
				&=\exp\left[-2(z_1-z_2)-\frac{z_1^2-z_2^2-(z_1-z_2)\frac{\log(n)}{2}}{2\log(2)\mathcal{I}_{\sigma^2}\left(\frac{\bar{k}}{n},\frac{n-\bar{l}}{n}\right)n}-x\frac{(z_1-z_2)}{\log(2)\mathcal{I}_{\sigma^2}\left(\frac{\bar{k}}{n},\frac{n-\bar{l}}{n}\right)n}\right].
			\end{align}
			Thus, we can rewrite $\mathbb{P}\left(E_{v,N}(z_1)\right)$ as
			\begin{align}\label{equation:c.22}
				&\int_{-\bar{l}^\gamma}^{\bar{l}^\gamma}\nu_{v,N}(z_2+x)e^{2(z_1-z_2)}
				\mathbbm{P}\left(\max_{u \in B_{K^\prime L^\prime}(v) }Y^N_u\geq 2\log(2)\bar{l}\sigma^2(1)-\bar{k}^{\gamma}-x\right)\nonumber\\ &\qquad \quad \times \exp\left[\frac{z_1^2-z_2^2-(z_1+z_2)\frac{\log(n)}{2}}{2\log(2)\mathcal{I}_{\sigma^2}\left(\frac{\bar{k}}{n},\frac{n-\bar{l}}{n}\right)n}+x\frac{(z_1-z_2)}{\log(2)\mathcal{I}_{\sigma^2}\left(\frac{\bar{k}}{n},\frac{n-\bar{l}}{n}\right)n}\right] \mathrm{d}x,
			\end{align}
			where the last factor tends to $1$, as $(\bar{L},N)\Rightarrow \infty$.
			Computing the quotient $\mathbb{E}\left[\Lambda_N(z_1)\right]/\mathbb{E}\left[\Lambda_N(z_2)\right]$ using \eqref{equation:c.22} and summing over all vertices, one obtains, when turning to limits, that \eqref{equation:c..15} holds.
	Combining the above steps, in particular \eqref{equation:c..15} with \eqref{equation:scale_step3}, completes the proof of \eqref{equation:prop_sharp_right_tail}, with some non-negative sequence $\{\beta_{K^\prime,L^\prime}\}_{K^\prime,L^\prime\geq 0}$. In the final step, we show that this sequence is bounded.
	Using \autoref{lemma:localization:fine_field}, one has for some $\epsilon>0$, being at most of order $O\left(e^{-2\bar{k}^{2\gamma-1}/(2\sigma^2(0)\log 2) }\right)$,
	\begin{align}\label{equation:c.33}
		c_\alpha e^{-2z}\leq \int_{- \bar{k}^\gamma}^{\bar{k}^\gamma} \nu_{v,N}^{c}(x)2^{2\bar{k}}\mathbb{P}\left(\max_{v\in B_{N/KL,i}}S^{N,f}_v\geq 2\log 2 \mathcal{I}_{\sigma^2}\left(\frac{\bar{k}}{n}, 1\right)n- \frac{\log n}{4}+z-x\right) +\epsilon.
	\end{align}
	Using the asymptotics \eqref{equation:6.112} for the probability in the integral in \eqref{equation:c.33}, one can instead compute the integral
	\begin{align}\label{equation:c.34}
		\int_{-\bar{k}^\gamma}^{  \bar{k}^\gamma} \frac{\exp\left[- \frac{\left(2\log 2 \mathcal{I}_{\sigma^2}\left(\frac{\bar{k}}{n}\right)n  +x\right)^2}{2\log 2 \mathcal{I}_{\sigma^2}\left(\frac{\bar{k}}{n} \right)n}\right]}{\sqrt{2 \pi \log 2 \mathcal{I}_{\sigma^2}\left(\frac{\bar{k}}{n} \right)n}}2^{2\bar{k}} \beta_{K^{\prime},L^{\prime}}e^{-2z+2x+2\log 2 \bar{k}(\sigma^2(0)-1)}\mathrm{d}x = \beta_{K^{\prime},L^{\prime}} e^{-2z}
		\int_{- \bar{k}^\gamma}^{\bar{k}^\gamma}\frac{\exp\left[-\frac{x^2}{2\log 2 \mathcal{I}_{\sigma^2}\left(\frac{\bar{k}}{n}\right)n}\right]}{\sqrt{2\pi  \log 2 \mathcal{I}_{\sigma^2}\left(\frac{\bar{k}}{n}\right)n }}\mathrm{d}x.
	\end{align}
	The integral in \eqref{equation:c.34} is bounded by $1$ and thus, when considering the lower bound in \eqref{equation:c.33}, one can deduce that $c_\alpha\leq \beta_{K^{\prime},L^{\prime}}$, for $K^\prime, L^\prime \geq 0$. The upper bound, i.e. $\beta_{K^{\prime},L^{\prime}}\leq C_\alpha$, for $K^\prime, L^\prime \geq 0$ and for some constant $C_\alpha>0$, follows from a union and a Gaussian tail bound, i.e.
	\begin{align}
		\mathbb{P}&\left(\max_{v\in B_{N/KL,i}}S^{N,f}_v\geq 2\log 2 \mathcal{I}_{\sigma^2}\left(\frac{\bar{k}}{n}, 1\right)n- \frac{\log n}{4}+z-\bar{k}^\gamma\right)\nonumber\\
		&\qquad\leq C_\alpha \frac{2^{2(n-\bar{k})}}{\sqrt{n}}\exp\left[-2\log 2 \mathcal{I}_{\sigma^2}\left(\frac{\bar{k}}{n},1\right)n-2\left(z-\bar{k}^\gamma+\frac{\log n}{4}\right)-\frac{\left(z-\bar{k}^\gamma-\frac{\log}{4}\right)^2}{2\log 2 \mathcal{I}_{\sigma^2}\left(\frac{\bar{k}}{n},1\right)n}\right]\nonumber\\
		&\qquad\leq C_\alpha \exp\left[2\log(2)\bar{k}(\sigma^2(0)-1)+2\bar{k}^\gamma-2z\right],
	\end{align} This concludes the proof of \autoref{proposition:sharp_right_tail_estimate}. \qed
	\nocite{*}
	
	\bibliography{literature.bib}
	\bibliographystyle{abbrv}
\end{document}